\documentclass[reqno, 11pt, twoside]{amsart}

\makeatletter
\def\specialsection{\@startsection{section}{1}%
  \z@{\linespacing\@plus\linespacing}{.5\linespacing}%
  {\normalfont}}
\def\section{\@startsection{section}{1}%
  \z@{.7\linespacing\@plus\linespacing}{.5\linespacing}%
  {\normalfont\scshape}}
\makeatother

\usepackage{etoolbox}
\patchcmd{\section}{\scshape}{\bfseries}{}{}
\makeatletter
\renewcommand{\@secnumfont}{}
\makeatother

\makeatletter
\def\subsection{\@startsection{subsection}{1}%
  \z@{.5\linespacing\@plus.7\linespacing}{-.5em}%
  {\normalfont\itshape}}
	\def\@sect#1#2#3#4#5#6[#7]#8{%
  	\edef\@toclevel{\ifnum#2=\@m 0\else\number#2\fi}%
  	\ifnum #2>\c@secnumdepth \let\@secnumber\@empty
  	\else \@xp\let\@xp\@secnumber\csname the#1\endcsname\fi
  	\@tempskipa #5\relax
  	\ifnum #2>\c@secnumdepth
  	  \let\@svsec\@empty
  	\else
  	  \refstepcounter{#1}%
    \edef\@secnumpunct{%
      \ifdim\@tempskipa>\z@ 
        \@ifnotempty{#8}{.\@nx\enspace}%
      \else
        \@ifempty{#8}{.}{.\@nx\enspace}%
      \fi
    }%
    \@ifempty{#8}{%
      \ifnum #2=\tw@ \def\@secnumfont{\bfseries}\fi}{}%
    \protected@edef\@svsec{%
      \ifnum#2<\@m
        \@ifundefined{#1name}{}{%
          \ignorespaces\csname #1name\endcsname\space
        }%
      \fi
      \@seccntformat{#1}%
    }%
  \fi
  \ifdim \@tempskipa>\z@ 
    \begingroup #6\relax
    \@hangfrom{\hskip #3\relax\@svsec}{\interlinepenalty\@M #8\par}%
    \endgroup
    \ifnum#2>\@m \else \@tocwrite{#1}{#8}\fi
  \else
  \def\@svsechd{#6\hskip #3\@svsec
    \@ifnotempty{#8}{\ignorespaces#8\unskip
       }%
    \ifnum#2>\@m \else \@tocwrite{#1}{#8}\fi
  }%
  \fi
  \global\@nobreaktrue
  \@xsect{#5}}
\makeatother

\makeatletter
\def\abstract#1{ \gdef\@abstract{#1}}
\def\@abstract{\@latex@error{No \noexpand\abstract given}\@ehc}
\makeatother

\makeatletter
\def\keywords#1{ \gdef\@keywords{#1}}
\def\@keywords{\@latex@error{No \noexpand\keywords given}\@ehc}
\makeatother

\makeatletter
\def\MSC#1{ \gdef\@MSC{#1}}
\def\@MSC{\@latex@error{No \noexpand\MSC given}\@ehc}
\makeatother

\makeatletter
\def\abstract#1{ \gdef\@abstract{#1}}
\def\@abstract{\@latex@error{No \noexpand\abstract given}\@ehc}
\makeatother

\usepackage{verbatim}
\usepackage[letterspace=150]{microtype}
\usepackage{multicol}
\usepackage[affil-it]{authblk}
\usepackage{titling}
\usepackage[latin1]{inputenc}
\usepackage{calligra}
\usepackage[OT1]{fontenc}
\usepackage[english]{babel}
\usepackage{amsfonts}
\usepackage{amsmath}
\usepackage{amssymb}
\usepackage{amsthm}
\usepackage{faktor}
\usepackage{float}
\usepackage{enumerate}
\usepackage{color}
\usepackage{pdfpages}
\usepackage{esint}
\usepackage{hyperref}
\usepackage{cite}
\usepackage{fancyhdr}
\usepackage{mathtools}
\usepackage{mathrsfs}
\usepackage{abraces}
\usepackage{MnSymbol}

\newcommand{\Div}{\mathrm{div}}
\newcommand{\curl}{\mathrm{curl}\,}

\newcommand{\ee}{\varepsilon}

\newcommand{\Cc}{\mathcal{C}}

\newcommand{\dd}{\mathrm{d}}

\newcommand{\pre}{ \mathrm{p}}
\newcommand{\uu}{ \mathsf{u}}

\newcommand{\s}{ \mathrm{s}}
\newcommand{\e}{ \mathrm{e}}
\newcommand{\Tt}{\vartheta}

\newcommand{\I}{\mathcal{I}}

\newcommand{\EE}{\mathcal{E}}

\newcommand{\FF}{\mathcal{F}}

\newcommand{\DD}{\mathfrak{D}}
\newcommand{\RR}{\mathbb{R}}
\newcommand{\ZZ}{\mathbb{Z}}
\newcommand{\R}{\mathbb{R}}

\newcommand{\BB}{\dot{B}}
\newcommand{\Hh}{\dot{H}}
\newcommand{\Dd}{\dot{\Delta}}
\newcommand{\Sd}{\dot{S}}
\newcommand{\Rd}{\dot{R}}
\newcommand{\Th}{\dot{T}}

\newcommand{\hra}{\hookrightarrow}
\newcommand{\rhu}{\rightharpoonup}
\newcommand{\loc}{\textnormal{loc}}
\newcommand{\NS}{Navier-Stokes }
\newcommand{\dx}{\textnormal{d}{x}}
\renewcommand{\d}{\textnormal{d}}
\renewcommand{\div}{\textnormal{div}}

\newcommand{\pare}[1]{\left( #1 \right)}
\newcommand{\angles}[1]{\left\langle #1 \right\rangle}

\newcommand{\norm}[1]{\left\| #1 \right\|}
\newcommand{\av}[1]{\left| #1 \right|}
\newcommand{\bra}[1]{\left[ #1 \right]}
\newcommand{\set}[1]{\left\{ #1 \right\}}
\newcommand{\psc}[2]{\left\langle \left. #1 \right| #2\right\rangle}

\newcommand{\cP}{\mathcal{P}}

\newcommand{\Ft}[1]{\mathcal{F}_{x, t} #1}%

\newcommand{\cC}{\mathcal{C}}

\newcommand{\cQ}{\mathcal{Q}}
\newcommand{\cF}{\mathcal{F}}
\newcommand{\cJ}{\mathcal{J}}
\newcommand{\cD}{\mathcal{D}}

\newcommand{\bR}{\mathbb{R}}
\newcommand{\bN}{\mathbb{N}}
\newcommand{\cG}{\mathcal{G}}

\newtheorem{theorem}{Theorem}[section]

\newtheorem{prop}[theorem]{Proposition}
\newtheorem{lemma}[theorem]{Lemma}
\newtheorem{definition}[theorem]{Definition}

\DeclareFontFamily{OT1}{rsfs}{}
\DeclareFontShape{OT1}{rsfs}{m}{n}{ <-7> rsfs5 <7-10> rsfs7 <10-> rsfs10}{}
\DeclareMathAlphabet{\mycal}{OT1}{rsfs}{m}{n}

\DeclareSymbolFont{letters}{OML}{ztmcm}{m}{it}

\begin{document}
 \begin{center}%
  \noindent\rule{155mm}{0.01cm}
  \let \footnote \thanks
  {\LARGE  \textsc{A global well-posedness result for the Rosensweig system of ferrofluids} \par}%
  \noindent\rule{155mm}{0.01cm}
  \vskip 1.em
  \today	
 \end{center}
 \vskip 1.em%
 {\let \footnote \thanks  Francesco De Anna$\,^1$,$\hspace{0.2cm}$Stefano Scrobogna$\,^2$ \par}
 {\textit{\footnotesize{$\,^1$ Department of Mathematics, Pennsylvania State University, University Park, PA 16802, US\\
			e-mail: fzd16@psu.edu\vspace{0.1cm} \\			
			$\,^2\,$ Basque Center for Applied Mathematics, Mazarredo 14, 48009, Bilbao, Spain\\
			e-mail: sscrobogna@bcamath.org } } }
 \par
 \vskip 1.em
 \noindent\rule{155mm}{0.01cm} 
 \vskip 1.em
 {\small
  \begin{minipage}{0.3\linewidth}
  \noindent\hspace{-0.6cm} \textsc{article info}
  
  \noindent\hspace{-0.4cm}\rule{4.3cm}{0.01cm} 
  
  \end{minipage}
  \begin{minipage}{0.6\linewidth}
   \noindent \noindent \textsc{ abstract}
      
   \noindent\rule{10.32cm}{0.01cm}

  \end{minipage}
	
  \smallskip
  \hspace{-0.4cm}
  \begin{minipage}{0.29\linewidth}
  
  \noindent \hspace{-0.0cm}\textit{Keywords: }
  Ferrofluids, fractional time derivative,  global well-posedness, long time behavior, decay estimates, regularity propagation.
  
  \vspace{0.4cm}
  \noindent \hspace{-0.0cm}\textit{MSC:} 	35Q30,	76A05, 	76W99. 
  
  \end{minipage}
  \hspace{0.34cm}
  \begin{minipage}{0.6\linewidth}
  
  \noindent $$\,$$

   In this Paper we study a Bloch-Torrey regularization of the Rosensweig system for ferrofluids. The scope of this paper is twofold. First of all, we investigate 
   the existence and uniqueness of solutions \`a la Leray of this model in the whole space $\RR^2$. Interesting enough, the well-posedness relies on a variation of the Aubin-Lions lemma for fractional time derivatives. 
   
   In the second part of this paper we investigate both the long-time behavior of weak solutions and the propagation of Sobolev regularities in dimension two
   
   $\,$
  \end{minipage}
  
  \vspace{0.1cm}
  \noindent\rule{155mm}{0.01cm} 
  } 	
\makeatother
\allowdisplaybreaks{}

\tableofcontents
\section{Introduction}

\noindent
The purpose of the present work is the mathematical study of a model describing the hydrodynamics of ferrofluids, under the action of an applied magnetic field. Ferrofluids are complex liquids presenting a strong magnetization under the action of a magnetic field.
The motion of these fluids is complicated by the existence of mesoscale or sub-domain structures: a carrier fluids (water, oil or other organic solvent) surround a composed of nanoscale particles, such as of magnetite and hermatite, or some other compounds of iron.

\smallskip\noindent
Ferrofluids are purely artificial, and they were invented by the NASA program in 1963 \cite{step}. The drive to create such a magnetic liquid was to control and direct rocket fuel in outer space.
Indeed, the absence of gravity leaded the fuel to float in the holding tank and it was therefore a challenge to pump the fuel efficiently into the rocket engine. The pioneristic work of Papell allowed to convert the nonmagnetic rocket fuel into a fuel having magnetic properties so that it could be controlled under zero gravity by powerful magnets.

\smallskip\noindent
Unlike solids and Newtonian liquids, the model equations for ferrofluids continue to evolve as new experimental evidence become available. New mathematical descriptions and self-consistent theories  resolve the ensemble of micro-elements, their intermolecular and distortional elastic interactions, the coupling between the hydrodynamics and the applied electric and magnetic fields. 
Thus, there exist several models describing the hydrodynamics of homogeneous micropolar fluids \cite{Nochetto, Shil}, and we investigate a regularization of Bloch-Torrey type \cite{Torrey} for the Rosensweig model \cite{Neuringer, Stefano1, Stefano2}. For the sake of unified presentation, it is convenient to introduce some terminology. The local magnetizing field (i.e. the dipole moment per unit volume) is described through a time-dependent function $M$  taking values from a domain of $\RR^3$ into the three dimensional vector space $\RR^3$ . The description of the ferrohydrodynamic interaction is driven by the equation governing the magnetic field $B$ and the demagnetizing field $H$, the Gauss law and the  Amp\'ere's law
\begin{equation*}
	\Div\,B(t,\,x)\,=\,F\quad\quad \text{and}\quad\quad \curl\, H(t,\,x)\,=\,0
\end{equation*}
with $H$ and $B$ vector functions of $(t,\,x)\in \RR_+\times\Omega$ and $F\,=\,F(t,\,x)$ the external magnetic field applied to the system. We consider here nonconduting ferrofluids for which the current density is null. Furthermore the relation connecting $B$ and $H$ is
\begin{equation*}
	B\,=\,\mu_0(\,M\,+\,H\,)
\end{equation*}
where $\mu_0$ is the permeability of vacuum. For monodispersion of spherical particles, the spin can be expressed in terms of the inertia $k>0$ per unit volume and the the average angular velocity $\Omega\,=\,\Omega(t,\,x)\in\RR^3$ of particles about their own center. 

\smallskip\noindent
The velocity $\uu\,=\,\uu(t,\,x)$ of centres of masses of particles obeys a forced incompressible Navier-Stokes system, with an additional stress tensor, a forcing term modelling the effect that the magnetic field
interacts on the dynamics of the centres of masses of the particles. The final system is then a combination of the Navier-Stokes equation, the magnetization equation and the magnetostatic equation.
Explicitly the equations, in non-dimensional form, are \cite{Amirat6, Rosenweig}:
\begin{equation}\label{main_system-init}
\begin{cases}
	\vspace{3pt}
	\;\rho_ 0\big(\partial_t \uu\,+\uu\cdot\nabla \uu\,\big)\,-\,(\eta +\zeta)\,\Delta \uu\,+\,\nabla \pre\,= \,\mu_0\, M\cdot \nabla H\,+2\zeta\, \curl \Omega,\\
	\vspace{3pt}
	\;\rho_0 k\big( \partial_t \Omega\,+\uu\cdot \nabla \Omega\,\big) - \eta'\Delta \Omega \,-\lambda' \nabla \Div\,\Omega\,=\,\mu_0 M\times H\,+\,2\zeta (\,\curl \uu \,-\,2\Omega\,), \\
	\vspace{3pt}
	\;\partial_t M \,+\,\uu\cdot \nabla M - \sigma \Delta M\,=\,\Omega \times M - \frac{1}{\tau}(\,M-\chi_0H\,),\\
	\vspace{3pt}
	\;\Div\,(\,M+H\,)\,=\,F,\\
	\vspace{3pt}
	\;\Div\,\uu\,=\,\curl\,H \,=\,0,	
\end{cases}
\end{equation}
supported by the initial condition:
\begin{equation*}
	(\uu,\,\Omega,\,M,\,H)_{|t=0}\,=\,(\uu_0,\,\Omega_0,\,M_0,\,H_0),
\end{equation*}
where $\uu_0$ is free-divergent, $H_0$ has null curl and the couple $(M_0,\,H_0)$ also satisfies the Amp\'ere's law
$
	\Div\,(\,H_0(x)\,+\,M_0(x)\,)\,=\,F(0,\,x).
$
The equations are defined in the whole three-dimensional space $\RR^3\times \RR_+$, and we assume that
\begin{equation*}
	F\,=\,F(x_1,\,x_2,\,t),
\end{equation*}
which corresponds to the case of a magnetic field independent of the vertical variable. Under this hypothesis we introduce the specific family of solutions of the form
\begin{equation*}
\begin{aligned}
	\uu \,&=\,(\uu(x_1,\,x_2,\,t),\,\uu_2(x_1,\,x_2,\,t),\,0), \\
	\Omega\,&=\,(\,0,\,0,\,\omega(\,x_1,\,x_2,\,t)\,),\\
	M\,&=\,(M_1(x_1,\,x_2,\,t),\,M_2(x_1,\,x_2,\,t),\,0).
\end{aligned}
\end{equation*}
Under these hypotheses, we recast the system \eqref{main_system-init} in the bidimensional form (cf. \cite{Stefano1})
\begin{equation}\label{main_system}
\begin{cases}
	\vspace{3pt}
	\;\rho_ 0\big(\partial_t \uu\,+\uu\cdot\nabla \uu\,\big)\,-\,(\eta +\zeta)\,\Delta \uu\,+\,\nabla \pre\,= \,\mu_0\, M\cdot \nabla H\,
												+\,2\zeta\, \left(\,\begin{matrix}\hspace{0.25cm}\partial_2  \omega\\-\partial_1 \omega	\end{matrix}\,\right),\\
	\vspace{3pt}
	\;\rho_0 k\big( \partial_t \omega\,+\uu\cdot \nabla \omega\,\big) - \eta'\Delta \omega\,=\,\mu_0\, M\times H\,+\,2\zeta (\,\curl \uu \,-\,2\omega\,), \\
	\vspace{3pt}
	\;\partial_t M \,+\,\uu\cdot \nabla M - \sigma \Delta M\,=\,\left(\,\begin{matrix}\hspace{0.25cm}M_2\\-M_1	\end{matrix}\,\right)\omega - \frac{1}{\tau}(\,M-\chi_0H\,),\\
	\vspace{3pt}
	\;\Div\,(\,M+H\,)\,=\,F,\\
	\vspace{3pt}
	\;\Div\,\uu\,=\,\curl\,H \,=\,0.	
\end{cases}
\end{equation}
Here, by an abuse of notation, we identify vector products and the curl operator by means of
\begin{equation*}
	M\times H\,=\,M_1\,H_2\,-\,H_1\,M_2\quad\text{and}\quad \curl\, \uu\,=\, \partial_1 \uu_2\,-\,\partial_2\,\uu_1,
\end{equation*}
respectively.

\noindent
Our analysis relies basically on energy estimates, so that we will look for weak solutions \`a la Leray, which are defined in a rather standard
manner:
\begin{definition}\label{def-weaksol}
	We say that $(\uu,\,\omega,\,M,\,H)$ is a weak solution of problem \eqref{main_system} if the conditions below are satisfied
	\begin{itemize}
		\item[(i)] 	The quad $(\uu,\,\omega,\,M,\,H)$ belongs to $L^\infty(0,T;L^2(\RR^2))\cap L^2(0,T;\,\Hh^1(\RR^2))$;
		\item[(ii)] The momentum equation of system \eqref{main_system} holds in the distributional sense: 
					for any compactly supported $\varphi_1, $ in $\Cc^\infty(\,[0,\,+\infty)\times \RR^2,\,\RR^2)$ with $\Div\,\varphi_1\,=\,0$, 
		\begin{equation*}
		\begin{aligned}
			\rho_0\int_{\RR^2}\uu(t,\,x)\cdot \varphi_1(t,\,x)\dd x\,+\,(\eta+\zeta)\int_0^t\int_{\RR^2} \nabla \uu(s ,\,x): \nabla \varphi_1(s ,\,x)\dd x\dd s \,=\\=\,
			\rho_0\int_{\RR^2}\uu_0(x)\cdot \varphi_1(0,x)\dd x\,+
			\int_0^t\rho_0\int_{\RR^2}\uu(s ,\,x)\cdot \partial_t \varphi_1(s ,x)\dd x\dd s \,+\\+\,
			\rho_0\int_0^t \int_{\RR^2}\uu(s ,x)\otimes\uu(s ,x):\nabla \varphi_1(t,x)\dd x\,+\\+
			\mu_0 \int_0^t\int_{\RR^2}(M(s ,x)\cdot \nabla H(t,x))\cdot \varphi_1(s ,x)\dd x\dd s \,-\,
			2\zeta
			\int_0^t\int_{\RR^2}\omega(s ,\,x)\,\curl \varphi_1(s ,x) \dd x\dd s ,
		\end{aligned}
		\end{equation*}
		for almost any $t\in(0,T)$.
		\item[(iii)] The angular momentum equation of system \eqref{main_system} holds in the distributional sense:
		for any compactly supported $\varphi_2\in \Cc^\infty(\,[0,\,+\infty)\times \RR^2)$
		\begin{equation*}
			\begin{aligned}
				\rho_0k\int_{\RR^2}\omega(t,x)\varphi_2(t,x)\dd x\,+\,
				(\eta'\,+\,\zeta)
				\int_{\RR^2}\nabla \omega(t,x)\cdot \nabla \varphi_2(t,x)\dd x
				\,=\\=\,
				\int_0^t\rho_0k\int_{\RR^2}\omega(s ,x)\partial_t\varphi_2(s ,x)\dd x\dd s \,+\,
				\rho_0k\int_{\RR^2}\omega_0(x)\varphi_2(0,x)\dd x\,+\\+\,
				4\zeta\rho_0\int_0^t\int_{\RR^2}\omega(s ,x)\varphi_2(s ,x)\dd x\dd s \,=
				\rho_0k\rho_0\int_0^t\int_{\RR^2}\omega(s ,x)\uu(s ,x)\cdot \nabla \varphi_2(t,x)\dd x\dd s \,+\\+\,
				\mu_0\rho_0\int_0^t\int_{\RR^2}M(s ,x)\times H(t,x)\varphi_2(s ,x)\dd x\dd s \,+\,
				2\zeta\int_{\RR^2}\uu(s ,x)\times \nabla \varphi_2 (s ,x)\dd x\dd s ,
			\end{aligned}
		\end{equation*}
		for almost any $t\in (0,T)$.
		\item[(iv)] The magnetizing equation and the magnetostatic equations hold in the distributional sense: for any compactly supported $\varphi_3\in \Cc^\infty(\,[0,\,+\infty)\times \RR^2, \RR^2)$ and 
					compactly supported $\varphi_4$ in $\Cc^\infty(\RR^2, \RR^2)$
		\begin{equation*}
		\begin{aligned}
			\int_{\RR^2}M(t,\,x)\cdot \varphi_3(t,\,x)\dd x\,+\,(\eta+\zeta)\int_0^t\int_{\RR^2} \nabla M(s ,\,x): \nabla \varphi_3(s ,\,x)\dd x\dd s 
			=\\=\,
			\int_0^t\int_{\RR^2}M(s ,\,x)\cdot \varphi_3(s ,\,x)\dd x\dd s \,+
			\int_{\RR^2}M_0(x)\cdot \varphi_3(0,x)\dd x\,+\\+\,
			\int_0^t \int_{\RR^2}\uu(s ,x)\otimes M(s ,x):\nabla \varphi_3(s ,x)\dd x\dd s \,+
			\int_0^t\int_{\RR^2}\omega(s ,\,x)M(s ,\,x)\times \varphi_3(s ,x)\dd x\dd s \,-\\-\,
			\frac{1}{\tau}
			\int_0^t\int_{\RR^2}(M(s ,\,x)-H(s ,\,x))\cdot \varphi_3(s ,x) \dd x\dd s ,
		\end{aligned}
		\end{equation*}
		for almost any time $t\in(0,T)$.
	\end{itemize}
	The solution is said to be global if the previous properties are satisfied for all fixed time $T>0$.
\end{definition}

\noindent
We state our first main result. It asserts the existence and uniqueness of weak solutions to our system, for any initial datum and external force.
\begin{theorem}\label{thm:uniqueness}
	For any initial datum $(\uu_0,\,\omega_0,\,M_0,\,H_0)$ in $L^2(\RR^2)$ and any external force 
	\begin{equation*}
		F\in L^2_{loc}(\RR_+; L^2(\RR^2))\cap W^{1,2}_{loc}(\RR_+, \Hh^{-1}(\RR^2)),
	\end{equation*} 
	there exists a unique global in time weak solution $(\uu,\,\omega,\,M,\,H)$ to system \eqref{main_system} in the sense of Definition \ref{def-weaksol}.
	Moreover for any $T>0$, such a solution satisfies the following energy inequality, for all time $t\in [0,T]$:
	\begin{equation}\label{eq:L2_energy_bound}
		\EE(t)\,+\,\int_0^t \DD(s)\dd s\,\leq \, C\left(\,\EE(0)\,+\, 
		\int_0^t
		\left[ 
			\|\,(F(s)\,\|_{L^2(\RR^2)}^2\,+\,\|\,(\,F(s),\,\partial_t F(s)\,)\,\|_{\Hh^{-1}(\RR^2)}^2
		\right]\dd s\,
		\right)
	\end{equation}
	where the energy $\EE(t)$ and its dissipation $\DD(t)$ are determined by
		\begin{equation*}
		\begin{aligned}
			\EE(t)\,&=\,\rho_0\,\|\,\uu\,\|_{L^2_x}\,+\,\rho_0k\|\,\omega\,\|_{L^2_x}^2\,+\,\mu_0\,\|\,H(t)\,\|_{L^2_x}^2\,+\,\|\,M\,\|_{L^2_x}^2,\\
			\DD(t)\,&=\,\,\|\,\nabla \uu\,\|_{L^2_x}\,+\,\|\,\nabla \omega\,\|_{L^2_x}^2\,+\,\|\,\nabla M\,\|_{L^2_x}^2\,+\,\|\,\Div\, M\,\|_{L^2_x}^2\,+\,\|\, M\,\|_{L^2_x}^2+\,\|\,H\,\|_{L^2}^2.
		\end{aligned}
		\end{equation*}
\end{theorem}

\noindent
The existence part of weak solutions will be studied in Section \ref{sec:existence}; the methodology adopted is a rather classical use of Faedo-Galerkin approximation scheme, with an interesting technical challenge. The approximated system reads in fact as
\begin{equation*}
\partial_t V_n = g_n, 
\end{equation*}
where $ V_n\in L^2_{\loc}\pare{\bR_+; H^1} $ (we refer to the uniform energy bounds provided in \cite{Stefano2}) while we can say that $ g_n\in L^1_{\loc}\pare{\bR_+; H^1} $ \textit{only}. Such low time-regularity does not allows us to use compactess theorems such as Aubin-Lions lemma \cite{Aubin} in order to deduce that the sequence $ \pare{V_n}_n $ is \textit{strongly} compact in some suitable topology. In this setting hence we use an approach similar to the one adopted in \cite[pp. 69--71]{Lions69} (we refer as well to \cite{Amirat10}); such approach consists in providing a uniform bound for the sequence
\begin{equation*}
\pare{\pare{1+\av{\partial_t}^\gamma}V_n}_n, \hspace{1cm} \ \gamma \in\pare{0, \frac{1}{4}}, 
\end{equation*}
which is indeed less restrictive than providing a uniform bound for the sequence $ \pare{\partial_t V_n}_n $. We prove (and we refer to Section \ref{sec:existence} and Appendix \ref{sec:compactness} for details) that, given any $ 0<T<\infty $, if the sequence $ \pare{\pare{1+\av{\partial_t}^\gamma}V_n}_n $ is uniformly bounded in $ L^2\pare{\bra{0, T}; H^{-N}}, \ N\gg 1 $ and $ \pare{V_n}_n $ is uniformly bounded in $  V_n\in L^2\pare{\bra{0, T}; H^1} $ then the sequence $ \pare{V_n}_n $ is compact in the space $ L^2\pare{\bra{0, T}; H^s_{\loc}}, \ s\in\pare{-N, 1} $. 

\noindent
The main difficulties associated with treating the uniqueness of weak solutions for system \eqref{main_system} are related in first place to the presence of the Navier-Stokes part, in particular to the conservative contribution of the Lorentz force. We should essentially think of the system \eqref{main_system} as an highly non-trivial perturbation of a Navier-Stokes system. It is known that for Navier--Stokes alone the uniqueness of weak solutions in $2D$ can be achieved through rather standard arguments, while in $3D$ it is  a major problem.

\noindent
The extended system that we deal with has an intermediary position. Indeed, the non-linear perturbation produced by the presence of the Lorentz force generates a significant technical challenge, which should first be attributed to the low regularities available both for the  magnetizing and demagnetizing field, $M$ and $H$. To deal with this issue and estimate the difference betweeen two solutions with same initial data, a rather common way is to introduce a weak norm being below the natural spaces in which the weak solutions are defined.
This approach is not new in literature, and it was used before in the context of the hydrodynamics of liquid cyrstals \cite{Li} as well as for the usual Navier-Stokes system in \cite{Furioli} and \cite{Marchand}.
We mention that evaluating the difference of two solutions at the same regularity level of the standard energy is not enough for our purpose. Indeed it would only allow to prove a weak-strong uniqueness result, along the same lines of \cite{Stefano2}.

\noindent
In our case, for technical convenience we use a homogeneous Sobolev space, namely $\Hh^{-1/2}(\RR^2)$. In particular, we are allowed to proceed with a negative regularity since the difference between two weak solution (with same initial data) is null at initial time $t=0$.  [35].
One of the main reasons for choosing the homogeneous setting is a specific product law, reflecting the continuity of the product within the functional spaces
\begin{equation*}
	\Hh^{s}\pare{\RR^\dd}\times \Hh^t\pare{\RR^\dd}\,\rightarrow\,\Hh^{s+t-\frac{\dd}{2}}\pare{\RR^\dd}
\end{equation*}
under suitable conditions for $s$ and $t$. 

\noindent
Our main work is then to determine a standard Gronwall inequality of the form
\begin{equation*}
	\Phi'(t)\,\leq\,f(t)\,\Phi(t),
\end{equation*}
$f(t)$ being a locally integrable function and  $\Phi(t)$ being a norm between the difference of two solutions. 
To this end, we need to overcome certain difficulties that are specific to this system. These are mainly of two different types, being related to:
\begin{itemize}
	\item Controlling the ``extraneous'' maximal derivatives, such as the highest derivatives of $M$ and $H$ in  the $\uu$-equation.
	\item Control any nonlinear terms that easily cancel at an $L^2(\RR^2)$ energy level, but persist in $\Hh^{-1/2}(\RR^2)$, because of the negative regularity.
\end{itemize}
The first issue is dealt with a specific cancellation property of the coupling system, that allows to 
simplify the worst terms, when considering certain physically meaningful combinations. 
The second one is dealt with by the mentioned product law in Sobolev space: introducing suitable indexes of regularity, we gather a standard Gronwall inequality, being careful to make use only of the norms provided by weak solutions.

Section \ref{sec:long-time} is instead dedicated to the qualitative analysis of the long-time dynamic of the global weak solutions constructed in Theorem \ref{thm:uniqueness}. The system being dissipative, we address the rate of convergence of solutions towards equilibrium. In particular we prove the following result:

\begin{theorem}\label{thm:long-time-dyn}
		Let $(\uu_0,\,\omega_0,\,M_0,\,H_0)$ be in $L^2(\RR^2)$ and $F$ in $W^{1,2}_{loc}(\RR_+, \Hh^{-1}(\RR^2))$.
		Denote by $(\uu,\,\omega,\,M,\,H)$ the unique weak solution to \eqref{main_system}  
		given by Theorem \ref{thm:uniqueness}.	Suppose that there exists a positive constants $K$ and an 
		exponent $\eta\in ]0,1[$ such that, for almost every $t>0$ one has
		\begin{equation}\label{eq:long_time_integrability_F}
			\|\,F(t)\,\|_{L^2(\RR^2)}^2
			\,+\,
			\|\,F(t)\,\|_{\Hh^{-1}(\RR^2)}^2
			+\,
			\|\,\partial_t F(t)\,\|_{H^{-1}(\RR^2)}^2\,\leq\,\frac{K}{(1+t)^{1+\eta}}
		\end{equation}
		Then, for any $\alpha<\eta$ there exists a constant $C_\alpha$ such that the following decay property is satisfied
		\begin{equation*}
			\|\,(\uu(t),\,\omega(t),\,M(t),\,H(t))\,\|_{L^2(\RR^2)}\leq\,\frac{C}{(1+t)^{\alpha}}.
		\end{equation*}
	\end{theorem}

	We would like to point out that,  under the integrability hypothesis \eqref{eq:long_time_integrability_F}, Theorem \ref{thm:long-time-dyn} asserts that the weak solutions constructed in Theorem \ref{thm:uniqueness} belong to the space $ L^2\pare{\bR_+; L^2\pare{\bR^2}} $. The approach we use in order to prove Theorem \ref{thm:long-time-dyn} is reminiscent to the one developed by M. Schonbek in \cite{Schonbek} in order to deduce decay bounds for solutions \textit{\`a la Leray} of the incompressible \NS equations.  The main strategy is the one of the Fourier splitting method, which express how the low frequencies of solutions determine the entire behavior, when considering a sufficient large time. The energy of solutions decays at the rate expected from the linear part: this fact is achieved by means of a Fourier localization within certain neighborhoods of the origin, whose sizes are time dependent, being first related to the decays of the forcing term $F$.

	At last we study the propagation of higher order Sobolev regularity for the system \eqref{main_system}. As already mentioned system \eqref{main_system} can be thought  as a highly-nontrivial perturbation of the incompressible \NS equations, in such setting we expect hence that the solutions \textit{\`a la Leray} constructed in Theorem \ref{thm:uniqueness} are solutions with \textit{critical regularity} for the parabolic system \eqref{main_system}, and being so the regularity of the solutions constructed in Theorem \ref{thm:uniqueness} should suffice to propagate any Sobolev  subcritical regularity. The main point is to establish a priori estimates, since both existence and uniqueness of solutions at this level of regularity are a straightforward adaptation of the analysis previously carried out.
Therefore we just focus on an energy ground. 
We first introduce the the following quantities
\begin{equation*}
\begin{alignedat}{32}
	\EE_s(t)\,&:=\,\rho_0 \|\,\uu(t)\, &&\|_{\Hh^s(\RR^2)}^2\,+\,\rho_0k\, \|\,\omega(t)\,  &&&&\|_{\Hh^s(\RR^2)}^2\,+\,\|\,M(t)\, &&&&&&&&\|_{\Hh^s(\RR^2)}^2\,\\
	\EE_s(0)\,&:=\,\rho_0 \|\,\;\uu_0\,&&\|_{\Hh^s(\RR^2)}^2\,+\,\rho_0k\, \|\,\;\omega_0\, &&&&\|_{\Hh^s(\RR^2)}^2\,+\,\|\,\;M_0\,&&&&&&&&\|_{\Hh^s(\RR^2)}^2\,,\,\\
\end{alignedat}
\end{equation*}
together with the dissipation
\begin{equation*}
	\DD_s(t)\,:=\,\|\,\nabla \uu(t)\,\|_{\Hh^s(\RR^2)}^2\,+\, \|\,\nabla \omega(t)\,\|_{\Hh^s(\RR^2)}^2\,+\,\|\,\nabla M(t)\,\|_{\Hh^s(\RR^2)}^2.
\end{equation*}
We then aim at proving the following statement:
\begin{theorem}\label{thm-prop-reg}
	Let us assume the initial data ${\rm U}_0\,:=\,(\uu_0,\,\omega_0,\,M_0,\,H_0)$ is in the non-homogeneous space $H^s(\RR^2)$, for a positive real $s>0$. Assuming the source term $F$ in 
	$
	L^2_{loc}(\RR_+, H^s(\RR^2))\cap W^{1,2}_{loc}(\RR_+,H^{s-1}(\RR^2)).
	$ 
	Let $(\uu,\,\omega,\,M,\,H)$ be the unique solution of system \eqref{main_system} given by Theorem \ref{thm:uniqueness}. Then
	\begin{equation*}
		(\uu,\,\omega,\,M,\,H)\,\in\,L^\infty_{\loc}(\RR_+,H^s(\RR^2))\cap L^2_{\loc}(\RR_+, H^{s+1}(\RR^2))
	\end{equation*}
	and the following dissipation formula holds true:
	\begin{equation*}
		\frac{1}{2}\EE_s(t)\,+\,\int_0^t\DD_s\pare{ t'}\dd t'\,\leq\, \Psi_s(\,{\rm U}_0,\,F,\,t),
	\end{equation*}
	where
	\begin{equation*}
		\Psi_s(\,{\rm U}_0,\, F,\,t)
		\,=\,
		\bigg(
		\frac{1}{2}\EE_s(0)
		\,+\,
		C
		\int_0^t
		\Big[
			\|\,				F\pare{ t'}\,\|_{\Hh^s}^2\,+\,
			\|\,				F\pare{ t'}\,\|_{\Hh^{s-1}}^2\,+\,
			\|\,	\partial_t	F\pare{ t'}\,\|_{\Hh^{s-1}}^2\,
		\Big]
		\dd t'
		\bigg),
	\end{equation*}
	with $C$, a suitable positive constant which depends on.
\end{theorem}

Compared to the main result proved in \cite{Stefano2} we can immediately remark that Theorem \ref{thm-prop-reg} improves the hi-regularity bounds in two ways; 

\begin{itemize}

\item At first Theorem \ref{thm-prop-reg} provides uniform bounds for any fractional Sobolev regularity $ H^s, \ s > 0 $. On the contrary in \cite{Stefano2} the propagation of hi-order regularity is proved for any $ H^k, \ k\in\bN, \ k\geq 1 $ only. Such generalization of the result is possible thanks to the tools of paradifferential calculus which will be used in the proof of Theorem \ref{thm-prop-reg}. 

\item Secondly we remark that, reading in detail the proof provided in \cite{Stefano2}, that the bound provided in \cite{Stefano2} are of the form
\begin{equation*}
\norm{U\pare{t}}_{H^k} + c\norm{\nabla U\pare{t}}_{H^k}\lesssim \widetilde{e}_k\pare{t}, 
\end{equation*}
where $ \widetilde{e}_k $ is defined inductively as
\begin{equation*}
\left\lbrace
\begin{aligned}
& \widetilde{e}_0\pare{t} = C\ e^t, \\
& \widetilde{e}_{n+1}\pare{t} =C\   \exp\set{\big.\ C\  \widetilde{e}_n\pare{t}\ },
\end{aligned}
\right. 
\end{equation*}
while Theorem \ref{thm-prop-reg} asserts that we can bound \textit{uniformly} any Sobolev norm with a uniform bound, improving remarkably the bounds provided in \cite{Stefano2}. 

\end{itemize}

The present paper is hence so structured

\begin{itemize}

\item In Section \ref{sec:existence} we prove by compactness methods that there exist solutions \textit{\`a la Leray} for the system \ref{main_system}. As already mentioned above such result is attained proving that the sequence $ \pare{V_n}_n $ of approximated solutions (in the sense of a Galerkin scheme, see system \eqref{eq:Rosensweig2D_approximated}) of system \eqref{main_system} are bounded in the space $ L^2\pare{\bra{0, T}; H^1} $, and proving that the sequence $ \pare{\av{\partial_t}^\gamma V_n}_n, \ 0<\gamma< 1/4  $ is uniformly bounded in $ L^2\pare{\bra{0, T}; H^{-N}}, \ N\gg 1 $, whence $ V_n\rhu V $ in $ H^\gamma\pare{\bra{0, T}; H^1, H^{-N}} $ (see \eqref{eq:def_Hgamma}) and thanks to Theorem \ref{thm:AL_fractional} $ V_n\to V $ in $ L^2_{\loc}\pare{\bR^2\times \bR_+} $. 

\item Section \ref{sec-uniq} is dedicated to prove that the solutions constructed in Section \ref{sec:existence} are unique in $ \dot{H}^{-1/2} $, concluding the proof of Theorem \ref{thm:uniqueness} initiated in Section \ref{sec:existence}. The proof consists in a careful and rather technical energy estimate performed in the lower regularity norm  $ \dot{H}^{-1/2} $. 

\item  Section \ref{sec:long-time} is devoted to the proof of Theorem \ref{thm:long-time-dyn}, under suitable assumption of decay on the external magnetic fields (see \eqref{eq:long_time_integrability_F}) we prove that the (unique) weak solution constructed in Theorem \ref{thm:uniqueness} decay as $ \left\langle t \right\rangle ^{-\alpha} $ for some $ \alpha > 0 $. This entails a decay result for the solutions constructed in Theorem \ref{thm:uniqueness}, improving the integrability results of Theorem \ref{thm:uniqueness}. 

\item Finally in Section \ref{sec:higher_regularity} we prove that the solutions constructed in Theorem \ref{thm:uniqueness} are critical and allow the propagation of any (nonhomogeneous) Sobolev regularity $ H^s, \ s>0 $. Such result is attained via some energy estimated performed using some paradifferential calculus techniques, such as the Littlewood-Paley dyadic decomposition and Bony paraproduct rules (see Appendix \ref{app:LP} and \cite[Chapter 2]{BCD}).

\end{itemize}

%
%

\section{Existence of weak solutions} 
\label{sec:existence}

This section is devoted to prove Theorem \ref{thm:uniqueness}, we begin with mentioning that the energy estimate provided in the statement of Theorem \ref{thm:uniqueness} has already been proven in \cite[Lemma 3.2]{Stefano2}. For this reason this section deals directly with a sequence of regularized solutions for the main system \eqref{main_system} (see as well Lemma \ref{lem:energy_bounds_approximate_system}). \\

We construct global weak solutions \textit{\`a la Leray} for the system \eqref{main_system} in a rather classical fashion using an approximation scheme. A similar result was proved in \cite{Stefano2}, we stress out though that in \cite{Stefano2} the solutions constructed belonged to the space 
\begin{align*}
\pare{\uu, \omega, M, H}& \in \Cc\pare{\RR_+; H^{\frac{1}{2}}\pare{\bR^2}}\cap L^\infty_{\loc}\pare{\RR_+; H^{\frac{1}{2}}\pare{\bR^2}},\\
\pare{ \nabla \uu, \nabla \omega, \nabla M, \nabla H} & \in L^2_{\loc}\pare{\RR_+; H^{\frac{1}{2}}\pare{\bR^2}},
\end{align*}
whence \cite{Stefano2} dealt with construction of global weak solutions with \textit{hi-regularity}. \\

Let us fix some notation. Let us consider a generic Banach space $ X $ and let us consider three function $ f\in L^2\pare{\bR; X}, \ g\in L^2 \pare{\bR^d} $ and $ h\in L^2\pare{\bR^{d+1}} $. Let us define respectively
\begin{align*}
\cF_t f \pare{\tau, x} & = \int _{-\infty}^{+\infty} f\pare{x, t} e^{-2\pi i \ \tau t}\d t
\\
\cF_x g \pare{\xi} 
& = \int _{\bR^d} g\pare{x} e^{-2\pi i \ \xi\cdot x}\dx \hspace{7mm} = \hat{g}\pare{\xi},\\
\Ft{h}\pare{\xi, \tau} & = \int_{-\infty}^{+\infty} \int _{\bR^d} h\pare{x, t} e^{-2\pi i\pare{\xi\cdot x + \tau t}}\dx\ \d t.
\end{align*}
Indeed $ \cF_t \cF_x h = \Ft{h} $.  \\

Given two topological spaces $ A, B $ we  say that $ A\subset B $ if $ A $ is topologically included in $ B $, and if the inclusion is compact we denote it as $ A\Subset B $.  Let us consider now three Banach spaces $ X_0, X_1 $ and $ X $ so that
\begin{equation}\label{eq:chain_Banach}
X_0\Subset X \subset X_1, \hspace{5mm} X_i, \ i=0, 1  \text{ is reflexive}. 
\end{equation}
Given any $ T, \gamma > 0 $ let us define the spaces
\begin{equation}\label{eq:def_Hgamma}
\begin{aligned}
H^\gamma \pare{\bR\big.  ; X_0, X_1} & = \set{v\in L^2\pare{\bR; X_0} \  \left| \ \av{\tau}^\gamma \cF_t{v} \in L^2\pare{\bR; X_1}\Big.  \right.  }, \\
H^\gamma \pare{\bra{0, T} \big. ; X_0, X_1} & = \text{ restriction of }H^\gamma \pare{\bR \big.  ; X_0, X_1} \text{ onto } \bra{0, T}. 
\end{aligned} 
\end{equation}

In the proof of the existence result we will use the following result, for a proof of which we refer the reader to Appendix \ref{sec:compactness}; 
\begin{theorem}\label{thm:AL_fractional}
Let us suppose $ X_0, X_1, X $ are Hilbert spaces which satisfy the condition \eqref{eq:chain_Banach}, then the space $ H^\gamma \pare{\bra{0, T} \big.  ; X_0, X_1}  $ is compactly embedded in $ L^2\pare{\bra{0, T} \big. ; X} $. 
\end{theorem}

We have now all the ingredients required in order to construct global weak solutions for the system \eqref{main_system}, the procedure is rather similar to \cite[Section 4]{Stefano2}. Let us define the following truncation operator
\begin{equation*}
\cJ_n v = \cF^{-1}\pare{1_{\set{\frac{1}{n}\leqslant\av{\xi}\leqslant n}} \hat{v}\pare{\xi}}, 
\end{equation*}
which localize a tempered distribution $ v $ away from  low and  high frequencies. With such we can define the following sequence of approximating systems of \eqref{main_system} 

\begin{equation}\label{eq:Rosensweig2D_approximated1}
\left\lbrace
\begin{aligned}
& \rho_0 \pare{\partial_t \uu_n +\cJ_n \pare{\uu_n \cdot \nabla \uu_n} } -\pare{\eta + \zeta} \Delta \uu_n + \nabla p_n = \mu_0 \  \cJ_n \pare{  M_n\cdot \nabla H_n} + 2\zeta \pare{\begin{array}{c}
\partial_2 \omega_n\\ -\partial_1\omega_n
\end{array}}, \\
& \rho_0 k \pare{\partial_t \omega_n + \cJ_n\pare{ \uu_n\cdot \nabla \omega_n}} - \eta' \Delta \omega_n = \mu_0 \cJ_n \pare{ M_n\times H_n} + 2\zeta \pare{\curl \ \uu_n -2\omega_n}, \\
& \partial_t M_n + \cJ_n\pare{ \uu_n\cdot \nabla M_n} - \sigma \Delta M_n = \cJ_n\pare{\pare{
\begin{array}{c}
-M_{2, n}\\ M_{1, n}
\end{array}
}\omega_n} -\frac{1}{\tau}\pare{M_n-\chi_0 H_n }, \\
& \div\pare{H_n+ M_n}=\cJ_n F, \\
& \div \ \uu_n = \curl \ H_n =0, \\
& \left. \pare{\uu_n, \omega_n, M_n, H_n}\right|_{t=0}= \pare{\cJ_n \uu_0, \cJ_n\omega_0, \cJ_n M_0, \cJ_n H_0 }. 
\end{aligned}
\right.
\end{equation}

Using the approximate magnetostatic equation $  \div\pare{H_n+ M_n}=\cJ_n F $ we can define $ H_n $ as a function of $ M_n $ and the external magnetic field as
\begin{equation}\label{eq:def_Hn}
H_n = -\cQ M_n +\cG_n, 
\end{equation}
where $ \cQ = \Delta^{-1}\nabla\div $ and  $ \cG_n= \nabla\Delta^{-1} \cJ_n F $. Whence, denoting as $ \cP $ the Leray projector onto divergence-free vector fields, explicitly  defined as
\begin{equation*}
\begin{aligned}
\cP v &  =  \pare{1-\cQ}v, \\
 & = \pare{1- \Delta^{-1}\nabla\div} v, 
\end{aligned}
\end{equation*}
we can write the approximated system \eqref{eq:Rosensweig2D_approximated1} in an evolutionary form:

\begin{equation}\label{eq:Rosensweig2D_approximated}
\left\lbrace
\begin{aligned}
& 
\begin{multlined}
\rho_0 \pare{\Big. \partial_t \uu_n +\cP \cJ_n \pare{\cP \uu_n \cdot \nabla \cP \uu_n} } -\pare{\eta + \zeta} \Delta \uu_n \\
 = \mu_0 \  \cP \cJ_n \pare{\Big.   M_n\cdot \nabla \pare{-\cQ M_n +\cG_n}} + 2\zeta \cP \pare{\begin{array}{c}
\partial_2 \omega_n\\ -\partial_1\omega_n
\end{array}},
\end{multlined}
 \\[5mm]
&
\begin{multlined}
 \rho_0 k \pare{\Big. \partial_t \omega_n + \cJ_n\pare{ \cP \uu_n\cdot \nabla \omega_n}} - \eta' \Delta \omega_n
 \\
 = \mu_0 \cJ_n \pare{ \Big.  M_n\times \pare{-\cQ M_n +\cG_n}} + 2\zeta \pare{ \curl \ \cP \uu_n -2\omega_n}, 
\end{multlined} 
 \\[5mm]
&
\begin{multlined}
 \partial_t M_n + \cJ_n\pare{ \cP \uu_n\cdot \nabla M_n} - \sigma \Delta M_n
\\
  = \cJ_n\pare{\pare{
\begin{array}{c}
-M_{2, n}\\ M_{1, n}
\end{array}
}\omega_n} -\frac{1}{\tau}\pare{M_n-\chi_0 \pare{-\cQ M_n +\cG_n} },
\end{multlined} 
 \\[5mm]
& \left. \pare{\uu_n, \omega_n, M_n}\right|_{t=0}= \pare{\cJ_n \uu_0, \cJ_n\omega_0, \cJ_n M_0 }. 
\end{aligned}
\right.
\end{equation}

Let us define the space
\begin{equation*}
L^2_n = \set{v \in L^2\pare{\bR^2}\ \left| \ \text{Supp} \ \hat{f} \subset B_n\pare{0}\setminus B_{1/n}\pare{0} \right. }, 
\end{equation*}
which endowed with the $ L^2\pare{\bR^2} $ scalar product is an Hilbert space. 
Denoting $ V_n=\pare{\uu_n, \omega_n, M_n} $ we can say that \eqref{eq:Rosensweig2D_approximated} can be written in the generic form
\begin{equation}\label{eq:Rosensweig2D_approximated_CL}
\left\lbrace
\begin{aligned}
& \frac{\d}{\d t} V_n = g_n\pare{V_n}, \\
& \left. V_n\right|_{t=0} =V_0, 
\end{aligned}
\right. 
\end{equation}
where $ \pare{g_n}_n $ is a sequence of nonlinear applications continuous onto $ L^2_n $, i.e.
\begin{equation}\label{eq:continuity_gn}
\norm{g_n\pare{V_n}}_{L^2_n} \leqslant C n^N \pare{ \norm{V_n}_{L^2_n}^2 +  \norm{V_n}_{L^2_n}}. 
\end{equation}
The value $ N $ in  \eqref{eq:continuity_gn} has to understood as "large" and uniform in $ n $. In such setting we can apply the following version of the Cauchy-Lipschitz theorem on the system \eqref{eq:Rosensweig2D_approximated_CL}, for a proof of which we refer the reader to \cite[Theorem 3.2, p. 124]{BCD}; 
	
	\begin{lemma}\label{lem:BUcriterion}
	Let us consider  the ordinary differential equation
	\begin{equation}\tag{ODE}\label{eq:ODE}
	\left\lbrace
	\begin{aligned}
	& \dot{\uu}= F \pare{ \uu, t }\\
	& \left. \uu \right|_{t=0} = \uu_0 \in \omega
	\end{aligned}
	\right. ,
	\end{equation}
	where $\omega$ is an open subset of a Banach space $X$. Let
\begin{equation*}
\begin{aligned}
 &F : && \omega\times \bR_+ && \to && X\\
 & && \pare{ \uu, t } && \mapsto && F \pare{ \uu, t }
\end{aligned},
\end{equation*}	
 be such that, for each $\uu_1, \uu_2 \in \omega$ there exists a function $L\in L^1_{\loc} \pare{ \bR_{+} }$ such that 
	\begin{equation*} 
	\norm{F\pare{ \uu_1, t }-F\pare{ \uu_2, t }}_X \leqslant L \pare{ t } \norm{\uu_1 - \uu_2}_X.
	\end{equation*}
	
	\noindent
	 Let us suppose moreover that
	\begin{equation*}
	\norm{F \pare{ \uu, t }}_{X} \leqslant \beta \pare{ t } M \pare{ \norm{\uu}_{X} },
	\end{equation*}
	where $M\in L^\infty_{\loc} \pare{ \R_+ }, \beta \in L^1_{\loc} \pare{ \R_+ }$. Then  there exists a unique maximal solution $\uu$ in the space $\mathcal{C}^1 \pare{ [0, t^\star ) ; X }$  of \eqref{eq:ODE},  such that, if $t^\star <\infty$,
	\begin{equation}\label{eq:BU_ODE}
	\limsup_{t\nearrow t^\star} \norm{\uu\pare{ t }}_{X} =\infty.
	\end{equation}
	\end{lemma}

	Lemma \ref{lem:BUcriterion} implies that for each $ n $ there exists a unique maximal $ T_n > 0 $ such that $ V_n $ is solution of \eqref{eq:Rosensweig2D_approximated_CL} and belongs to the space
	\begin{equation*}
	V_n\in\cC^1\pare{\left[ 0, T_n\right); L^2_n \cap H^\infty}.
	\end{equation*}
	Indeed in such setting it could happen that $ T_n\xrightarrow{n\to \infty} 0 $. In order to avoid such situation we state the following result, whose proof can be found in \cite[Lemma 3.5 and 4.2]{Stefano2}

	\begin{lemma}
\label{lem:energy_bounds_approximate_system}
Let $ V_0\in H^{\frac{1}{2}}, \ F\in L^{2}_{\loc}\pare{\bR_+; L^2}, \ \cG_F \in W^{1, \infty}_{\loc} \pare{\bR_+; H^{\frac{3}{2}}} $ and let $ V_n = \pare{\uu_n, \omega_n, M_n} $ be the unique smooth solution of \eqref{eq:Rosensweig2D_approximated_CL} identified by Lemma \ref{lem:BUcriterion}. Fixed any  $ T>0 $  there exists a positive, finite, constant $ \widetilde{c} =  \widetilde{c}\pare{T, \rho_0, \eta, \zeta, \mu_0, \kappa, \eta' , \sigma ,  \tau } $ independent of $ n $ and  $ t\in\bra{0, T} $,  such 
\begin{align*}
\norm{V_n}_{L^\infty \pare{\bra{0, t}; L^2}} & \leqslant \widetilde{c} , \\
\norm{ M_n}_{L^2 \pare{\bra{0, t}; L^2}} &  \leqslant {\widetilde{c}}, \\
\norm{\nabla V_n}_{L^2 \pare{\bra{0, t}; L^2}} &  \leqslant \frac{\widetilde{c}}{c}  , 
\end{align*}
where $ c $ is defined as
\begin{equation*}
c  = \min \set{  \frac{\eta+\zeta}{2}, \frac{\eta'}{2} , 4\zeta , \frac{\sigma}{2} , \frac{1}{\tau} , \frac{\chi_0}{2\tau} }. 
\end{equation*}
\end{lemma}

We can combine  the uniform bounds stated in Lemma \ref{lem:energy_bounds_approximate_system} with the blow-up criterion  \eqref{eq:BU_ODE} in order to deduce that $ T_n=\infty $ for any $ n $, hence
\begin{equation*}
V_n\in \cC^1\pare{ \bR_+; L^2_n}, 
\end{equation*}
moreover the uniform bounds provided in Lemma \ref{lem:energy_bounds_approximate_system} allow us to deduce that, for each $ T > 0 $
\begin{equation*}
\pare{V_n}_n \ \text{is uniformly bounded in}\ L^2 \pare{\bra{0, T}; H^{1}\pare{\bR^2}}. 
\end{equation*}

Let us now consider $ g_n=g_n\pare{V_n} $ appearing in \eqref{eq:Rosensweig2D_approximated_CL} and explicitly defined as the right hand side of \eqref{eq:Rosensweig2D_approximated}. Let us consider in particular the term
\begin{equation*}
\cJ_n \pare{ M_n\cdot\nabla \cQ M_n }, 
\end{equation*}
some simple computations using the uniform bounds provided by Lemma \ref{lem:energy_bounds_approximate_system} show us that, for any $ T >0 $ and $ n\in\bN $; 
\begin{equation*}
\norm{\cJ_n \pare{ M_n\cdot\nabla \cQ M_n }}_{L^1\pare{\bra{0, T}; \dot{H}^{-1/2}}}\leq C <\infty. 
\end{equation*}
In particular this consideration implies that the sequence $ \pare{g_n}_n $ has only $ L^1_{\loc} $ uniform regularity in-time, hence $ \pare{g_n}_n  $ can be uniformly bounded  only in some space with very low regularity such as $ L^1_{\loc}\pare{\bR; H^{-N}} $ for $ N $ large.\footnote{Let us remark that for the incompressible \NS equations it is possible to prove that $ \pare{\partial_t \uu_n}_n $ is uniformly bounded in $ L^2_{\loc}\pare{\bR_+; H^{-N}} $, making hence possible to apply Aubin-Lions lemma \cite{Aubin}. } \\

Let us now denote\footnote{Here we must cutoff the functions $ V_n, g_n $ in the interval $ \bra{0, T} $ in order to assure their $ L^2 $-in-time integrability, and hence the fact that their space-time Fourier transform is well defined. }
\begin{align*}
\widetilde{V}_n\pare{\xi, \tau} & = \Ft{\pare{ 1_{\bra{0, T}}\pare{t} \  V_n \pare{x, t} }}\pare{\xi, \tau}, \\
\widetilde{g}_n \pare{\xi, \tau} & = \Ft{\pare{ 1_{\bra{0, T}}\pare{t} \  g_n \pare{x, t} }}\pare{\xi, \tau},
\end{align*}
with such notation equation \eqref{eq:Rosensweig2D_approximated_CL} becomes
\begin{equation} \label{eq:Rosensweing_Fourier}
2\pi i  \tau \  \widetilde{V}_n \pare{\xi, \tau} = \widetilde{g}_n \pare{\xi, \tau}. 
\end{equation}
Let us denote as $ \overline{\widetilde{V}}_n\pare{\xi, \tau} $ the complex conjugate of $ {\widetilde{V}}_n\pare{\xi, \tau} $, and let us multiply  \eqref{eq:Rosensweing_Fourier} for 
\begin{equation*}
 \frac{1}{\pare{1+\av{\xi}^2}^N}\ \overline{\widetilde{V}}_n\pare{\xi, \tau}, 
\end{equation*}
$ N >0 $ large, integrate in $ \xi $ and take the imaginary part of the resulting equation and using Plancherel theorem in the space variables, we deduce
\begin{equation*}
2\pi \av{\tau} \norm{\widetilde{V}_n\pare{\tau}}_{H^{-\frac{N}{2}}}^2 \leq \norm{\widetilde{g}_n\pare{\tau}}_{H^{-{N}}} \norm{\widetilde{V}_n\pare{\tau}}_{L^2}, 
\end{equation*}
whence multiplying the above equation for $ \pare{1+\av{\tau}}^{-\beta}, \ \beta \in\left[\frac{1}{2}, \ 1\right) $ and integrating in $ \tau $ we deduce that
\begin{equation*}
 2\pi \ \int _{-\infty}^{\infty} \frac{\av{\tau}}{\pare{1+\av{\tau}}^\beta} \norm{\widetilde{V}_n\pare{\tau}}_{H^{-\frac{N}{2}}}^2 \d \tau \leq \pare{\int _{-\infty}^{\infty} \frac{\d \tau}{\pare{1+\av{\tau}}^{2\beta}} }^{\frac{1}{2}} \norm{\widetilde{g}_n}_{L^\infty\pare{\bR; H^{-N}}} \norm{\widetilde{V}_n}_{L^2\pare{\bR; L^2}}. 
\end{equation*}

Indeed if $ \beta \in\left[\frac{1}{2}, \ 1\right) $ we have that 
\begin{equation*}
\pare{\int _{-\infty}^{\infty} \frac{\d \tau}{\pare{1+\av{\tau}}^{2\beta}} }^{\frac{1}{2}} \leq C <\infty, 
\end{equation*}
moreover thanks to standard considerations on the Fourier transform we can argue that
\begin{equation*}
 \norm{\widetilde{g}_n}_{L^\infty_{\tau}\pare{\bR; H^{-N}}} \leq \norm{1_{\bra{0, T}} \ g_n}_{L^1_t\pare{\bR; H^{-N}}} = \norm{g_n}_{L^1_t\pare{\bra{0, T}; H^{-N}}} \leq C<\infty, 
\end{equation*}
moreover the bounds provided in Lemma \ref{lem:energy_bounds_approximate_system} allow us to deduce that
\begin{equation*}
\norm{\widetilde{V}_n}_{L^2_{\tau}\pare{\bR; L^2}} = \norm{V_n}_{L^2_t\pare{\bra{0, T}; L^2}}\leq T^{1/2} \norm{V_n}_{L^\infty_t\pare{\bra{0, T}; L^2}}\leq C <\infty, 
\end{equation*}
which in turn implies that
\begin{equation*}
2\pi \ \int _{-\infty}^{\infty} \frac{\av{\tau}}{\pare{1+\av{\tau}}^\beta} \norm{\widetilde{V}_n\pare{\tau}}_{H^{-\frac{N}{2}}}^2 \d \tau \leq C <\infty. 
\end{equation*}
As already explained $ \pare{\widetilde{V}_n}_n $ is uniformly bounded in $ L^2_{\tau}\pare{\bR; L^2} $, whence since $ L^2\hra H^{-\frac{N}{2}}, \ N>0 $ continuously we deduce that

\begin{equation*}
\int _{-\infty}^{\infty} \pare{1+ \frac{\av{\tau}}{\pare{1+\av{\tau}}^\beta}} \norm{\widetilde{V}_n\pare{\tau}}_{H^{-\frac{N}{2}}}^2 \d \tau \leq C <\infty, 
\end{equation*}
moreover 
\begin{equation*}
1+ \frac{\av{\tau}}{\pare{1+\av{\tau}}^\beta} \geq \av{\tau}^{2\gamma}, \hspace{7mm} \forall \ \tau \in\bR, \hspace{3mm} \gamma\in \pare{0, \frac{1-\beta}{2}}. 
\end{equation*}

\noindent
Considering hence the definition provided in \eqref{eq:def_Hgamma} we deduce that
\begin{equation*}
\pare{1_{\bra{0, T}}\ V_n}_n \text{ is uniformly bounded in } H^\gamma \pare{\bR; \  H^1\pare{\bR^2}, H^{-\frac{N}{2}}\pare{\bR^2}} \text{ for } \gamma\in \pare{0, \frac{1}{4}}. 
\end{equation*}

\noindent
Using now the result stated in Theorem \ref{thm:AL_fractional} we can state that
\begin{equation*}
H^\gamma \pare{\bra{0, T}; \  H^1_{\loc}\pare{\bR^2}, H^{-\frac{N}{2}}_{\loc}\pare{\bR^2}}
\Subset
L^2_t\pare{\bra{0, T};  H^{1-\varepsilon}_{\loc}\pare{\bR^2}},  \ \forall\  \varepsilon \in \pare{0, 1+\frac{N}{2}}
\end{equation*}
from which we deduce that there exists\footnote{Here indeed we implicitly use the continuous embedding 
\begin{equation*}
H^\gamma \pare{\bR; \  H^1\pare{\bR^2}, H^{-\frac{N}{2}}\pare{\bR^2}}
\hra 
H^\gamma \pare{\bra{0, T}; \  H^1_{\loc}\pare{\bR^2}, H^{-\frac{N}{2}}_{\loc}\pare{\bR^2}}
\end{equation*}
} a $ V = \pare{\uu, \omega , M}\in L^2_t\pare{\bra{0, T}; L^2_{\loc, x}\pare{\bR^2}} $ such that 
\begin{equation}\label{eq:conv_Leray}
V_n\xrightarrow{n \to \infty} V \hspace{4mm}\text{ in } L^2\pare{\bra{0, T}; \ H^{1-\varepsilon}_{\loc}\pare{\bR^2}},  \ \forall\  \varepsilon \in \pare{0, 1+\frac{N}{2}}. 
\end{equation}

\noindent Considering the uniform bounds provided by Lemma \ref{lem:energy_bounds_approximate_system} for the sequence $ \pare{V_n}_n $ we can say that
\begin{align*}
V & \in L^\infty\pare{\bra{0, T}; L^2\pare{\bR^2}}\cap L^2\pare{\bra{0, T}; \dot{H}^1\pare{\bR^2}}, \\
M & \in L^2\pare{\bra{0, T}; L^2\pare{\bR^2}}. 
\end{align*}

At last we must hence to prove that each limit point $ V=\pare{\uu, \omega, M} $ solves the system
\begin{equation}\label{eq:Rosensweig2D_limit}
\left\lbrace
\begin{aligned}
& 
\begin{multlined}
\rho_0 \pare{\Big. \partial_t \uu +\cP  \pare{\cP \uu \cdot \nabla \cP \uu} } -\pare{\eta + \zeta} \Delta \uu \\
 = \mu_0 \  \cP  \pare{\Big.   M\cdot \nabla \pare{-\cQ M +\cG_F}} + 2\zeta \cP \pare{\begin{array}{c}
\partial_2 \omega\\ -\partial_1\omega
\end{array}},
\end{multlined}
 \\[3mm]
&
\begin{multlined}
 \rho_0 k \pare{\Big. \partial_t \omega + \pare{ \cP \uu\cdot \nabla \omega}} - \eta' \Delta \omega
 \\
 = \mu_0  \pare{ \Big.  M\times \pare{-\cQ M +\cG_F}} + 2\zeta \pare{ \curl \ \cP \uu -2\omega}, 
\end{multlined} 
 \\[5mm]
&
\begin{multlined}
 \partial_t M + \pare{ \cP \uu\cdot \nabla M} - \sigma \Delta M
\\
  = \pare{
\begin{array}{c}
-M_{2}\\ M_{1}
\end{array}
}\omega -\frac{1}{\tau}\pare{M-\chi_0 \pare{-\cQ M +\cG_F} },
\end{multlined} 
 \\[5mm]
& \left. \pare{\uu, \omega, M}\right|_{t=0}= \pare{ \uu_0, \omega_0, M_0 },
\end{aligned}
\right.
\end{equation}
in $ \cD'\pare{\bR^2\times\bra{0, T}} $. Let us remark that system \eqref{eq:Rosensweig2D_limit} is equivalent to \eqref{main_system}. Such procedure is rather standard and was carried out in \cite{Stefano1}, nonetheless we sketch a proof  for the sake of self completeness of the work. Among all the nonlinear interactions the convergence which is less immediate to prove is the one arising from the less regular term, i.e.
\begin{equation*}
M_n\cdot\nabla \cQ M_n \xrightarrow{n\to\infty} M\cdot\nabla \cQ M,.
\end{equation*}
Considering hence a test function $ \psi \in\cD\pare{\bR^2 \times \bra{0, T}} $
\begin{multline*}
\int M_n\cdot\nabla \cQ M_n \ \psi \ \dx \d t- \int M\cdot\nabla \cQ M \ \psi \ \dx \d t \\
\begin{aligned}
 & = \int \pare{M_n -M}\cdot \nabla \cQ M_n \ \psi \ \dx \d t + \int M\cdot\nabla \cQ \pare{M_n-M} \ \psi \ \dx \d t, \\
& = I_{1, n} + I_{2,n}. 
\end{aligned}
\end{multline*}

Applying H\"older inequality we deduce the bound
\begin{align*}
I_{1, n} & \lesssim \norm{M_n - M}_{L^2_{\loc}\pare{\bR_+; L^4_{\loc}}} \norm{\nabla\cQ M_n}_{L^2_{\loc}\pare{\bR_+; L^2_{\loc}}} \norm{\psi}_{L^\infty\pare{\bR_+; L^4}}. 
\end{align*}
Standard Sobolev embeddings and \eqref{eq:conv_Leray} imply that
\begin{equation*}
\norm{M_n - M}_{L^2_{\loc}\pare{\bR_+; L^4_{\loc}}}\leqslant \norm{M_n - M}_{L^2_{\loc}\pare{\bR_+; H^{\frac{1}{2}}_{\loc}}}\to 0 \text{ as } n\to \infty, 
\end{equation*}
moreover since $ \nabla\cQ $ is  a pseudo-differential operator of order one we argue that
\begin{equation*}
\norm{\nabla\cQ M_n}_{L^2_{\loc}\pare{\bR_+; L^2_{\loc}}} \lesssim \norm{ M_n}_{L^2_{\loc}\pare{\bR_+; H^1_{\loc}}} < C < \infty, 
\end{equation*}
thanks to the results of Lemma \ref{lem:energy_bounds_approximate_system}, proving that $ I_{1, n}\xrightarrow{n\to \infty}0 $. \\
Similarly it can be proved that $ I_{2, n}\to 0 $ as $ n\to \infty $.

\section{Uniqueness of weak solutions}\label{sec-uniq}

\noindent 
In this section we provide the proof of the uniqueness result for the weak solutions of system \eqref{main_system}. The main idea is to evaluate the difference 
between two weak solutions in a functional space which is less regular than $L^2(\RR^2)$ such as $\Hh^{-1/2}(\RR^2)$. 
For convenience, we introduce the following notation:
	\begin{equation*}
	(\delta \uu,\,\delta \omega,\,\, \delta M,\, \delta H\,)\,:=\,(\uu-\tilde \uu,\,\omega-\tilde \omega,\,\,M-\tilde M,\,H - \tilde H\,),
	\end{equation*}
	where $(\uu,\,\omega,\,M,\,H)$ and $(\tilde \uu,\,\tilde \omega,\,\tilde M,\,\tilde H)$ stand for the first and second solutions, respectively. Under the 
	assumption that the initial data of both solutions coincide, we remark that
	\begin{equation*}
		(\delta \uu,\,\delta \omega,\,\, \delta M,\, \delta H\,)_{|t=0}\,=\,0.
	\end{equation*}
	We then consider the $\Hh^{-1/2}$-energy $\delta \EE_{-1/2}(t)$ of $(\delta \uu,\,\delta \omega,\,\, \delta M,\, \delta H\,)$
	\begin{equation*}
	\begin{alignedat}{32}
		\delta \EE_{-\frac{1}{2}}(t)\,&:=\,\rho_0 \|\,\delta \uu(t)\, &&\|_{\Hh^{-\frac{1}{2}}(\RR^2)}^2\,
				+\,\rho_0k\, \|\,\delta \omega(t)\,  &&&&\|_{\Hh^{-\frac{1}{2}}(\RR^2)}^2\,+\,\|\,\delta M(t)\, &&&&&&&&\|_{\Hh^{-\frac{1}{2}}(\RR^2)}^2\,\\
	\end{alignedat}	
	\end{equation*}
	together with its corresponding dissipation
	\begin{equation*}
		\delta \DD_{-\frac{1}{2}}(t)\,:=\,
		\eta
		\|	\,	\nabla	\delta \uu		\,	\|_{\Hh^{-\frac{1}{2}}}^2	\, +\,
		\eta'
		\|	\,	\nabla	\delta \omega	\,	\|_{\Hh^{-\frac{1}{2}}}^2	\,+\,
		\lambda'
		\|	\,	\Div\,	\delta \omega	\,	\|_{\Hh^{-\frac{1}{2}}}^2	\,+\,
		\sigma
		\|	\,	\nabla	\delta M			\,	\|_{\Hh^{-\frac{1}{2}}}^2	
	\end{equation*}
	We then aim at proving a standard Gronwall type inequality, which can be formulated as follows:
	\begin{equation*}
		\frac{\dd}{\dd t}\delta \EE_{-\frac{1}{2}}(t)\,+\, \,\int_0^t\delta \DD_{-\frac{1}{2}}\pare{ t'}\dd t'\,\leq \,f(t)\,\delta \EE_{-\frac{1}{2}}(t),\quad\quad
		\delta \EE_{-\frac{1}{2}}(0)\,=\,0,
	\end{equation*}
	for a suitable function $f(t)$ which is locally integrable in time.
	
	\noindent
	We begin with considering the difference between the momentum equations of the two solutions, driving the evolution of $\delta \uu$:
	\begin{align*} 
		\;\rho_ 0\big(\partial_t \delta \uu\,+\uu\cdot\nabla \delta \uu\,&+\,\delta \uu\cdot\nabla \tilde \uu \big)\,-\,(\eta +\zeta)\,\Delta \delta \uu\,+\,\nabla \delta \pre\,=\\&= \,\mu_0\,\big(\, 
		M \cdot \nabla \delta  H\,+
		\,\delta M\cdot \nabla \tilde H \,\big)	\,+2\zeta\, \left(\,\begin{matrix}\hspace{0.25cm}\partial_2 \delta \omega\\-\partial_1 \delta \omega	\end{matrix}\,\right).
	\end{align*}
	We then take the  $\Hh^{-1/2}$-inner product of the above equation with $\delta \uu$, to get:
	\begin{equation}\label{uniq-momeq}
	\begin{aligned}
		\frac{\rho_0}{2}\frac{\dd}{\dd t}\|\,\delta \uu\,\|_{\Hh^{-\frac{1}{2}}}^2+
		(\eta +\zeta)\|\,\nabla \delta \uu\,\|_{\Hh^{-\frac{1}{2}}}^2
		=
		\rho_0\langle	\,	\uu\cdot\nabla \delta \uu		,	\delta \uu	\,	\rangle_{\Hh^{-\frac{1}{2}}}+
		\rho_0\langle	\,	\delta \uu\cdot\nabla \tilde \uu 		,	\delta \uu	\,	\rangle_{\Hh^{-\frac{1}{2}}}+\\+\,
		\mu_0\langle	\,	M \cdot \nabla \delta  H		,	\delta \uu	\,	\rangle_{\Hh^{-\frac{1}{2}}}+
		\mu_0\langle	\,	\delta M\cdot \nabla \tilde H 			,	\delta \uu	\,	\rangle_{\Hh^{-\frac{1}{2}}}+
		2\zeta\langle	\,	\left(\,\begin{matrix}\hspace{0.25cm}\partial_2 \delta \omega\\-\partial_1 \delta \omega	\end{matrix}\,\right)			,	\delta \uu	\,	\rangle_{\Hh^{-\frac{1}{2}}}.
	\end{aligned}
	\end{equation}
	We then proceed to estimate each term on the right-hand side. To this end, we will repeatedly make use of the following product law between Sobolev spaces:
	\begin{align*}
		\Hh^{\frac{1}{2}}(\RR^2)\times L^2(\RR^2)\,\rightarrow \, \Hh^{-\frac{1}{2}}(\RR^2).
	\end{align*}
	Furthermore, in order to implement the estimates of the main convective terms, we need to introduce the following inequality  at the level of homogeneous Sobolev space $\Hh^{-\frac{1}{2}}(\RR^2)$:
	\begin{lemma}\label{comm-est} For any divergence-free vector field $\rm v$ in $\Hh^{\frac{1}{2}}(\RR^2)$ and any vector field $B$ in $\Hh^{-\frac{1}{2}}(\RR^2)$
		\begin{equation*}
			\langle\,  {\rm v}\cdot \nabla B,\, B\,\rangle_{\Hh^{-\frac{1}{2}}(\RR^2)}\,\leq\, C \|\,\nabla {\rm v}\,\|_{L^2(\RR^2)}\|\,\nabla B\,\|_{\Hh^{-\frac{1}{2}}}\|\,B\,\|_{\Hh^{-\frac{1}{2}}},
		\end{equation*}
		for a suitable positive constant $C$.
	\end{lemma}
	\noindent
	For the sake of clearness, we postpone the proof of the above Lemma to the Appendix (cf. Lemma \ref{comm-est:appendix}).
	Coming back to our problem, we remark that $\Div\,\uu\,=\,0$. Thus, in accordance with Lemma \ref{comm-est}, we gather that
	\begin{align*}	
		\langle	\,	\uu\cdot\nabla \delta \uu		,	\delta \uu	\,	\rangle_{\Hh^{-\frac{1}{2}}}
		\,
		\lesssim
		\,
		\|\,\nabla			\uu 	\,\|_{L^2	}
		\|\,			\delta 	\uu		\,\|_{H^{-\frac{1}{2}}	}
		\|\,	\nabla 	\delta 	\uu		\,\|_{H^{-\frac{1}{2}}	}
		\,
		\lesssim	
		\,
		\|\,\nabla			\uu 	\,\|_{L^2	}^2
		\|\,			\delta 	\uu		\,\|_{H^{-\frac{1}{2}}	}^2\,+\,
		\ee
		\|\,	\nabla	\delta 	\uu		\,\|_{H^{-\frac{1}{2}}	}^2,
	\end{align*}
	for a suitable small constant $\ee$ to be fixed later.Here the first intrinsic cancellation holds, thanks to the free-divergence condition on $\uu$.
	Next, the second term is estimated as follows:
	\begin{align*}	
		\rho_0\langle	\,	\delta \uu\cdot\nabla \tilde \uu 		,\,	&\delta \uu	\,	\rangle_{\Hh^{-\frac{1}{2}}}
		\,\lesssim\,
		\|\,			\delta 	\uu	\cdot\nabla \tilde \uu 	\,\|_{H^{-\frac{1}{2}}	}
		\|\,			\delta 	\uu						\,\|_{H^{-\frac{1}{2}}	}\\
		&\lesssim
		\|\,			\delta 	\uu		\,\|_{\Hh^{\frac{1}{2}}	}
		\|\,\nabla			\tilde \uu 	\,\|_{L^2				}
		\|\,			\delta 	\uu		\,\|_{\Hh^{-\frac{1}{2}}	}
		\,\lesssim\, 
		\|\,\nabla			\tilde \uu 	\,\|_{L^2				}^2
		\|\,			\delta 	\uu		\,\|_{\Hh^{-\frac{1}{2}}	}^2\,+\,
		\ee
		\|\,	\nabla 	\delta 	\uu		\,\|_{\Hh^{-\frac{1}{2}}	}^2.
	\end{align*}
	Now, we consider the contribution of the Lorentz force $\delta F \,= \,\mu_0(M \cdot \nabla \delta  H + \delta M\cdot \nabla \tilde H 	)$. It is worth to emphasize 
	that the level of regularity $\Hh^{-1/2}$ is a-priori in accordance with the regularity provided by the standard $L^2$-energy, $F_1,\,F_2\in L^1_{loc}(\RR_+,\,\Hh^{-1/2}(\RR^2))$. 
	
	\noindent
	We handle the first term of $\delta F$ as follows: 
	\begin{align*}
		\mu_0&\langle	\,	M \cdot \nabla \delta  H		,	\delta \uu	\,	\rangle_{\Hh^{-\frac{1}{2}}}
		\,=\,
		-\mu_0\langle	\, (\Div\,M  )\delta  H			,	\delta \uu			\,	\rangle_{\Hh^{-\frac{1}{2}}}\,
		-\mu_0\langle	\,	M \otimes 	 \delta  H		,	\nabla	\delta \uu	\,	\rangle_{\Hh^{-\frac{1}{2}}}\,\\
		&\lesssim\,
		\|\,		\Div\,M 			\,\|_{L^2				}
		\|\,		\delta  H			\,\|_{\Hh^{	\frac{1}{2}}}
		\|\,		\delta \uu			\,\|_{\Hh^{-\frac{1}{2}}}\,+\,
		\|\,		M 					\,\|_{\Hh^\frac{3}{4}}						
		\|\,		\delta  H			\,\|_{\Hh^{-\frac{1}{4}}}
		\|\,		\nabla	\delta \uu	\,\|_{\Hh^{-\frac{1}{2}}}\\
		&\lesssim\,
		\|\,		\nabla M 			\,\|_{L^2					}
		\|\,		\nabla	\delta  H	\,\|_{\Hh^{	-\frac{1}{2}}	}
		\|\,		\delta \uu			\,\|_{\Hh^{	-\frac{1}{2}}	}\,+\,
		\|\,		M 					\,\|_{L^2}^{\frac{1}{4}}
		\|\,		\nabla M 			\,\|_{L^2}^{\frac{3}{4}}		{\scriptstyle\times}\\&\hspace{5cm}{\scriptstyle\times}				
		\|\,		\delta  H			\,\|_{\Hh^{-\frac{1}{2}}}^{\frac{3}{4}}						
		\|\,		\nabla \delta  H		\,\|_{\Hh^{-\frac{1}{2}}}^{\frac{1}{4}}
		\|\,		\nabla	\delta \uu	\,\|_{\Hh^{-\frac{1}{2}}}\\
		&\lesssim
		\|\,		\nabla M 			\,\|_{L^2					}^2
		\|\,				\delta  H	\,\|_{\Hh^{	-\frac{1}{2}}	}^2\,+\,
		\|\,		M 					\,\|_{L^2}^{\frac{2}{3}}
		\|\,		\nabla M 			\,\|_{L^2}^2
		\|\,		\delta  H			\,\|_{\Hh^{-\frac{1}{2}}}^2\,+\,
		\ee
		\|\,		\nabla	\delta  H	\,\|_{\Hh^{	-\frac{1}{2}}	}^2\,+\,
		\ee
		\|\,		\nabla	\delta  \uu	\,\|_{\Hh^{	-\frac{1}{2}}	}^2,
	\end{align*}
	where we have used an integration by parts and the following product law between Sobolev spaces:
	\begin{equation*}
		\Hh^{\frac{3}{4}}(\RR^2)\times \Hh^{-\frac{1}{4}}(\RR^2) \rightarrow \Hh^{-\frac{1}{2}}(\RR^2).
	\end{equation*}
	Next, the second term of $\delta F$ is handled through
	\begin{align*}
		\mu_0\langle	\,	\delta M\cdot \nabla \tilde H 			,	\delta \uu	\,	\rangle_{\Hh^{-\frac{1}{2}}}
		&\lesssim
		\|\,	\delta M	\,\|_{\Hh^{\frac{1}{2}}}
		\|\,	\nabla \tilde H 	\,\|_{L^2}
		\|\,	\delta \uu	\,\|_{\Hh^{-\frac{1}{2}}}\\
		&\lesssim\,
		\|\,	\nabla \tilde H 	\,\|_{L^2}^2
		\|\,	\delta \uu	\,\|_{\Hh^{-\frac{1}{2}}}^2\,+\,
		\ee
		\|\,	\nabla \delta \uu	\,\|_{\Hh^{-\frac{1}{2}}}^2.
	\end{align*}
	It remains to control the last term of the kinetic energy \eqref{uniq-momeq}, which can be treated by a standard Cauchy-Schwarz inequality:
	\begin{align*}
		2\zeta\langle	\,	\left(\,\begin{matrix}\hspace{0.25cm}\partial_2 \delta \omega\\-\partial_1 \delta \omega	\end{matrix}\,\right)				,	\delta \uu	\,	\rangle_{\Hh^{-\frac{1}{2}}}
		\,\lesssim\,
		\|\,\nabla \delta \omega	\,\|_{\Hh^{-\frac{1}{2}}}
		\|\,		\delta \uu	\,\|_{\Hh^{-\frac{1}{2}}}
		\lesssim
		\ee \|\,\nabla \delta \omega	\,\|_{\Hh^{-\frac{1}{2}}}^2
		\,+\,
		\|\,		\delta \uu	\,\|_{\Hh^{-\frac{1}{2}}}^2.
	\end{align*}
	Integrating all the previous estimates into the identity \eqref{uniq-momeq}, we obtain
	\begin{equation}
	\begin{aligned}\label{ineq1}
		\frac{\rho_0}{2}&\frac{\dd}{\dd t}\|\,\delta \uu\,\|_{\Hh^{-\frac{1}{2}}}^2+
		(\eta +\zeta)\|\,\nabla \delta \uu\,\|_{\Hh^{-\frac{1}{2}}}^2
		\lesssim\\
		&\lesssim
		\Big(
			1+\|\, M \,\|_{L^2}^\frac{2}{3}
		\Big)
		\|\,\nabla( \uu ,\, \tilde \uu ,\,M ,\, \tilde H )\,\|_{L^2}^2 	
		\|\, ( \delta \uu,\,\delta H)\,\|_{\Hh^{-\frac{1}{2}}}^2
		+\ee	\|\, \nabla( \delta \uu,\, \delta \omega,  \delta H)\,\|_{\Hh^{-\frac{1}{2}}}^2.
	\end{aligned}
	\end{equation}
	We now take into account the difference between the angular velocity equations for the two solutions of system \eqref{main_system}. We hence gather the 
	evolutionary equation for $\delta \omega$:
	\begin{equation}
	\begin{aligned}
		\;\rho_0 k\big( \partial_t \delta \omega\,+\delta \uu \cdot \nabla  \omega\,&+\,\tilde \uu\cdot \nabla \delta\omega\,\big) - \eta'\Delta \omega \,-\lambda' \nabla \Div\,\delta \omega\,=\\&=\,
		\mu_0\big(\,  M \times \delta H\,
		+\,\delta M\times  \tilde H \,\big)\,+\,2\zeta (\,\curl \delta \uu \,-\,2 \delta \omega\,).
	\end{aligned}
	\end{equation}
	We then proceed similarly as for proving \eqref{ineq1}: we take the $\Hh^{-\frac{1}{2}}(\RR^2)$-inner product of the above identity with $\delta \omega$ and we integrate by parts, to get:
	\begin{equation}\label{inertia}
	\begin{aligned}
		\frac{\dd}{\dd t}
		\rho_0 k
		\|\, \delta \omega			\,\|_{\Hh^{-\frac{1}{2}}}^2\,+\,
		\eta'
		\|\,\nabla \delta \omega	\,\|_{\Hh^{-\frac{1}{2}}}^2\,+\,
		\lambda'
		\|\,\Div \delta \omega	\,\|_{\Hh^{-\frac{1}{2}}}^2
		\,=\,
		\rho_0 k
		\langle	\,	\delta	\uu\cdot \nabla \omega 	,	\delta \omega		\,\rangle_{\Hh^{-\frac{1}{2}}}\,+\\+\,
		\rho_0 k
		\langle	\,	\tilde \uu \cdot \nabla	\delta \omega	,	\delta \omega		\,\rangle_{\Hh^{-\frac{1}{2}}}\,+\,
		\mu_0
		\langle	\,	\delta M\times H 					,	\delta \omega		\,\rangle_{\Hh^{-\frac{1}{2}}}\,+\,
		\mu_0
		\langle	\,	\delta M \times	\delta H				,	\delta \omega		\,\rangle_{\Hh^{-\frac{1}{2}}}\,+\\+\,
		2\zeta
		\langle	\,	\curl\delta \uu						,	\delta \omega		\,\rangle_{\Hh^{-\frac{1}{2}}}\,-\,
		4\zeta
		\|	\,	\delta\omega						\,\|_{\Hh^{-\frac{1}{2}}}.
	\end{aligned}
	\end{equation}
	We then proceed by estimating each term on the right-hand side. First we have
	\begin{align*}
		\rho_0 k\langle	\,	\delta	\uu\cdot \nabla \omega 	,	\delta \omega		\,\rangle_{\Hh^{-\frac{1}{2}}}
		\,&\lesssim\,
		\|\,\delta	\uu\cdot \nabla \omega 		\,\|_{\Hh^{-\frac{1}{2}}}
		\|\,\delta \omega							\,\|_{\Hh^{-\frac{1}{2}}}\\
		\,&\lesssim\,
		\|\,\delta	\uu								\,\|_{\Hh^{\frac{1}{2}}}
		\|\,\nabla \omega 							\,\|_{L^2}
		\|\,\delta \omega							\,\|_{\Hh^{-\frac{1}{2}}}\\
		&\lesssim\,
		\|\,\nabla \omega 							\,\|_{L^2}^2
		\|\,\delta \omega							\,\|_{\Hh^{-\frac{1}{2}}}^2\,+\,
		\ee
		\|\,\nabla \delta	\uu						\,\|_{\Hh^{-\frac{1}{2}}}.		
	\end{align*}
	Next, the free divergence condition on $\tilde \uu $ and Lemma \ref{comm-est} lead to
	\begin{align*}
		\rho_0 k\langle	\,	\tilde \uu \cdot \nabla	\delta \omega	,	\delta \omega			\,\rangle_{\Hh^{-\frac{1}{2}}}	
		\,&\lesssim\,
		\|\,	\nabla \tilde \uu 			\,\|_{L^2}
		\|\,	\delta \omega			\,\|_{\Hh^{\frac{1}{2}}}
		\|\,			 \delta \omega	\,\|_{\Hh^{-\frac{1}{2}}}\\
		&\,\lesssim\,
		\|\,	\nabla \tilde \uu 			\,\|_{L^2}^2
		\|\,	\delta \omega			\,\|_{\Hh^{-\frac{1}{2}}}^2
		\,+\,\ee
		\|\,	\nabla \delta \omega	\,\|_{\Hh^{-\frac{1}{2}}}^2.
	\end{align*}
	Now, we bound the terms related to $\mu_0$ as follows:
	\begin{align*}
		\mu_0\langle	\,	\delta M\times H 					,	\delta \omega		\,\rangle_{\Hh^{-\frac{1}{2}}}
		\,&\lesssim\,
		\|\,\delta M\,\|_{\Hh^{\frac{1}{2}}}
		\|\, H \,\|_{L^2_x}
		\|\,\delta \omega\,\|_{\Hh^{-\frac{1}{2}}}\\
		\,&\lesssim\,
		\|\,\nabla \delta M\,\|_{\Hh^{-\frac{1}{2}}}
		\|\, H \,\|_{L^2_x}
		\|\,\delta \omega\,\|_{\Hh^{-\frac{1}{2}}}\\
		\,&\lesssim\,
		\|\, H \,\|_{L^2_x}^2
		\|\,\delta \Omega\,\|_{\Hh^{-\frac{1}{2}}}^2
		\,+\,
		\ee
		\|\,\delta \omega\,\|_{\Hh^{-\frac{1}{2}}}^2,
	\end{align*}
	together with
	\begin{align*}
		\mu_0\langle	\,	\delta M \times	\delta H				,	\delta \omega		\,\rangle_{\Hh^{-\frac{1}{2}}}\,
		&\lesssim\,	
		\|\, \delta M  \,\|_{L^2_x}
		\|\,\delta H\,\|_{\Hh^{\frac{1}{2}}}			
		\|\,\delta \omega\,\|_{\Hh^{-\frac{1}{2}}}\\
		\,&\lesssim\,
		\|\, \delta M  \,\|_{L^2}
		\|\,\nabla \delta H\,\|_{\Hh^{-\frac{1}{2}}}			
		\|\,\delta \omega\,\|_{\Hh^{-\frac{1}{2}}},
		\\
		\,&\lesssim\,
		\|\, \delta M  \,\|_{L^2}^2
		\|\,\delta \omega\,\|_{\Hh^{-\frac{1}{2}}}^2
		\,+\,
		\ee
		\|\,\nabla \delta H\,\|_{\Hh^{-\frac{1}{2}}}^2.
	\end{align*}
	Finally, the last term in \eqref{inertia} can be handled through
	\begin{align*}
		2\zeta\langle	\,	\curl\delta \uu						,	\delta \omega		\,\rangle_{\Hh^{-\frac{1}{2}}}
		\,\lesssim\,
		\|\,\nabla \delta \uu		\,\|_{\Hh^{-\frac{1}{2}}}
		\|\,\delta \omega	\,\|_{\Hh^{-\frac{1}{2}}}
		\,\lesssim\,
		\|\,\delta \omega	\,\|_{\Hh^{-\frac{1}{2}}}^2
		\,+\,
		\ee
		\|\,\nabla \delta \uu		\,\|_{\Hh^{-\frac{1}{2}}}^2.
	\end{align*}
	Thus, summarizing all previous estimates and integrating them in \eqref{inertia} we get
	\begin{equation}\label{ineq2}
	\begin{aligned}
		\frac{\dd}{\dd t}
		\rho_0 k
		\|\, \delta \omega			\,\|_{\Hh^{-\frac{1}{2}}}^2\,&+\,
		\eta'
		\|\,\nabla \delta \omega	\,\|_{\Hh^{-\frac{1}{2}}}^2\,+\,
		\lambda'
		\|\,\Div \delta \omega	\,\|_{\Hh^{-\frac{1}{2}}}^2
		\,
		\lesssim
		\Big(1+	\|\,( H ,\,\delta M )\,\|_{L^2}^2\,+\\&+\,
			\|\,\nabla(\tilde \uu ,\,\omega )\,\|_{L^2}^2 	
		\Big)\|\, \delta \omega\,\|_{\Hh^{-\frac{1}{2}}}^2
		+\ee	\|\, \nabla( \delta \uu,\, \delta \omega,\,  \delta M,  \delta H)\,\|_{\Hh^{-\frac{1}{2}}}^2.
	\end{aligned}
	\end{equation}

	\noindent
	We now take into account the difference between the magnetization equations of the two solutions, namely
	\begin{equation*}\label{magnet-eq}
	\begin{aligned}
		\;\partial_t \delta M \,+\,\uu \cdot \nabla \delta M +\,\delta \uu\cdot \nabla \delta M  \,- \sigma \Delta \delta M\,=\,\left(\,\begin{matrix}\hspace{0.25cm}M_2\\-M_1 	\end{matrix}\,\right)\delta \omega\,+\,
		\left(\,\begin{matrix}\hspace{0.25cm}\delta M_2\\- \delta M_1	\end{matrix}\,\right)\tilde \omega - \frac{1}{\tau}(\, \delta M-\chi_0 \delta H\,).
	\end{aligned}
	\end{equation*}
	We then take into account the $\Hh^{-\frac{1}{2}}$-inner product between the above identity and $\delta M$, to obtain
	\begin{equation}\label{magneto-energy}
	\begin{aligned}
		\frac{\dd}{\dd t}
		\|\,	&\delta M	\,\|_{L^2_x}^2\,+\,
		\sigma \|\,\nabla M\,\|_{L^2_x}^2
		\,=\,
		-\langle\,	\uu \cdot \nabla \delta M	,	\delta M		\,\rangle_{\Hh^{-\frac{1}{2}}}	-
		\langle\,	\delta \uu\cdot \nabla \delta M 	,	\delta M		\,\rangle_{\Hh^{-\frac{1}{2}}}	+	\\&+
		\langle\,	\delta \omega \left(\,\begin{matrix}\hspace{0.25cm}M_2\\-M_1	\end{matrix}\,\right)		,	\delta M		\,\rangle_{\Hh^{-\frac{1}{2}}}	+
		\langle\,	\tilde \omega  \left(\,\begin{matrix}\hspace{0.25cm}\delta M_2\\- \delta M_1	\end{matrix}\,\right)		,	\delta M		\,\rangle_{\Hh^{-\frac{1}{2}}}	-\frac{1}{\tau}
		\|\,	\delta M						\,\|_{\Hh^{-\frac{1}{2}}}^2	+\frac{\chi_0}{\tau}
		\langle\,	\delta H						,	\delta M		\,\rangle_{\Hh^{-\frac{1}{2}}}.	
	\end{aligned}
	\end{equation}
	We hence proceed to estimate each term on the right-hand side. First, the divergence-free condition on $\uu $ together with Lemma \ref{comm-est} imply
	\begin{align*}
		\langle\,	\uu \cdot \nabla \delta M	,	\delta M		\,\rangle_{\Hh^{-\frac{1}{2}}}
		\,&\lesssim\,
		\|\,\nabla \uu \,\|_{L^2}^2
		\|\, \delta M\,\|_{\Hh^{-\frac{1}{2}}}^2\,+\,\ee
		\|\, \nabla \delta M\,\|_{\Hh^{-\frac{1}{2}}}^2.
	\end{align*}
	The second convective term on the right-hand side of \eqref{magneto-energy} is handled by
	\begin{equation*}
		\langle\,\delta \uu\cdot \nabla \delta M 	,	\delta M		\,\rangle_{\Hh^{-\frac{1}{2}}} 
		\,\lesssim\,
		\|\,\delta \uu	\,\|_{\Hh^{\frac{1}{2}}}
		\|\,\nabla \delta M  	\,\|_{L^2}
		\|\,\delta M		\,\|_{\Hh^{-\frac{1}{2}}}
		\,\lesssim\,
		\|\,\nabla \delta M  	\,\|_{L^2}^2
		\|\,\delta M		\,\|_{\Hh^{-\frac{1}{2}}}^2
		+\ee
		\|\,\nabla \delta \uu	\,\|_{\Hh^{\frac{1}{2}}}^2.
	\end{equation*}
	Furthermore, we observe that
	\begin{align*}
		\langle\,	\delta \omega \left(\,\begin{matrix}\hspace{0.25cm}M_2\\-M_1	\end{matrix}\,\right)		,	\delta M		\,\rangle_{\Hh^{-\frac{1}{2}}}
		\,\lesssim\,
		\|\,\delta \omega	 	\,\|_{\Hh^{\frac{1}{2}}}
		\|\,\,M  				\,\|_{L^2}
		\|\, \delta M			\,\|_{\Hh^{-\frac{1}{2}}}
		\,\lesssim\,
		\|\,\,M  				\,\|_{L^2}^2
		\|\, \delta M			\,\|_{\Hh^{-\frac{1}{2}}}^2+
		\ee
		\|\,\nabla \delta \omega	 	\,\|_{\Hh^{\frac{1}{2}}}^2,
	\end{align*}
	and that
	\begin{align*}
		\langle\,	\tilde \omega  \left(\,\begin{matrix}\hspace{0.25cm}\delta M_2\\- \delta M_1	\end{matrix}\,\right)		,	\delta M		\,\rangle_{\Hh^{-\frac{1}{2}}}
		\,\lesssim\,
		\|\,	\delta M\,		\,\|_{\Hh^{\frac{1}{2}}}
		\|\,		\tilde \omega 		\,\|_{L^2}
		\|\,	\delta M			\,\|_{\Hh^{-\frac{1}{2}}}
		\,\lesssim\,
		\|\,		\tilde \omega 		\,\|_{L^2}^2
		\|\,	\delta M			\,\|_{\Hh^{-\frac{1}{2}}}^2
		+\ee
		\|\,\nabla 	\delta M\,		\,\|_{\Hh^{-\frac{1}{2}}}^2.
	\end{align*}
	Finally, the last term of \eqref{magneto-energy} is controlled as follows:
	\begin{align*}
		\langle\,	\delta H						,	\delta M		\,\rangle_{\Hh^{-\frac{1}{2}}}
		\,\lesssim\,
		\|  \,	\delta H		\,\|_{\Hh^{-	\frac{1}{2}}}^2+
		\|	\,	\delta M		\,\|_{\Hh^{-	\frac{1}{2}}}^2.
	\end{align*}
	Integrating all previous estimates together with the identity \eqref{magneto-energy}, we finally obtain
	\begin{equation}\label{ineq3}
	\begin{aligned}
		\frac{\dd}{\dd t}
		\|\,\delta M	\,\|_{L^2_x}^2\,&+\,
		\sigma \|\,\nabla M\,\|_{L^2_x}^2
		\,\lesssim\,
		\Big(1+	\|\,( M ,\,\tilde \omega )\,\|_{L^2}^2\,+\\&+\,
			\|\,\nabla( \uu ,\,\delta M )\,\|_{L^2}^2 	
		\Big)
			\|\, ( \delta M,  \delta H)\,\|_{\Hh^{-\frac{1}{2}}}^2
		+\ee	\|\, \nabla( \delta \uu,\, \delta \Omega,\,  \delta M)\,\|_{\Hh^{-\frac{1}{2}}}^2.
	\end{aligned}	
	\end{equation}

	\noindent We are now in the position to get a final bound for the $\Hh^{-1/2}$-energy and its dissipation:
	\begin{align*}
		\delta \EE_{-\frac{1}{2}}(t)\,&:=\,	\frac{\rho_0}{2}\|\, \delta \uu(t)\,\|_{\Hh^{-\frac{1}{2}}}^2\,+\,
							\frac{\rho_0 k }{2}\|\,\delta \Omega(t)\,\|_{\Hh^{-\frac{1}{2}}}^2\,+\,
							\frac{1}{2}\|\,\delta M(t) \,\|_{\Hh^{-\frac{1}{2}}}^2,\\
		\delta \DD_{-\frac{1}{2}}(t)\,&:=\,(\eta + \zeta)
		\|	\,	\nabla	\delta \uu		\,	\|_{\Hh^{-\frac{1}{2}}}^2	\, +\,
		\eta'
		\|	\,	\nabla	\delta \omega	\,	\|_{\Hh^{-\frac{1}{2}}}^2	\,+\,
		\lambda'
		\|	\,	\Div\,	\delta \omega	\,	\|_{\Hh^{-\frac{1}{2}}}^2	\,+\,
		\sigma
		\|	\,	\nabla	\delta M			\,	\|_{\Hh^{-\frac{1}{2}}}^2	
	\end{align*}
	we can integrate all the previous estimates into the following inequality:
	\begin{equation}\label{ineq-almost-finished}
	\begin{aligned}
		\frac{\dd}{\dd t}
		\delta \EE_{-\frac{1}{2}}(t)\,+\,
		\delta \DD_{-\frac{1}{2}}(t)
		\,\lesssim
		f(t)\Big(\delta \EE_{-\frac{1}{2}}(t)\,+\,\| \, \delta H(t)\,\|_{\Hh^{-\frac{1}{2}}}^2\Big)
		+\ee	\delta \DD_{-\frac{1}{2}}(t)\,+\,\ee\| \, \nabla \delta H(t)\,\|_{\Hh^{-\frac{1}{2}}}^2,
	\end{aligned}
	\end{equation}
	for a suitable function $f(t)$ in $L^1_{loc}(\RR_+)$. To finally conclude  making use of a standard Gronwall inequality, 
	we need to reformulate the terms depending on the effective magnetizing field $\delta H$ in terms of the energy $\delta \EE_{-1/2}(t)$ and its 
	dissipation $\delta \DD_{-1/2}(t)$. 
	This is achieved through the following Lemma:
	\begin{lemma}\label{lemma-ineqHM-uniq}
		The $\Hh^{-1/2}(\RR^2)$-norms of $\delta H$ and $\nabla \delta H$ are bounded by
		\begin{equation}\label{ineq1}
			\|\, \delta H \,\|_{\Hh^{-\frac{1}{2}}} \,\lesssim \,\|\, \delta M \,\|_{\Hh^{-\frac{1}{2}}}\hspace{4cm}\|\,	\nabla \delta H	\,\|_{\Hh^{-\frac{1}{2}}}\,=\, \|\,\Div \delta M\,\|_{\Hh^{-\frac{1}{2}}}.
		\end{equation}
	\end{lemma}
	\begin{proof}
		The result is a straightforward consequence of the the free-curl condition on $\delta H$: introducing the potential function $\delta \phi$ satisfying
		$
			\nabla \delta \phi\,=\,\delta H
		$, the magnetostatic equation $\Div\,\delta H\,+\,\Div\,\delta M\,=\,0$ reduces to
		$
			\Delta \delta \phi\,=\, \Div\, \delta M
		$. Hence, taking the $\Hh^{-1/2}$-inner product together with $\delta \phi$ leads to 
		\begin{align*}
			\|\,\delta H\,\|_{\Hh^{-\frac{1}{2}}}\,&=\,-\langle \,\Div\, \delta H,\,\delta \phi\,\rangle_{\Hh^{-\frac{1}{2}}}\,=\,\langle \,\Div\, \delta M,\,\delta \phi\,\rangle_{\Hh^{-\frac{1}{2}}}\\&=
			\,\langle \, \delta M,\,\delta H\,\rangle_{\Hh^{-\frac{1}{2}}}\,\lesssim\,\|\,\delta M\,\|_{\Hh^{-\frac{1}{2}}}^2\,+\,\tilde \ee \|\,\delta H\,\|_{\Hh^{-\frac{1}{2}}}^2,
		\end{align*}
		for a suitable small positive constant $\tilde \ee$. The first inequality is then achieved by absorbing the last term on the right-hand side by the left-hand side.
		
		\noindent In order to get the second inequality, it is sufficient to multiply the magnetostatic equation $\Div\,\delta H = - \Div\,\delta M$ by $\Delta \delta \phi$ in $\Hh^{-1/2}$. Hence, thanks to an integration by parts, we gather
		\begin{align*}
			\|\,\nabla \delta H\,\|_{\Hh^{-\frac{1}{2}}}
			\,&=\,-\langle\,	\delta H,\,\Delta \delta H\, \rangle_{\Hh^{-\frac{1}{2}}}
			\,=\,-\langle\, 	\Div\,\delta M,\,\Delta \delta \phi\, \rangle_{\Hh^{-\frac{1}{2}}}
			\\&=\,-\langle\, \Div\,\delta M ,\,\Div\, \delta H,\rangle_{\Hh^{-\frac{1}{2}}}
			\,=\,\|\,\Div\, \delta M\,\|_{\Hh^{-\frac{1}{2}}},
		\end{align*}
		which concludes the proof of the lemma.
	\end{proof}
	\noindent
	Thanks to Lemma \ref{lemma-ineqHM-uniq},  inequality \eqref{ineq-almost-finished} consequently reduces to
	\begin{align*}
		\frac{\dd}{\dd t}
		\delta \EE_{-\frac{1}{2}}(t)\,+\,
		\delta \DD_{-\frac{1}{2}}(t)
		\,\lesssim
		f(t)\delta \EE_{-\frac{1}{2}}(t)
		+\ee	\delta \DD_{-\frac{1}{2}}(t).
	\end{align*}
	Hence, assuming the positive parameter $\ee$ small enough, we can absorb the last term in the above inequality by the left-hand side. We thus finally deduce 
	that
	\begin{equation*}
		\frac{\dd}{\dd t}
		\delta \EE_{-\frac{1}{2}}(t)
		\,\lesssim\,
		f(t)\delta \EE_{-\frac{1}{2}}(t)
	\end{equation*}
	Then, thanks to the initial condition $\delta \EE(0)=0$, the Gronwall inequality yields $\delta \EE(t)$ to be constantly null, especially
	\begin{equation*}
		\delta \uu\,=\, \uu \,-\,\tilde \uu \,=\,0\quad \delta \omega \,=\, \omega \, -\, \tilde \omega \, =\,0
		\quad \delta M\,=\,M\,-\,\tilde M\,=\,0\quad \delta H\,=\,H \,-\,\tilde H\,=\,0,
	\end{equation*}
	and this concludes the proof of Theorem \ref{thm:uniqueness}.

	\section{Long-time dynamics} \label{sec:long-time}
	\noindent
	The purpose of this section is to establish dispersion properties for solutions $(\uu,\,\omega,\,M,\,H)$ to the non-linear system \eqref{main_system} and to prove Theorem \ref{thm:long-time-dyn}. For the sake of a unified presentation we recall here the statement of Theorem \ref{thm:long-time-dyn};

	\begin{theorem}\label{thm:long-time-dyn2}
		Let $(\uu_0,\,\omega_0,\,M_0,\,H_0)$ be in $L^2(\RR^2)$ and $F$ in $W^{1,2}_{loc}(\RR_+, \Hh^{-1}(\RR^2))$.
		Denote by $(\uu,\,\omega,\,M,\,H)$ the unique weak solution to \eqref{main_system}  
		given by Theorem \ref{thm:uniqueness}.	Suppose that there exists a positive constants $K$ and an 
		exponent $\eta\in ]0,1[$ such that, for almost every $t>0$ one has
		\begin{equation}\label{eq:long_time_integrability_F}
			\|\,F(t)\,\|_{L^2(\RR^2)}^2
			\,+\,
			\|\,F(t)\,\|_{\Hh^{-1}(\RR^2)}^2
			+\,
			\|\,\partial_t F(t)\,\|_{H^{-1}(\RR^2)}^2\,\leq\,\frac{K}{(1+t)^{1+\eta}}
		\end{equation}
		Then, for any $\alpha<\eta$ there exists a constant $C_\alpha$ such that the following decay property is satisfied
		\begin{equation*}
			\|\,(\uu(t),\,\omega(t),\,M(t),\,H(t))\,\|_{L^2(\RR^2)}\leq\,\frac{C}{(1+t)^{\alpha}}.
		\end{equation*}
	\end{theorem}

	\noindent
	The first preliminary result which is needed in our study is the following pointwise estimate for the effective magnetizing fields.
	\begin{lemma}\label{lemma-ineq-Hxi} Assume $M$ in $L^2_{loc}(\RR_+, L^2(\RR^2))$ and the source term $F$ in $L^2(\RR_+, \Hh^{-1}(\RR^2))$. 
	Then the solution $H = H(t,x)$ of 
	magnetostatic equation
	\begin{equation*}
		\Div\,(\,H\,+\,M\,)\,=\,F,\quad\quad\text{with}\quad\quad \curl\, H\,=\,0,
	\end{equation*}
	satisfies
	\begin{equation}\label{ineq-Hxi}
			|\,\hat H\pare{\xi, \ t}\,|^2\,\leq \,2 |\,\hat M\pare{\xi, \ t}\,|^2\,+\,|\xi|^{-2}|\,\hat F\pare{\xi, \ t}\,|^2,
	\end{equation}
	for any time $t>0$ and $\xi\in\RR^2$.
	\end{lemma}
	\begin{proof}	The $\curl$ of $H$ being null, we can reformulate the equation on $H$ in terms of the potential $\phi$ as follows:
		\begin{equation*}
			H\,=\,\nabla \phi\quad\Rightarrow\quad -\Delta \phi\,=\,\Div\,M\,-\,F\quad\Rightarrow\quad -\Delta H\,=\,\nabla \Div\, M \,-\, \nabla F.
		\end{equation*}
	Then, passing in Fourier variables, we recast the latter equation in the following form:
		\begin{equation*}
			H\pare{\xi, \ t}\,=\,-\frac{\xi\otimes \xi}{|\xi|^2}\hat{M}\pare{\xi, \ t} + |\xi|^{-1}\frac{i\xi}{|\xi|}\hat{F}\pare{\xi, \ t}
		\end{equation*}
		the desired inequality \eqref{ineq-Hxi} then follows multiplying both sides by $H\pare{\xi, \ t}$ and applying a Cauchy-Schwarz inequality.
	\end{proof}
	\noindent
	The second preliminary result which is needed in our study is the following one:
	\begin{prop}\label{prop-pointwise-Fourier-estimate}
		Under the hypotheses of Theorem \ref{thm:long-time-dyn2}, the solution $(\uu,\,\omega,\,M,\,H)$ satisfies the following pointwise estimate:
		\begin{align*}
			|\,\hat{\uu}\pare{\xi, \ t}\,|^2\,&+\,|\,\hat{\omega}\pare{\xi, \ t}\,|^2\,+\,|\,\hat{M}\pare{\xi, \ t}\,|^2\,+|\,\hat{H}\pare{\xi, \ t}\,|^2\,
			 \,
			\,\leq\\
			&\leq \,
			|\,\hat{\uu}_0(\xi)\,|^2\,+\,|\,\hat{\omega}_0(\xi)\,|^2\,+\,|\,\hat{M}_0(\xi)\,|^2\,
			\,+\,
			C 
			\,+\,
			\int_0^t|\,\xi|^{-2}|\,\hat{F}\pare{\xi, \ t'} |^2 \dd t',
		\end{align*}
		for almost every $t>0$ and $\xi\in \RR^2$, and for a suitable positive constant $C$ depending on the initial data $(\uu_0,\,\omega_0,\,H_0,\,M_0)$ as well as on the 
		parameter $\rho_0,\,k,\,t',\,\eta,\,\eta',\,\chi_0,\,\mu_0$.
	\end{prop}
	\begin{proof}
		Passing in Fourier variables, we recast the momentum equation in the following form:
		\begin{equation*}
			\;\rho_ 0\partial_t \hat{\uu},\,+\,(\eta +\zeta)\,|\xi|^2\hat{ \uu}\,+\,i\xi \hat{\pre}\,= -\FF(\uu\cdot\nabla \uu) \,+\,\mu_0\,\FF( M\cdot \nabla H)\,
												+\,2\zeta\, \left(\,\begin{matrix}\hspace{0.25cm}i\xi_2\hat\omega\\-i\xi_1 \hat \omega	\end{matrix}\,\right)
		\end{equation*}
		Multiplying by $\hat{\uu}$, integrating in time and applying a Cauchy-Schwarz inequality, we get
		\begin{equation}\label{ineq-hatu2}
		\begin{aligned}
			\;\rho_ 0|\,\hat \uu\pare{\xi, \ t}\,|^2&+\,\Big(\frac{\eta}{2}\, +\,\zeta\,\Big)\,|\xi|^2\int_0^t|\,\hat \uu\pare{\xi, \ t'}\,|^2\,\dd\tau\leq
			\frac{\rho_0}{2}|\,\hat \uu_0(\xi)\,|^2\,+\\&+\,
			\frac{2}{\eta}\int_0^t|\,\FF(\uu\cdot\nabla \uu)\pare{\xi, \ t'}\,|^2\dd\tau\,	+													
			\, \frac{4\mu_0^2}{\eta}\int_0^t\,|\,\FF( M\cdot \nabla H)\pare{\xi, \ t'}\,|^2\dd\tau
			\,+\\&\hspace{6cm}+
			2\zeta|\,\xi\,|\int_0^t |\,\hat{\omega}\pare{\xi, \ t'}||\,\hat \uu\pare{\xi, \ t'}\,|^2\,\dd t'.
		\end{aligned}
		\end{equation}
		The standard energy space for weak solutions allows to estimate the non-linear terms on the right-hand:
		\begin{equation*}
			\frac{2}{\eta}\int_0^t	|\,\FF_x(\uu\cdot\nabla \uu)\pare{\xi, \ t'}\,|^2\dd t'
			\,\lesssim\,
			\int_0^t	\|\,\uu\pare{ t'}\cdot\nabla \uu\pare{ t'}\,\|_{L^1(\RR^2)}^2
			\,\leq\,
			\|\,\uu\,\|_{L^\infty_t L^2}^2
			\int_0^t	
			\|\,\nabla \uu\,\|_{L^2}^2
			\,\leq
			C,
		\end{equation*}
		together with
		\begin{equation*}
			\, 
			\frac{4\mu_0^2}{\eta}\,
			\int_0^t	
			|\,\FF_x( M\cdot \nabla H)\pare{\xi, \ t'}\,|^2\dd t'
			\,\leq\,
			\int_0^t	
			\|\,M\pare{ t'}\cdot \nabla H\pare{ t'}\,\|_{L^1}^2
			\,\leq\,
			\|\,M\,\|_{L^\infty_t L^2}^2\int_0^t	\|\,\nabla H\pare{ t'}\,\|_{L^2}^2\dd t'
			\,\leq\,C.
			\end{equation*}
			Plugging the last inequalities in \eqref{ineq-hatu2} eventually yields to the bound
			\begin{equation}\label{ineq-hatu2b}
			\begin{aligned}
						\frac{\rho_0}{2}|\,\hat \uu\pare{\xi, \ t}\,|^2\,&+\,\Big(\,\frac{\eta}{2}\, +\,\zeta\Big)\,|\xi|^2\,\int_0^t|\,\hat \uu\pare{\xi, \ t'}\,|^2\dd t'\,
						\leq\\&\leq
						C\Big(\,
						\frac{\rho_0}{2}|\,\hat \uu_0(\xi)\,|^2\,+\,
						\frac{\rho_0}{2}|\,\hat \uu_0(\xi)\,|^2
						\,\Big)
						\,+\,
			2\zeta |\,\xi\,|\int_0^t|\,\hat{\omega}\pare{\xi, \ t'}\,||\,\hat \uu\pare{\xi, \ t'}\,|\dd t' + C.
			\end{aligned}
			\end{equation}
			Similarly, we multiply the equation for the angular velocity 
			\begin{equation*}
				\;\rho_0 k\big( \partial_t \omega\,+\uu\cdot \nabla \omega\,\big) - \eta'\Delta \omega\,=\,\mu_0\, M\times H\,+\,2\zeta (\,\curl \uu \,-\,2\omega\,)	
			\end{equation*}
			by $\omega$, we integrate in time and we apply a standard Cauchy-Schwartz inequality, to get
			\begin{equation}\label{ineq-anghat1}
			\begin{aligned}
				\frac{\rho_0 k}{2}  |\,\hat{\omega}\pare{\xi, \ t}\,|^2+\,\Big(\,|\,\xi\,|^2\frac{\eta'}{2}+2\zeta\,\Big)\int_0^t |\,\hat{\omega}\pare{\xi, \ t'}\,|^2\dd t'\leq
				\,\frac{4\mu_0^2}{\eta'}\int_0^t|\,\FF_x( M\times H)\pare{\xi, \ t'}\,|^2\dd t'\,
				\,+\\\,+
				\int_0^t	|\,\uu\cdot \nabla \omega\pare{\xi, \ t'}\,|^2\dd t'
				\,+\,
				2\zeta|\,\xi\,|\int_0^t|\,\hat{\omega}\pare{\xi, \ t'}\,||\,\uu\pare{\xi, \ t'}\,|\dd t'.
			\end{aligned}
			\end{equation}
			We then handle the non-linear terms on the right-hand side through the relations
			\begin{equation*}
			\begin{aligned}
			\frac{2}{\eta}\int_0^t	|\,\FF_x(\uu\cdot\nabla \uu)\pare{\xi, \ t'}\,|^2\dd t'
			\,&\lesssim\,
			\int_0^t	\|\,\uu\pare{ t'}\cdot\nabla \uu\pare{ t'}\,\|_{L^1(\RR^2)}^2\\
			\,&\leq\,
			\|\,\uu\,\|_{L^\infty_t L^2}^2
			\int_0^t	
			\|\,\nabla \uu\,\|_{L^2}^2
			\,\leq
			C,\\
			\, 
			\frac{4\mu_0^2}{\eta'}\,
			\int_0^t	
			|\,\FF_x( M \times H)\pare{\xi, \ t'}\,|\dd t'
			\,&\leq\,
			\int_0^t	
			\|\,M\pare{ t'} \times H\pare{ t'}\,\|_{L^1}^2\dd t'\\
			\,&\leq\,
			\|\,H\,\|_{L^\infty_t L^2}^2\int_0^t	\|\,M\pare{ t'}\,\|_{L^2}^2\dd t'
			\,\leq\,C.	
			\end{aligned}
			\end{equation*}
			Hence, plugging the above estimates into \eqref{ineq-anghat1} leads to
			\begin{equation}\label{ineq-anghat2}
			\begin{aligned}
						\frac{\rho_0 k}{2}  |\,\hat{\omega}\pare{\xi, \ t}\,|^2\,+\,\Big(\,|\,\xi\,|^2\frac{\eta'}{2}&\,+\,2\zeta\,\Big)\int_0^t |\,\hat{\omega}\pare{\xi, \ t'}\,|^2\dd t'
						\,\leq\\\,&\leq\,
						\frac{\rho_0 k}{2}  |\,\hat{\omega}_0(\xi)\,|^2
						\,+\,
						C\,+\,
						2\zeta|\,\xi\,|\int_0^t|\,\hat{\omega}\pare{\xi, \ t'}\,||\,\uu\pare{\xi, \ t'}\,|\dd t'.
			\end{aligned}
			\end{equation}
			We now take the sum between \eqref{ineq-hatu2b} and \eqref{ineq-anghat2}
			\begin{equation}\label{ineq-hatw-hatu}
			\begin{aligned}
					\frac{\rho_0 }{2}
					\Big(\,  
					k|\,\hat{\omega}\pare{\xi, \ t}\,|^2\,+\,
					|\,\hat \uu\pare{\xi, \ t}\,|^2
					\Big)\,+\,					
					\,|\,\xi\,|^2\int_0^t
					\Big(\,
					&\frac{\eta}{2}|\,\hat \uu\pare{\xi, \ t'}\,|^2\,+\,
					\frac{\eta'}{2}|\,\hat{\omega}\pare{\xi, \ t'}\,|^2
					\,\Big)
					\,\dd\tau
					\,\leq\\&\leq
					\frac{\rho_0 }{2}
					\Big(
					|\,\hat \uu_0(\xi)\,|^2\,+\,
					k|\,\hat{\omega}_0(\xi)\,|^2\,					
					\Big)
					\,+\,
					C,
			\end{aligned}
			\end{equation}
			where we have also used the positive sign of the non-linear term depending on $\zeta$:
			\begin{equation*}
				2\zeta	\int_0^t\left( |\,\hat{\omega}\pare{\xi, \ t'}\,\,|^2\,+\,|\xi|^2|\,\uu\pare{\xi, \ t'}\,|^2\,-\,2|\xi||\,\hat{\omega}\pare{\xi, \ t'}\,\,||\,\uu\pare{\xi, \ t'}\,|\right)\dd t'\,\geq\,0,
			\end{equation*}
			To conclude, we recast the equation of the magnetization $M$ in Fourier space
			\begin{equation*}
				\partial_t \hat{M} \,+\,\FF_x(\uu\cdot \nabla M) + \sigma |\xi|^2 \hat{M}\,=\,
				\FF_x\left(\left(\,\begin{matrix}\hspace{0.25cm}M_2\\-M_1	\end{matrix}\,\right)\omega\right)- \frac{1}{t'}(\,\hat{M}-\chi_0\hat{H}\,),
			\end{equation*}
			we then multiply by $\hat M$ and integrate in time $(0,t)$ to obtain
			\begin{equation}\label{ineq-Mxi}
			\begin{aligned}
				\frac{1}{2}|\,\hat{M}\pare{\xi, \ t}\,|^2\,&+\,\Big(\,\frac{\sigma}{2}|\,\xi\,|^2\,+\,\frac{1}{\tau}\,\Big)\int_0^t|\,\hat{M}\pare{\xi, \ t'}\,|^2\,\dd t'
				\,\lesssim\,
				\frac{1}{2}|\,\hat{M}_0(\xi)\,|^2\,+\\&+\,
				\frac{2}{\sigma}\int_0^t|\,\FF_x(\uu\cdot \nabla M)\,|^2\pare{\xi, \ t'}\,\dd t'\,+
				\frac{2}{\sigma}\int_0^t\,
				\,\left|\,\FF_x\left(\,\begin{matrix}\hspace{0.25cm}M_2\\-M_1	\end{matrix}\,\right)\omega\pare{\xi, \ t'}\right|^2\,\dd t'
				\,+\\&\hspace{6.5cm}+\,
				\frac{\chi_0}{\tau}
				\int_0^t
				|\,\hat{H}\pare{\xi, \ t'}\,||\,\hat{M}\pare{\xi, \ t}\,|\dd t'.
			\end{aligned}
			\end{equation}
			Now, we pass to estimate the last terms in the previous inequality, first by 
			\begin{equation*}
			\begin{aligned}
				\frac{2}{\sigma}\int_0^t|\,\FF_x(\uu\cdot \nabla M)\,|^2\pare{\xi, \ t'}\,\dd t'
				\,&\leq\,
				\int_0^t\|\,\uu\pare{ t'}\cdot \nabla M\pare{ t'}\|_{L^1(\RR^2)}^2\dd t' \\
				&\leq\,
				\|\,\uu\pare{ t'}\,\|_{L^\infty_t L^2}^2
				\int_0^t\|\nabla M\pare{ t'}\,\|_{L^2}^2
				\,\leq\,C,
			\end{aligned}
			\end{equation*}
			and then by
			\begin{equation*}
			\begin{aligned}		
			\\
				\frac{2}{\sigma}\int_0^t\,
				\,\FF_x\left(\left(\,\begin{matrix}\hspace{0.25cm}M_2\\-M_1	\end{matrix}\,\right)\omega\right)\pare{\xi, \ t'}\,\dd t'
				&\lesssim
				\,
				\int_0^t\|\,
				\left(\,	
					\begin{matrix}	
						\hspace{0.25cm}M_2\\-M_1	
					\end{matrix}\,
					\omega								
				\right)\|_{L^1(\RR^2)}^2\dd t' \\
				&\lesssim
				\| M\pare{ t'}\,\|_{L^\infty_tL^2}^2	
				\int_0^t\| \omega\pare{ t'}\,\|_{L^2}^2
				\lesssim
				C.
			\end{aligned}
			\end{equation*}
			Now, we resort  Lemma \ref{lemma-ineq-Hxi} in order to bound the last term which depends on the effective magnetization field $H$:
			\begin{equation*}
				\frac{\chi_0}{\tau}\int_0^t
				|\,\hat{H}\pare{\xi, \ t'}\,||\,\hat{M}\pare{\xi, \ t}\,|\dd t'
				\,\leq\,
				\frac{(1+\nu)\chi_0}{\tau}
				\int_0^t
				|\,\hat{M}\pare{\xi, \ t'}\,|^2\dd t'\,+\,
				\frac{1}{\nu\tau}\int_0^t
				|\,\xi\,|^{-2}|\,F\pare{\xi, \ t'}\,|^2\dd t'.				
			\end{equation*}
			Putting these relations together with \eqref{ineq-Mxi}, we deduce the final estimate on $M$
			\begin{equation}\label{ineq-hatM}
			\begin{aligned}
				\frac{1}{2}|\,\hat{M}\pare{\xi, \ t}\,|^2\,+\,\Big(\,\frac{\sigma}{2}|\,\xi\,|^2\,&+\,\frac{1-(1+\nu)\chi_0}{\tau}\,\Big)\int_0^t|\,M\pare{\xi, \ t'}\,|^2\,\dd t'
				\,\leq\\&\leq\,
				\frac{1}{2}|\,\hat{M}_0(\xi)\,|^2\,+\,
				C
				\,+\,
				\frac{1}{\nu\tau}\int_0^t
				|\,\xi\,|^{-2}|\,F\pare{\xi, \ t'}\,|^2\dd t'.
			\end{aligned}
			\end{equation}
			The proof of the proposition is then accomplished combining inequalities \eqref{ineq-hatu2b}, \eqref{ineq-anghat2} and \eqref{ineq-hatM}. 
	\end{proof}
	\noindent We are now in the position of proving Theorem \ref{thm:long-time-dyn2}.
	\begin{proof}[Proof of Theorem \ref{thm:long-time-dyn2}] 
		We begin with recalling the a-priori energy bound of system \eqref{main_system}:
		\begin{equation*}
			\frac{1}{2}	\frac{\dd}{\dd t}\EE(t) + \tilde{c} \DD(t)
			\leq 
			C 
			\Big(\,
					\|\,F(t)\,\|_{L^2(\RR^2)}^2\,+\,
					\|\,F(t)\,\|_{\Hh^{-1}(\RR^2)}^2+\,
					\|\,\partial_tF(t)\,\|_{\Hh^{-1}(\RR^2)}^2
			\Big),
		\end{equation*}				
		for a suitable positive constant $\tilde c>0$, 
		where the energy $\EE(t)$ and the dissipation $\DD(t)$ are determined by
		\begin{equation*}
		\begin{aligned}
			\EE(t)\,&=\,\rho_0\,\|\,\uu\,\|_{L^2_x}\,+\,\rho_0k\|\,\omega\,\|_{L^2_x}^2\,+\,\mu_0\,\|\,H(t)\,\|_{L^2_x}^2\,+\,\|\,M\,\|_{L^2_x}^2,\\
			\DD(t)\,&=\,\,\|\,\nabla \uu\,\|_{L^2_x}\,+\,\|\,\nabla \omega\,\|_{L^2_x}^2\,+\,\|\,\nabla M\,\|_{L^2_x}^2\,+\,\|\,\Div\, M\,\|_{L^2_x}^2\,+\,\|\, M\,\|_{L^2_x}^2+\,\|\,H\,\|_{L^2}^2.
		\end{aligned}
		\end{equation*}
		Making use of the Fourier-Plancherel theorem, we recast the $L^2$-energy estimate  for system \eqref{main_system} in Fourier variables:
		\begin{equation*}
			\begin{aligned}
				\frac{\dd}{\dd t}
				\int_{\RR^2}
				\Big[\,
					\underbrace{
					\frac{\rho_0}{2}\,|\,\hat{\uu}\pare{\xi, \ t}\,|^2\,+\,
					\frac{\rho_0k}{2}\,|\,\hat{\omega}\pare{\xi, \ t}\,|^2\,+\,	
					\frac{1}{2}\,|\,\hat{M}\pare{\xi, \ t}\,|^2
					}_{r\pare{\xi, \ t}}
					\,
					+\,
					|\,\hat{H}\pare{\xi, \ t}\,|^2
				\Big]
				\dd \xi\,+
				\\
				\,+\,
				\tilde c \int_{\RR^2}
				|\xi|^2
				\underbrace{
				\Big[
					\,|\,\hat{\uu}\pare{\xi, \ t}\,|^2
					\,+\,
					\,|\,\hat{\omega}\pare{\xi, \ t}\,|^2
					\,+\,
					|\,M\pare{\xi, \ t}\,|^2
				\Big]
				}_{\hat{R}\pare{\xi, \ t}}
				\dd\xi
				\,\leq\\\leq\,
				\int_{\RR^2}
				\Big(\,
					|\,\hat F\pare{\xi, \ t}\,|^2\,+\,
					|\,\xi\,|^{-2}
					|\,\hat F\pare{\xi, \ t}\,|^2
					\,+\,
					|\,\xi\,|^2
					|\,\partial_t\hat F\pare{\xi, \ t}\,|^2
				\Big)
				\dd\xi
				\leq
				\frac{K}{(1+t)^{1+\eta}}.
			\end{aligned}
		\end{equation*}
				
		\noindent
		Next, for $t\geq 0$ we define the ball $B_t(0)\subseteq \RR^2$ as the ball centered in $0$ and of radius $\nu(t)$, for a function $\nu(t)$ to be determined later. Then we can write
		\begin{equation*}
			\begin{aligned}
				&\int_{\RR^2}
				|\xi|^2
				R\pare{\xi, \ t}
				\dd\xi
				\geq\,
				\int_{\RR^2\setminus B_t(0)}
				\av{\xi}^2
				R\pare{\xi, \ t}
				\dd\xi
				\geq\,
				\nu(t)^2
				\left(
				\int_{\RR^2}
				R\pare{\xi, \ t}
				\dd\xi
				-
				\int_{B_t(0)}
				|\xi|^2
				R\pare{\xi, \ t}
				\dd\xi
				\right);
			\end{aligned}
		\end{equation*}
		We summarize the previous inequalities in the following estimate
		\begin{equation}\label{ineq111}
				\frac{\dd}{\dd t}
				\int_{\RR^2}\,r\pare{\xi, \ t}\,\dd\xi
				\,+\,
				\nu(t)^2\int_{\RR^2}\,r\pare{\xi, \ t}\,\dd\xi
				\,\lesssim\,
				\frac{1}{(1+t)^{1+\eta}}
				\,+\,
				\nu(t)^2\int_{B_t(0)}\,r\pare{\xi, \ t}\,\dd\xi.
		\end{equation}
		where we have also used the relation $r\pare{\xi, \ t}\,\lesssim R\pare{\xi, \ t}\,\lesssim r\pare{\xi, \ t}$.
		At this point, we notice that, after setting $\rm{V}(t)\,:=\,\int_0^t \nu\pare{ t'}^2\dd t'$, we can write
		\begin{equation*}
			\e^{-\rm{V}(t)}
			\frac{\dd}{\dd t}
			\left[	
				\e^{\rm{V}(t)}
				\int_{\RR^2}r\pare{\xi, \ t} \dd\xi
			\right]
			\,=\,
			\frac{\dd}{\dd t}
				\int_{\RR^2}\,r\pare{\xi, \ t}\,\dd\xi
				\,+\,
				\nu(t)^2\int_{\RR^2}\,r\pare{\xi, \ t}\,\dd\xi
		\end{equation*}
		Putting this relation in \eqref{ineq111} we gather
		\begin{equation}\label{ineq-112}
			\e^{-\rm{V}(t)}
			\frac{\dd}{\dd t}
			\left[	
				\e^{\rm{V}(t)}
				\int_{\RR^2}r\pare{\xi, \ t} \dd\xi
			\right]
			\,+\,
			\int_{\RR^2}
			|\,H\pare{\xi, \ t}\,|^2
			\dd\xi +
			\,\lesssim\,
			\frac{1}{(1+t)^{1+\eta}}
			\,+\,
				\nu(t)^2\int_{B_t(0)}\,r\pare{\xi, \ t}\,\dd\xi.
		\end{equation}
		Now, we pass to estimate the last term in the previous inequality: Proposition \ref{prop-pointwise-Fourier-estimate} implies
		\begin{equation*}
			\begin{aligned}
				\nu(t)^2\int_{B_t(0)}\,r\pare{\xi, \ t}\,\dd\xi
				\,=\,
				\nu(t)^2\int_{B_t(0)}\,
				\left[
				\frac{\rho_0}{2}\,|\,\hat{\uu}\pare{\xi, \ t}\,|^2\,+\,
					\frac{\rho_0k}{2}\,|\,\hat{\omega}\pare{\xi, \ t}\,|^2\,+\,	
					\frac{1}{2}\,|\,\hat{M}\pare{\xi, \ t}\,|^2
				\right]\,\dd\xi\\
				\lesssim 
				\nu(t)^2
				\Big(
						1\,+\,\|\,(\,\uu_0,\,\omega_0,\,M_0,\,H_0)\,\|_{L^2_x}^2
				\,+\,
				\int_{B_t(0)}
				\left(
				\int_0^t
					|\,\xi\,|^{-2}
					|\hat{F}\pare{\xi, \ t'}|^{2}
					\dd t'
				\right)
				\dd\xi
				\Big)
				\lesssim
				\nu(t)^2.
			\end{aligned}
		\end{equation*}
		Hence we can insert this bound into \eqref{ineq-112} and integrate the resulting expression in time:
		\begin{equation*}
			\e^{\rm{V}(t)}
				\int_{\RR^2}r\pare{\xi, \ t} \dd\xi\,
			\,+\,
			\int_0^t
			\int_{\RR^2}
			\e^{\rm{V}\pare{ t'}}
			|\,\hat{H}\pare{\xi, \ t}\,|^2
			\dd\xi
			\,\lesssim\,
			\int_0^t \e^{\rm{V}\pare{ t'}}
			\frac{1}{(1+t')^{1+\eta}}\dd t'
			\,+\,
			\int_0^t \e^{\rm{V}\pare{ t'}}
			\nu(t)^2\dd t'.
		\end{equation*}
		To conclude the proof we choose the function $\nu^2(t)\,=\,\alpha/(1+t)$ for a positive constant $\alpha<\eta$.
		After observing that
		\begin{equation*}
		\int_0^t 
			\frac{
			\e^{\rm{V}\pare{ t'}}}
			{(1+t')^{1+\eta}}
			\dd t'
			\,=\,
			1-\frac{1}{(1+t)^{\eta-\alpha}}
					\quad\quad\text{and}\quad\quad
			\int_0^t \e^{\rm{V}\pare{ t'}}
			\nu(t)^2\dd t'\,=\,1- \frac{\alpha^2}{(1+t)^{1-\alpha}}
		\end{equation*}
		we finally discover that 
		\begin{equation*}
			\int_{\RR^2}
			\left[
			\frac{\rho_0}{2}\,|\,\hat{\uu}\pare{\xi, \ t}\,|^2+
					\frac{\rho_0k}{2}\,|\,\hat{\omega}\pare{\xi, \ t}\,|^2+	
					\frac{1}{2}\,|\,\hat{M}\pare{\xi, \ t}\,|^2
			\right]
			\dd\xi
			=
			\int_{\RR^2}r\pare{\xi, \ t} \dd\xi
			\lesssim
			\ee^{-{\rm V}(t)}\,=\,
			\frac{1}{(1+t)^\alpha}.
		\end{equation*}
		The proof of the Proposition is then accomplished with the further remark that
		\begin{equation*}
			\int_{\RR^2}|\,\hat{H}\pare{\xi, \ t}\,|^2\dd\xi
			\,\leq\,
			\int_{\RR^2}|\,M\pare{\xi, \ t}\,|^2\dd \xi
			\,+\,
			\|\,F(t)\,\|_{\Hh^{-1}(\RR^2)}
			\,\lesssim\,\frac{1}{(1+t)^\alpha}.
		\end{equation*}	
	\end{proof}

\section{Propagation of Sobolev regularities}\label{sec:higher_regularity}

\noindent
In the previous sections we have established the theory of weak solutions for system \eqref{main_system}, under minimum smoothness assumption on the initial data and the the external magnetic force $F$. The purpose of this section is to investigate solutions with higher regularities. Our goal is to prove propagation of the initial smoothness, i.e. to prove Theorem \ref{thm-prop-reg}.\\

\noindent For the sake of a unified presentation we recall here the main statement;

\begin{theorem}\label{thm-prop-reg2}
	Let us assume the initial data ${\rm U}_0\,:=\,(\uu_0,\,\omega_0,\,M_0,\,H_0)$ is in the non-homogeneous space $H^s(\RR^2)$, for a positive real $s>0$. Assuming the source term $F$ in 
	$
	L^2_{loc}(\RR_+, H^s(\RR^2))\cap W^{1,2}_{loc}(\RR_+,H^{s-1}(\RR^2)).
	$ 
	Let $(\uu,\,\omega,\,M,\,H)$ be the unique solution of system \eqref{main_system} given by Theorem \ref{thm:uniqueness}. Then
	\begin{equation*}
		(\uu,\,\omega,\,M,\,H)\,\in\,L^\infty_{\loc}(\RR_+,H^s(\RR^2))\cap L^2_{\loc}(\RR_+, H^{s+1}(\RR^2))
	\end{equation*}
	and the following dissipation formula holds true:
	\begin{equation*}
		\frac{1}{2}\EE_s(t)\,+\,\int_0^t\DD_s\pare{ t'}\dd t'\,\leq\, \Psi_s(\,{\rm U}_0,\,F,\,t),
	\end{equation*}
	where
	\begin{equation*}
		\Psi_s(\,{\rm U}_0,\, F,\,t)
		\,=\,
		\bigg(
		\frac{1}{2}\EE_s(0)
		\,+\,
		C
		\int_0^t
		\Big[
			\|\,				F\pare{ t'}\,\|_{\Hh^s}^2\,+\,
			\|\,				F\pare{ t'}\,\|_{\Hh^{s-1}}^2\,+\,
			\|\,	\partial_t	F\pare{ t'}\,\|_{\Hh^{s-1}}^2\,
		\Big]
		\dd t'
		\bigg),
	\end{equation*}
	with $C$, a suitable positive constant which depends on.
\end{theorem}

\noindent
\textit{Proof of Theorem \ref{thm-prop-reg2}:} 
For the sake of unified presentation, it is convenient to first introduce some terminology. We recall that the definition of the homogeneous paraproduct operator $\Th_a b$ and the homogeneous reminder $\dot R(a,\,b)$ correspond to
\begin{equation*}
	\Th_a b \,=\,\sum_{j\in\ZZ}\Sd_{j-1}a \,\Dd_j b\quad\quad\text{and}\quad \quad \Rd(a,\,b)\,=\,\sum_{j\in\ZZ,\;|\nu|\leq 1}\Dd_{j}a\, \Dd_{j+\nu}b,
	\quad\text{where}\quad \Sd_{j}a\,=\,\sum_{q\leq j-1}\Dd_q a.
\end{equation*} 
respectively. Introducing the notation $T'_ab\,=\,\Th_ab\,+\,\dot R(a,b)\,=\,\sum_{j\in\ZZ}\Dd_j a\,\Sd_{j+1}b$, we recall that the identity $a\,b\,=\, \Th_a b\,+\,\Th'_b a$ holds true for all homogeneous tempered distributions $a$ and $b$ for which the product is well-defined. By an abuse of notation we also introduce the terminology
\begin{equation*}
	\Th_A\cdot \nabla B\,=\,\sum_{j\in\ZZ}\Sd_{j-1}A\cdot \nabla \Dd_jB\quad\quad\text{and}\quad\quad
	\Th_{\nabla B}'A\,=\,\sum_{j\in\ZZ}\Dd_{j}A\cdot \nabla \Sd_{j+1}B,
\end{equation*}
for any vector field $A$ and $B$ returning value in $\RR^2$.

\noindent 
To begin with, we make a paralinearization of system \eqref{main_system}. We first fix an index $j\in \ZZ$ and next we apply the homogeneous dyadic block $\Dd_j$ to the $\uu$-equation of system \eqref{main_system}. We hence remark that the function $\uu_j:=\,\Dd_j\uu$ is a classical solution of the following PDE:
\begin{equation*}
\begin{aligned}
	\vspace{3pt}
	\;\rho_ 0\big(\partial_t \uu_j\,+\,\Sd_{j-1}\uu\cdot\nabla \uu_j\,
	\,\big)\,-\,(\eta +\zeta)\,\Delta \uu_j\,+\,\nabla \pre_j\,= 
	\,-\rho_0
	(\Th_{\uu}\cdot\nabla \uu_j\,-\,\Sd_{j-1}\uu\cdot \nabla \uu_j\,)\,+\\
	-\rho_0\Dd_j(\,\Th'_{\nabla \uu} \uu\,)\,+
	\mu_0\Dd_j(\,M\cdot \nabla H\,)\,+\,
	2\zeta\, \left(\,\begin{matrix}\hspace{0.25cm}\partial_2  \omega_j\\-\partial_1 \omega_j	\end{matrix}\,\right).
\end{aligned}
\end{equation*}
Next, we perform an energy estimate multiplying both sides by $\uu_j$ and integrating over $\RR^2$. We then achieve
\begin{equation}\label{reg_prop_equj1}
	\begin{aligned}
	\frac{\rho_0}{2}\frac{\dd}{\dd t}
	\|\,	\uu_j	\,\|_{L^2}^2
	\,+\,
	(\eta\,+\,\zeta\,)
	\|\,\nabla \uu_j\,\|_{L^2}^2
	\,=\,
	\rho_0 
	\underbrace{
	\int_{\RR^2}(\,	\Th_\uu\cdot\nabla \uu_j-\dot S_{j-1} \uu\cdot\nabla \uu_j\,)\cdot \uu_j
	}_{\I_1^j(\uu(t),\,\uu(t))}
	\,+\\
	\,+\,
	\rho_0
	\underbrace{
	\int_{\RR^2}\Dd_j(\,\Th'_{\nabla \uu} \uu\,)\,\cdot \uu_j}_{\I_2^j(\uu(t),\,\uu(t))}
	\,+
	\mu_0\int_{\RR^2}\Dd_j(\,\Th_{M}\cdot \nabla H\,)\,\cdot \uu_j
	-
	2\zeta
	\int_{\RR^2}
	\omega_j\curl \uu_j,
	\end{aligned}
\end{equation}
where the last identity holds true also thanks to the free divergence condition on $\uu$, leading to
\begin{equation*}
	\int_{\RR^2}(\Sd_{j-1} \uu\cdot\nabla \uu_j)\cdot \uu_j\,=\,0.
\end{equation*}
We also introduce the operators
\begin{equation*}
\begin{aligned}
	\I_1^j({\rm v},\,{\rm w})\,&:=\,\int_{\RR^2}(\,	\Th_{\rm v}\cdot\nabla {\rm w}_j-\dot S_{j-1} {\rm v}\cdot\nabla {\rm w}_j\,)\cdot {\rm w}_j,\\
	\I_2^j({\rm v},\,{\rm w})\,&:=\,\int_{\RR^2}\Dd_j(\,\Th'_{\nabla {\rm w}} {\rm v}\,)\,\cdot {\rm w}_j,
\end{aligned}
\end{equation*}
for two suitable vector field ${\rm{ v}}(x)\in\RR^2$ and ${\rm w(x)}\in \RR^2$. Furthermore, the above definitions easily extend to the case of a real value function ${\rm w(x)}\in \RR$.

\noindent
The nonlinear terms $\I_1^j(\uu(t),\,\uu(t))$ and $\I_2^j(\uu(t),\,\uu(t))$ of \eqref{reg_prop_equj1} are then bounded making use of the following lemma, :
\begin{lemma}\label{lemma:hig_reg}
Let ${\rm v}(x)\in \RR^2$ be a vector field in $\Hh^s(\RR^2)\cap \Hh^1(\RR^2)$ and let ${\rm w}(x)$ be a vector field (or a function) in $\Hh^s(\RR^2)\cap \Hh^{s+1}(\RR^2)$ 
There exists a suitable constant $C$ which does not depend on ${\rm v}$ and ${\rm w}$, such that
	\begin{equation*}
			\sum_{j\in\ZZ}2^{2js}
			\Big(\,
				\I_1^j({\rm v},\,{\rm w})\,+\,
				\I_2^j({\rm v},\,{\rm w})
			\Big)
				\,	\leq\, \frac{C}{\ee}
				\Big(\,
					\|\, \nabla {\rm v}\,\|_{L^2}^2\|\,{\rm w}\,\|_{\dot H^s}^2\,+\,
					\|\, \nabla {\rm w}\,\|_{L^2}^2\|\,{\rm v}\,\|_{\dot H^s}^2
				\Big)
				\,+\,
			\ee\|\,\nabla {\rm w}\,\|_{\Hh^s}^2,
	\end{equation*}
	for any small positive parameter $\ee$.
\end{lemma}
\noindent
Hence, we multiply both sides of \eqref{reg_prop_equj1} by $2^{2js}$ and we perform a sum over $j\in \ZZ$; we finally deduce that

\noindent 

\begin{equation}\label{reg_prop_equj2}
	\begin{aligned}
	\frac{\rho_0}{2}\frac{\dd}{\dd t}
	\|\,	\uu	\,\|_{\Hh^s}^2
	\,&+\,
	(\eta\,+\,\zeta\,)
	\|\,\nabla \uu\,\|_{\Hh^s}^2
	\,\leq
	\,\sum_{j\in\ZZ}2^{2js}
			\Big(\,
				\I_1^j(\uu(t),\,\uu(t))\,+\,
				\I_2^j(\uu(t),\,\uu(t))
			\Big)\,+\\&+\,
	\mu_0
	\langle\,(M \cdot \nabla H\,),\,\uu\rangle_{\dot H^s}\,+\,
	2\zeta
	\left\langle
	\,
	\left(\,\begin{matrix}\hspace{0.25cm}\partial_2  \omega\\-\partial_1 \omega	\end{matrix}\,\right),\, \uu
	\,
	\right\rangle_{\dot H^s}.
	\end{aligned}
\end{equation}
We use now the estimate
\begin{multline}
 \label{eq:Lorentz_estimate_higher_Sobolev}
\angles{M\cdot\nabla H, \ \uu}_{\Hh^s} \leq \varepsilon \pare{ \norm{\nabla \uu}_{\Hh^s}^2 + \norm{\nabla M}_{\Hh^s}^2 }\\
  + \frac{C}{\varepsilon} \bra{ \Big. 
 \pare{1+\norm{M}_{L^2}^2}\norm{\nabla M}_{L^2}^2 + \norm{\nabla \cG_F}_{L^2}^2 
 }\pare{ \norm{ \uu}_{\Hh^s}^2 + \norm{ M}_{\Hh^s}^2 } + \norm{\nabla\cG_F}_{\Hh^s}^2,
 \end{multline}
whose proof is provided in Appendix \ref{sec:tec_est_Sobolev}.  let us now define the function
\begin{equation}\label{eq:def_Phi}
\Phi_{M, H, \uu, F} = \Big(\, 1+
	\|\, M		\,\|_{L^2}^2
		\,+\,
		\|\, H		\,\|_{L^2}^2
		\,+\,
		\|\,  \uu\,\|_{L^2}^2
	\Big)\Big(\,
	\|\,\nabla M		\,\|_{L^2}^2
		\,+\,
		\|\,\nabla H		\,\|_{L^2}^2
		\,+\,
		\|\, \nabla \uu\,\|_{L^2}^2
	\Big) + \norm{\nabla \cG_F}_{L^2}^2 .
\end{equation}

\begin{lemma}
Under the assumptions of Theorem \ref{thm:uniqueness} the function $ \Phi_{M, H, \uu, F} $ defined in \eqref{eq:def_Phi} is well defined and belongs to the space $ L^1\pare{\bR_{+}} $. 
\end{lemma}

\begin{proof}
It suffice to consider the uniform $ L^2 $ energy bound provided in \eqref{eq:L2_energy_bound}. 
\end{proof}

 Combining the  identity \eqref{reg_prop_equj2} and Lemma \ref{lemma:hig_reg} together with the above inequality \eqref{eq:Lorentz_estimate_higher_Sobolev}, we gather that there exists a suitable positive constant $C$ for which the following inequality holds true:
\begin{multline}\label{reg_prop_equj4}
	\frac{\rho_0}{2}\frac{\dd}{\dd t}
	\|\,	\uu	\,\|_{\Hh^s}^2
	\,+\,
	(\eta\,+\,\zeta\,)
	\|\,\nabla \uu\,\|_{\Hh^s}^2
	 \\
	\,\leq
	C
	\ \Phi_{M, H, \uu, F} \
	\Big(
		\|\,	M	\,\|_{\dot H^s}^2\,+\,
		\|\,\uu	\,\|_{\dot H^s}^2
	\Big) + \varepsilon \norm{\nabla M}_{\Hh^s}^2 + C \  \norm{\nabla\cG_F}_{\Hh^s}^2\,+\,
	2\zeta
	\left\langle
	\,
	\left(\,\begin{matrix}\hspace{0.25cm}\partial_2  \omega\\-\partial_1 \omega	\end{matrix}\,\right),\, \uu
	\,
	\right\rangle_{\dot H^s}. 
\end{multline}
We now apply the dyadic block $\Dd_j$ to the equation of the angular momentum $\omega$ in system \eqref{main_system}. We first observe that $\omega_j\,:=\,\Dd_j \omega$ is a smooth solution of the following PDE
\begin{equation*}
\begin{aligned}
	\rho_0 k\big( \partial_t \omega_j\,-\Sd_{j-1}\uu\cdot \nabla \omega_j\,\big) - \eta'\Delta \omega_j\,&=\,\mu_0\, \Delta_j(M\times H)\,+\,2\zeta (\,\curl \uu_j \,-\,2\omega_j\,)\,-\\&-\,
	\,\rho_0k(\Th_\uu\cdot \nabla \omega_j\,-\,\,\Sd_{j-1}\uu\cdot \nabla \omega_j)\,-\rho_0k\Dd_j(\,\Th'_{\nabla \omega} \uu\,).
\end{aligned}
\end{equation*}
Multiplying both sides by $2^{2js}\omega_j$, integrating over $\RR^2$ and performing a sum over $j\in \ZZ$ we get
\begin{equation*}
	\begin{aligned}
	\frac{\rho_0k}{2}\frac{\dd}{\dd t}
	\|\,	\omega	\,\|_{\Hh^s}^2
	+
	\eta'
	\|\,\nabla \omega\,\|_{\Hh^s}^2
	+
	4\zeta
	\|\, \omega\,\|_{\Hh^s}^2
	=
	-\rho_0k
	\sum_{j\in\ZZ}2^{2js}
	\underbrace{
	\int_{\RR^2}(\,	\Dd_j\Th_\uu\cdot\nabla \omega_j\,-\,\Sd_{j-1}\uu\cdot \nabla \omega_j\,)\cdot \omega_j
	}_{\I^j_1(\uu(t),\,\omega(t))}
	-\\-
	\rho_0k
	\sum_{j\in\ZZ}2^{2js}
	\underbrace{
	\int_{\RR^2}\Dd_j(\,\Th'_{\nabla \omega} \uu\,)\,\cdot \omega_j}_{\I^j_2(\uu(t),\,\omega(t))}
	\,+\,
	\sum_{j\in\ZZ}2^{2js}
	\mu_0\int_{\RR^2}\Dd_j(\,M \times H\,)\cdot \omega_j\,+\,
	2\zeta
	\sum_{j\in\ZZ}2^{2js}
	\int_{\RR^2}
	\curl \uu_j\omega_j.
	\end{aligned}
\end{equation*}
Let us state the following inequality, whose proof is postpones in Appendix \ref{sec:est_MxH}; 
\begin{multline}\label{eq:est_Hs_MxH}
\sum_{j\in\ZZ}2^{2js}
	\int_{\RR^2}\Dd_j(\,M \times H\,)\cdot \omega_j
	\\
\begin{aligned}
	&=\,
	\langle\,
		M \times H
		,\,
		\omega
		\,
	\rangle_{\dot H^s}
	\\
	&\leq\,
	\  {\varepsilon}\norm{\omega}_{\Hh^s}^2 + {\varepsilon} \norm{\nabla M}_{\Hh^s}^2 + \frac{\varepsilon}{2} \norm{\nabla \cG_F}_{\Hh^s}^2 \\
    & \hspace{5mm}  + \frac{C}{\varepsilon} \pare{\norm{M}_{L^2}^{2} \norm{ \nabla M}_{L^2}^{2} + \norm{H}_{L^2}^{2} \norm{ \nabla H }_{L^2}^{2}} \norm{M}_{\Hh^s}^{2} + \frac{C}{\varepsilon} \norm{M}_{L^2}^{2} \norm{ \nabla M}_{L^2}^{2} \norm{\cG_F}_{\Hh^s}^{2}.
    \end{aligned}
\end{multline}
for any $ \varepsilon > 0 $. 
We then proceed similarly as for proving \eqref{reg_prop_equj4}: we combine the above identity with Lemma \ref{lemma:hig_reg}, to get
\begin{multline}\label{reg_prop_equj5}
	\frac{\rho_0k}{2}\frac{\dd}{\dd t}
	\|\,	\omega	\,\|_{\Hh^s}^2
	+
	\eta'
	\|\,\nabla \omega\,\|_{\Hh^s}^2
	+
	2\zeta
	\|\, \omega\,\|_{\Hh^s}^2
	-{\varepsilon} \norm{\nabla M}_{\Hh^s}^2 
	\\ 
	\leq\,
	C \ \Phi_{M, H, \uu, F}
	\Big(\,
	\|\,		\omega	\,\|_{\Hh^s}^2
	\,+\,	
	\|\,		\uu	\,\|_{\Hh^s}^2
	\,+\,
	\|\,		M	\,\|_{\Hh^s}^2
	\Big)
	\,+\,
	2\zeta
	\langle\,
	\curl \uu,\,\omega\,\rangle_{\dot H^s}\\
	\frac{C}{\varepsilon} \norm{M}_{L^2}^{2} \norm{ \nabla M}_{L^2}^{2} \norm{\cG_F}_{\Hh^s}^{2} + C \norm{\nabla \cG_F}_{\Hh^s}^2	
	.
\end{multline}
We now apply the dyadic block $\Dd_j$ to the equation of the magnetic induction $M$ in system \eqref{main_system}. We first observe that $M_j\,:=\,\Dd_j M$ is a smooth solution of the following PDE
\begin{equation*}
	\;\partial_t M_j \,+\,\Sd_{j-1}\uu\cdot \nabla  M_j - \sigma \Delta M_j\,=\,
	-(\Th_\uu\cdot \nabla \omega_j\,-\,\,\Sd_{j-1}\uu\cdot \nabla \omega_j) \,+\,\Dd_j\left(\,\left(\,\begin{matrix}\hspace{0.25cm}M_2\\-M_1	\end{matrix}\,\right)\omega \,\right)\,-\, \frac{1}{\tau}(\,M_j-\chi_0H_j\,)
\end{equation*}
Multiplying both sides by $2^{2js}\omega_j$, integrating over $\RR^2$ and performing a sum over $j\in \ZZ$ we get
\begin{equation*}
	\begin{aligned}
	\frac{\rho_0k}{2}\frac{\dd}{\dd t}
	\|\,M	\,\|_{\Hh^s}^2
	+
	\sigma
	\|\,\nabla M\,\|_{\Hh^s}^2
	+
	\frac{1}{\tau}
	\|\, M\,\|_{\Hh^s}^2
	=
	\sum_{j\in\ZZ}2^{2js}
	\underbrace{
	\int_{\RR^2}(\,	\Dd_j\Th_\uu\cdot\nabla M_j\,-\,\Sd_{j-1}\uu\cdot \nabla M_j\,)\cdot M_j
	}_{\I^j_1(\uu(t),\,M(t))}
	-\\-
	\sum_{j\in\ZZ}2^{2js}
	\underbrace{
	\int_{\RR^2}\Dd_j(\Th'_{\nabla M} \uu)\cdot M_j}_{\I^j_2(\uu(t),\,M(t))}
	\,+\,
	\mu_0
	\Big\langle\Big(\begin{matrix}\hspace{0.25cm}M_2\\-M_1	\end{matrix}\Big)\omega,\, M\Big\rangle_{\dot H^s}
	\,+\,
	\frac{\chi_0}{\tau}
	\langle\,
		H,\, M
	\rangle_{\dot H^s}
	\end{aligned}
\end{equation*}
With a procedure very similar to the one adopted in order to prove the inequality \eqref{eq:est_Hs_MxH} we deduce the following bound; 
\begin{multline*}
	\Big\langle\Big(\begin{matrix}\hspace{0.25cm}M_2\\-M_1	\end{matrix}\Big)\omega,\, M\Big\rangle_{\dot H^s}
	\\ \leq\
	\pare{\norm{\omega}_{L^2}^{1/2}\norm{\nabla \omega}_{L^2}^{1/2}\norm{M}_{\Hh^s}^{1/2}\norm{\nabla M}_{\Hh^s}^{1/2}
	+
	\norm{M}_{L^2}^{1/2}\norm{\nabla M}_{L^2}^{1/2}\norm{\omega}_{\Hh^s}^{1/2}\norm{\nabla \omega}_{\Hh^s}^{1/2}	
	} \norm{M}_{\Hh^s}, 
\end{multline*}
and hence applying a Young convexity inequality
\begin{multline*}
\Big\langle\Big(\begin{matrix}\hspace{0.25cm}M_2\\-M_1	\end{matrix}\Big)\omega,\, M\Big\rangle_{\dot H^s}
	\\ 
	\leq  
	\frac{1}{2\tau}\norm{M}_{\Hh^s}^2 + \varepsilon\pare{\norm{\nabla M}_{\Hh^s}^2 + \norm{\nabla \omega}_{\Hh^s}^2}  +\frac{C}{\varepsilon} \ \Phi_{M, H, \uu, F} \ \Big(\,
	\|\,		\omega	\,\|_{\Hh^s}^2
	\,+\,	
	\|\,		\uu	\,\|_{\Hh^s}^2
	\,+\,
	\|\,		M	\,\|_{\Hh^s}^2
	\Big). 
\end{multline*}
Moreover
\begin{equation*}
	\frac{\chi_0}{\tau}\|\,H		\,\|_{\dot H^s}^2
	\,+\,
	\frac{\chi_0}{\tau}
	\langle
	\,H,\,M\,
	\rangle_{\dot H^s}
	\,\leq\,
	C\|\,F\,\|_{\dot H^s}^2
\end{equation*}
Proceeding as for proving \eqref{reg_prop_equj4}, we combine the above identity with Lemma \ref{lemma:hig_reg}, to get
\begin{multline}\label{reg_prop_equj6}
	\frac{\rho_0k}{2}\frac{\dd}{\dd t}
	\|\,M	\,\|_{\Hh^s}^2
	+
	\sigma
	\|\,\nabla M\,\|_{\Hh^s}^2
	+
	\frac{1}{2 \tau}
	\|\, M\,\|_{\Hh^s}^2
	\,+\,
	\frac{\chi_0}{\tau}
	\|\,H\,\|_{\dot H^s}^2
	\\ \leq
	C\|\,\cG_F\,\|_{\Hh^{s}}^2\,+\,
	\frac{C}{\varepsilon} \ \Phi_{M, H, \uu, F} \ \Big(\,
	\|\,		\omega	\,\|_{\Hh^s}^2
	\,+\,	
	\|\,		\uu	\,\|_{\Hh^s}^2
	\,+\,
	\|\,		M	\,\|_{\Hh^s}^2
	\Big) + \varepsilon \norm{\nabla \omega}_{\Hh^s}
\end{multline}
Let us now define the following auxiliary function
\begin{equation*}
\psi = \frac{C}{\varepsilon} \norm{M}_{L^2}^{2} \norm{ \nabla M}_{L^2}^{2} \norm{\cG_F}_{\Hh^s}^{2} + C \norm{\nabla \cG_F}_{\Hh^s}^2	, 
\end{equation*}
indeed $ \psi\in L^1\pare{\bR_+} $ thanks to the $ L^2 $ energy bound for $ M $ and the regularity hypothesis assumed on $ F $.  
Combining the inequalities \eqref{reg_prop_equj4}, \eqref{reg_prop_equj5}, \eqref{reg_prop_equj6}, considering the cancellation
\begin{equation*}
	2\zeta
	\left\langle
	\,
	\left(\,\begin{matrix}\hspace{0.25cm}\partial_2  \omega\\-\partial_1 \omega	\end{matrix}\,\right),\, \uu
	\,
	\right\rangle_{\dot H^s} + 2\zeta
	\langle\,
	\curl \uu,\,\omega\,\rangle_{\dot H^s} =0, 
\end{equation*}
  integrating in time,  we immediately obtain the bound
\begin{equation*}
	\EE_s(t)
	\,+\,
	\int_0^t
	\DD_s\pare{ t'}
	\dd t'
	\,\leq\,	
	\EE_s(0)\,+\,
	C
	\int_0^t
	\Phi_{M, H, \uu, F}\pare{t'}\ 
	\EE_s\pare{ t'}\dd t' +  \int_0^t \psi\pare{t'}\ \d t',
\end{equation*}
for all time $ t \geq 0$. Therefore,  this relation together with the Gronwall's inequality conclude the proof of Theorem \ref{thm-prop-reg2}, provided we show the bounds of Lemma \ref{lemma:hig_reg}.
\hfill $ \Box $\\

\noindent 
In order to complete the proof of Theorem \ref{thm-prop-reg2}, it remains us to get the estimates of Lemma \ref{lemma:hig_reg}.
\begin{proof}[Proof of Lemma \ref{lemma:hig_reg}] We begin with proving inequality $(i)$, which follows from the following commutator estimate (cf. \cite{BCD}, Lemma $10.25$)
\begin{lemma}\label{lemma-commutator:TandS}
 There exists a positive constant $C$ such that, for any suitable functions ${\rm v}$ and ${\rm w}$, the following inequality holds true
 \begin{equation*}
 	\Big\|\,\Dd_j\big(\,\dot{T}_{\rm v}{\rm w}\,\big)\,-\,\Sd_{j-1}{\rm v}\,\Dd_j{\rm w}\,\Big\|_{L^2(\RR^2)}
 	\,\leq\,
 	C
 	\|\,\nabla {\rm v}\,\|_{L^2(\RR^2)}
 	\sum_{|q-j|\leq5}
 	\|\, {\rm w}_q\,\|_{L^2(\RR^2)}
 \end{equation*}
\end{lemma}
\noindent By taking advantage of the mentioned lemma, we get the following bound:
\begin{equation*}
	\I_1^j(t)\,\leq\, C\|\,\nabla \uu(t)\,\|_{L^2}\|\,\uu(t)\,\|_{\Hh^s}\|\,\nabla \uu_j(t)\,\|_{L^2}2^{-js}a_j(t),
\end{equation*}
where the sequence $(\,a_j(t)\,)_{j\in\ZZ}$ belongs to $\ell^2(\ZZ)$, being defined by means of
\begin{equation*}
	a_j(t)\,=\,\sum_{|q-j|\leq5}
 	\frac{2^{qs}\|\, \uu_q(t)\,\|_{L^2(\RR^2)}}{\|\,\uu(t)\,\|_{\Hh^s(\RR^2)}}.
\end{equation*}
\noindent
It remains to control $\I_2^{(j)}$, the non-linear term associated to the reminder $\Delta_jT'_{ \nabla \uu\cdot} \uu$. Since $\Div \uu\,=\,0$, by a repeated use of Bernstein inequalities,
one has
\begin{align*}
	\I^{(j)}_2\,&=\,\int\Dd_j T'_{ \nabla \uu\cdot} \uu\,\uu_j\,=\, 
	\sum_{q\geq j-5}
	\int\Dd_j\bigr( \Dd_q \uu\,\cdot \nabla \Sd_{q+2} \uu\bigr)\,\uu_j \\
&\leq\,	 	
	C\sum_{q\geq j-5} 2^j\,
	\|\,  \Dd_j \bigl( \Delta_q \uu\, \Sd_{q+2}\nabla \uu\bigr)	\,\|_{L^1}\,
	\|			\uu_j			  							\|_{L^2}\,\\
	&\leq\, 	
	C\sum_{q\geq j-5} \|\,\Dd_q\uu\,\|_{L^2}\,\|\Sd_{q+2}\nabla \uu\,\|_{L^2}\,
	\|			\nabla \uu_j		  								\|_{L^2} \\
&\leq\,C\,\|\,\nabla \uu\,\|_{L^2}\,\|\,\nabla \uu_j\,\|_{L^2}\sum_{q\geq j-5} \|\Dd_q \uu\|_{L^2}\,.
\end{align*}
At this point, we remark that
$$
\sum_{q\geq j-5}\|\,\Dd_q \uu\,\|_{L^2}\,\leq\,\left\|\, \uu\,\right\|_{\dot H^s}\,\sum_{q\geq j-5}2^{-qs}\,b_q\,,
$$
where the sequence $\bigl(b_q(t)\bigr)_{q\geq-1}$ is defined, for all time $t\geq0$, by the formula
$$
b_q(t)\,:=\,\frac{2^{qs}\,\|\,\Dd_q\, \uu(t)\|_{H^s}}{\| \,\uu(t)\,\|_{H^s}}\,.
$$
Notice that $\bigl(b_q(t)\bigr)_{q\geq-1}$ belongs to $\ell^2$ and it has unitary norm. In the end, we gather

$$
\I^{(j)}_2\,\leq\,C\,\|\,\nabla \uu\,\|_{L^2}\,\|\, \uu\,\|_{\dot H^s}\,\|\,\nabla \uu_j\,\|_{L^2}\,\sum_{q\geq j-5}2^{-qs}\,b_q\,.
$$
This estimate concludes the proof of inequality $(iii)$, and so of the whole Lemma.

\end{proof}

\appendix

\section{A compactness result}\label{sec:compactness}

In this small section we provide a self-contained proof of Theorem \ref{thm:AL_fractional}, the proof we present here is a slight modification of \cite[Theorem 5.2, p. 61]{Lions69}.\\

 Let us consider a sequence $ \pare{v_n}_n $ in a bounded subset of $ H^\gamma\pare{\bra{0, T}; \ X_0, X_1} $, up to a (non relabeled) subsequence $ v_n \rhu v $ in $ H^\gamma\pare{\bra{0, T}; \ X_0, X_1} $, whence w.l.o.g. we can assume $ v=0 $. If we prove that $ v_n\to 0 $ in $ L^2\pare{\bra{0, T}; X_1} $ we can argue by interpolation  that $ v_n\to 0 $ in $ L^2\pare{\bra{0, T}; X} $. \\
 
 Let us suppose that $ v_n $ is the restriction on $ \bra{0, T} $ of a $ w_n \in H^\gamma\pare{\bR; \ X_0, X_1} $ supported (in time) in $ \bra{-1, T+1} $. Indeed $ w_n \rhu 0 $ in $ H^\gamma\pare{\bR_+ ; \ X_0, X_1} $, we must hence prove that
 
 \begin{equation}\label{eq:conv_Jn}
J_n = \int _{-\infty}^{+\infty} \norm{\cF_t w_n \pare{\tau}}_{X_1}^2 \d \tau 
 \xrightarrow{n\to\infty}0. 
 \end{equation}
 Selecting a $ M>0 $ we can say that
 \begin{align*}
 J_n & = J_{n, M} + J_n^M, 
 \end{align*}
 where
 \begin{align*}
 J_n^M & = \int _{\av{\tau}\geq M} \norm{\cF_t w_n \pare{\tau}}_{X_1}^2 \d \tau, \\
 J_{n, M} & = \int _{\av{\tau}\leq M} \norm{\cF_t w_n \pare{\tau}}_{X_1}^2 \d \tau.
 \end{align*}
 Since, by hypothesis, the sequences $ \pare{\ \cF_t w_n}_n, \  \pare{\av{\tau}^\gamma \cF_t w_n}_n $ are bounded in the space $ L^2\pare{\bR; \ X_1 } $ we can control the high frequency part $ J_n^M  $ as 
 \begin{align*}
 J_n^M & = \int _{\av{\tau}\geq M} \pare{1+\av{\tau}^{2\gamma}} \norm{\cF_t w_n \pare{\tau}}_{X_1}^2 \cdot \frac{1}{\pare{1+\av{\tau}^{2\gamma}}} \d \tau, \\
 & \leq \frac{C}{1+M^{2\gamma}} < \frac{\varepsilon}{2}, 
 \end{align*}
 if $ M>\pare{\frac{2C}{\varepsilon}-1}^{\frac{1}{2\gamma}} $. \\
 
 We must now control the low-frequency part $ J_{n, M} $. In order to do so let us consider a function $ \psi\in\cC^\infty_c\pare{\bR} $, such that $ \psi\pare{t}\equiv 1 $ for each $ t\in\bra{-1, T+1} $. In such setting $ w_n = w_n \psi $, whence
 \begin{equation}\label{eq:Fourier_wn}
 \cF_t w_n\pare{\tau}=\int _{-\infty}^{+\infty} w_n\pare{t}\bra{e^{-2\pi i \ \tau t}\psi\pare{t}}\d t.
 \end{equation}
 
 \noindent We want to prove that
 \begin{equation*}
 \text{for each }\tau\in \bR \hspace{5mm} \cF_t w_n\pare{\tau}\rhu 0 \text{ in } X_0. 
 \end{equation*}

 \noindent In order to do so let us consider a $ \phi\in X_0' $ and a generic $ \tau\in\bR $, indeed
 \begin{align*}
 \psc{\cF_t w_n\pare{\tau}}{\phi}_{X_0\times X_0'} & = \int \cF_t w_n\pare{x, \tau} \ \phi\pare{x}\dx, \\
 & = \int \pare{\int _{-\infty}^{+\infty} w_n\pare{x, t}\bra{e^{-2\pi i \ \tau t}\psi\pare{t}}\d t} \ \phi\pare{x}\dx, \\
 & = \int \int _{-\infty}^{+\infty}  w_n\pare{x, t} \underbrace{\bra{e^{-2\pi i \ \tau t}\psi\pare{t} \phi\pare{x}}}_{=\Phi_\tau\pare{x, t}} \d \pare{x, t}, \\
 & = \psc{w_n}{\Phi_\tau}_{L^2\pare{\bra{0, T}; X_0}\times L^2\pare{\bra{0, T}; X_0'}}, \\
 & \xrightarrow{n\to \infty} 0, 
 \end{align*}
 since $ w_n\rhu 0 $ in $ L^2\pare{\bra{0, T}; X_0} $. Moreover since $ X_0\Subset X_1 $ we deduce that
  \begin{equation*}
 \text{for each }\tau\in \bR \hspace{5mm} \cF_t w_n\pare{\tau}\to 0 \text{ in } X_1. 
 \end{equation*}
 
  Using the identity \eqref{eq:Fourier_wn} we can deduce now that
 \begin{equation*}
 \norm{\cF_t w_n\pare{\tau}}_{X_1}\leqslant\norm{w_n}_{L^2\pare{\bR; X_1}} \norm{e^{-2\pi i \ \tau t}}_{L^2\pare{\bR}}\leq C < \infty, 
 \end{equation*}
 whence using Lebesgue dominated convergence theorem we deduce that
 \begin{equation*}
 \cF_t w_n \xrightarrow{n\to \infty} 0 \text{ in } L^1_\tau\pare{\bR; X_1}. 
 \end{equation*}
 
 \noindent This does not suffice still in order to prove the convergence \eqref{eq:conv_Jn}, it is hence here that we use the localization of $ J_{n , M} $ onto the low frequencies $ \av{\tau}\leq M $, using a Bernstein inequality in fact we can argue that
 \begin{equation*}
 J_{n , M} \leq C M^{1/2} \norm{\cF w_n}_{L^1_\tau\pare{\bR; X_1}}\xrightarrow{n\to\infty} 0, 
 \end{equation*}
 concluding the proof of Theorem \ref{thm:AL_fractional}.

\section{Technical estimates}

In the present technical section we will use continuously the following technical estimate
\begin{equation}\label{eq:regularity_dyadic}
\norm{\Dd_j f}_{L^2}\leq C\ c_j \  2^{-qs}\norm{f}_{\Hh^s}, 
\end{equation}
for any $ j\in\ZZ $ and a $ \pare{c_j}_{j\in\ZZ}=\pare{c_j\pare{f}}_{j\in\ZZ}\in \ell^2\pare{\ZZ} $ so that $ \sum_j c_j^2=1 $. 

\subsection{Proof of Lemma \ref{comm-est}}
\noindent This section is devoted to the proof of Lemma \ref{comm-est}, which played a major role in the uniqueness result of weak solutions for system \eqref{main_system}, in Section \ref{sec-uniq}. We recall that the uniqueness result holds thanks to suitable energy estimates at a level of Sobolev regularity 
$\Hh^{-1/2}(\RR^2)$, in particular the one of Lemma \ref{comm-est}, which can be summarized into
\begin{equation*}
			\langle\,  {\rm v}\cdot \nabla B,\, B\,\rangle_{\Hh^{-\frac{1}{2}}(\RR^2)}\,\leq\, C \|\,\nabla {\rm v}\,\|_{L^2(\RR^2)}\|\,\nabla B\,\|_{\Hh^{-\frac{1}{2}}}\|\,B\,\|_{\Hh^{-\frac{1}{2}}},
		\end{equation*}
We then aim to prove
a generalization of the above inequality, more precisely we perform a suitable estimate at a level of  Sobolev regularity $\Hh^\Tt(\RR^2)$, for any real index $\Tt\in\RR_+$ with $\Tt>-1$. This range specifically includes the case of $\Hh^{-1/2}(\RR^2)$, with $\Tt=-1/2$. We aim at proving the following general statement.
\begin{lemma}\label{comm-est:appendix}
Let $\Tt>-1$, then 
for any divergence-free vector field $\rm v$ in $\Hh^{1}(\RR^2)$ and any vector field $B$ in $\Hh^{\Tt}(\RR^2)$, the following bound
holds
		\begin{equation*}
			\langle\,  {\rm v}\cdot \nabla B,\, B\,\rangle_{\Hh^{\Tt}(\RR^2)}\,\leq\, C \|\,\nabla {\rm v}\,\|_{L^2(\RR^2)}\|\,\nabla B\,\|_{\Hh^{\Tt}(\RR^2)}\|\,B\,\|_{\Hh^{\Tt}(\RR^2)},
		\end{equation*}
		for a suitable positive constant $C$.
	\end{lemma}
\begin{proof}
	We recast the nonlinear term ${\rm v}\cdot \nabla B$ making use of the Bony decomposition
	\begin{equation*}	
		{\rm v}\cdot \nabla B\,=\,\Th_{\rm v} \nabla B\,+\,\Th'_{\nabla B}{\rm v},
	\end{equation*}
	where we recall that the homogeneous paraproduct $,\Th_{\rm v} \nabla B$ and the homogeneous reminder $\Th'_{\nabla B}{\rm v}$ are defined as
	\begin{equation*}
	\Th_{\rm v}\cdot \nabla B\,=\,\sum_{j\in\ZZ}\Sd_{j-1}{\rm v}\cdot \nabla \Dd_jB\quad\quad\text{and}\quad\quad
	\Th_{\nabla B}'{\rm v}\,=\,\sum_{j\in\ZZ}\Dd_{j}{\rm v}\cdot \nabla \Sd_{j+1}B,
	\end{equation*}
	For the sake of an unified presentation, we now denote by $B_j:=\Dd_j B$ and by ${\rm v}_j:= \Dd_j{\rm v}$. 
	Thus, 
	thanks to the isomorphism between Hilbert spaces $\BB_{2,2}^{s}(\RR^2)\,\simeq\,\Hh^{\Tt}(\RR^2)$, we can reformulate the $\Hh^{\Tt}$-inner product  
	of the statement  as follows:
	\begin{equation}\label{appx-lemma-com-ineq1}
		\begin{aligned}
			\sum_{j\in\ZZ}2^{2j\Tt}\int_{\RR^2}\Dd_j({\rm v}\cdot \nabla B)\cdot B_j\,=\,
			\sum_{j\in\ZZ}
			2^{2j\Tt}
			\bigg\{
			\int_{\RR^2}(\Th_{\rm v} \cdot \nabla B_j\,-\,\Sd_{j-1}{\rm v}\nabla B_j)\cdot B_j
			\,+\\+\,
			\int_{\RR^2}
			(\Sd_{j-1}{\rm v}\cdot \nabla B_j)
			\cdot B_j
			\,+\,
			\int_{\RR^2}
			\Dd_j(\Th_{\nabla B}{\rm v}) \cdot B_j
			\bigg\}.
		\end{aligned}
	\end{equation}
	Now, the free divergence condition on $\rm v$ is automatically  transferred in its localization $\Sd_{j-1}{\rm v}$, which means 
	$\Div\,\Sd_{j-1}{\rm v}\,=\,0$. 
 	This property leads to the standard cancellation 
 	\begin{equation}\label{appx-lemma-eq2}
 		\int_{\RR^2}	(\,\Sd_{j-1}{\rm v}\cdot \nabla B_j\,)\cdot  B_j\,=\,0, \quad\quad
 		\text{for any}\quad j\in \ZZ.
 	\end{equation}
	We then proceed to estimate the remaining term, the first making use of Lemma \ref{lemma-commutator:TandS} while the second one by  a standard 
 	Bernstein-type inequality. More precisely, Lemma \ref{lemma-commutator:TandS}  yields
 	\begin{equation}\label{appx-lemma-ineq4}
 	\begin{aligned}
 		\sum_{j\in\ZZ}2^{2j\Tt}
 		\int_{\RR^2}(\Th_{\rm v} \cdot \nabla B_j\,-\,\Sd_{j-1}{\rm v}\cdot \nabla B_j)\cdot B_j
 		&\lesssim\,
 		\sum_{j\in\ZZ}2^{2js}
 		\|\,\Th_{\rm v} \cdot \nabla B_j\,-\,\Sd_{j-1}{\rm v}\cdot \nabla B_j\,\|_{L^2}\|\,B_j\,\|_{L^2}
 		\\
 		&\lesssim
 		\sum_{j\in\ZZ}2^{2js}
 		\sum_{|q-j|\leq 5}\|\,\nabla {\rm v}\,\|_{L^2}\|\,B_q\|_{L^2}\|\,B_j\,\|_{L^2}\\
 		&\lesssim
 		\|\,\nabla {\rm v}\,\|_{L^2}
 		\Bigg(\,
 			\sum_{q\in\ZZ}
 			2^{2qs}
 			\|\,B_q\|_{L^2}^2
 		\,\Bigg)^{\frac{1}{2}} 
 		\Bigg(\,
 			\sum_{j\in\ZZ}
 			2^{2js}
 			\|\,B_j\|_{L^2}^2
 		\,\Bigg)^{\frac{1}{2}} 
 		\\
 		&\lesssim
 		\|\,\nabla {\rm v}\,\|_{L^2(\RR^2)}\|\,\nabla B\,\|_{\Hh^{s}(\RR^2)}\|\,B\,\|_{\Hh^{s}(\RR^2)},
 		\end{aligned}
	\end{equation} 	 
	It then remains to bound the last term in \eqref{appx-lemma-com-ineq1}, related to the low frequencies of $B$, that is
	\begin{equation*}
		\sum_{j\in \ZZ}2^{2j\Tt}	
		\int_{\RR^2}
		\Dd_j(\Th_{\nabla B}{\rm v}) \cdot B_j		
		\,
		\lesssim
		\,
		\sum_{j\in \ZZ}2^{2j\Tt}
		\sum_{q>j-5}
		\|\,\Dd_j({\rm v}_q\cdot \nabla \Sd_{q+2}B)\,\|_{L^2}
		\|\,B_j\,\|_{L^2}.
	\end{equation*}
	Now, we separately analyze any $j$-term on the right hand sides of the above inequality:
	a standard Bernstein type inequality first leads to the following bound 
	\begin{equation*}
	\begin{aligned}
		\|\,\Dd_j({\rm v}_q \nabla \Sd_{q+2}\nabla B)\,\|_{L^2}
		\|\,B_j\,\|_{L^2}
		\,\lesssim\,
		2^j\|\,\Dd_j({\rm v}_q\cdot  \Sd_{q+2}\nabla B)\,\|_{L^1}
		\|\,B_j\,\|_{L^2}
		\,\lesssim\,
		\|\,{\rm v}_q\,\|_{L^2}
		\|\,\Sd_{q+2}\nabla B\,\|_{L^2}
		\|\,\nabla B_j\,\|_{L^2}\\
		\lesssim		
		2^{j-q}
		\|\,\nabla {\rm v}_q\,\|_{L^2}
		\|\,\Sd_{q+2}\nabla B\,\|_{L^2}
		\|\, B_j\,\|_{L^2}
		\,\lesssim\,
		\|\,\nabla {\rm v}_q\,\|_{L^2}
		\|\,\Sd_{q+2}\nabla B\,\|_{L^2}
		\|\, B_j\,\|_{L^2},
	\end{aligned}
	\end{equation*}
	from which we deduce 
	\begin{equation}\label{appx-lemma-ineq2}
		\begin{aligned}
			\sum_{j\in \ZZ}2^{2j\Tt}
			\sum_{q>j-5}
		&\|\,\Dd_j({\rm v}_q \nabla \Sd_{q+2}\nabla B)\,\|_{L^2}
		\|\,B_j\,\|_{L^2}
		\,\lesssim\\
		&\lesssim
		\|\,\nabla {\rm v}\,\|_{L^2}
		\|\, B\,\|_{\Hh^{\Tt}}
		\|\,\nabla B\,\|_{\Hh^{\Tt}}
		\sum_{j\in \ZZ}
		\sum_{q>j-5}
		2^{(j-q)(\Tt+1)}
		z_q(t),		
		\end{aligned}
	\end{equation}
	Here the sequence $(z_q(t))_{q\in\ZZ}$ is a priori in $l^1(\ZZ)$, being defined by means of
	\begin{equation*}
		z_q(t)\,:=\,2^{2q\Tt}\frac{\|\, B_q(t)\,\|_{L^2}\|\,\nabla B_q(t)\,\|_{L^2}}{\|\, B(t)\,\|_{\Hh^{\Tt}}
		\|\,\nabla B(t)\,\|_{\Hh^{\Tt}}}
		\,\Rightarrow\,
		\|\,(z_q(t))_{q\in\ZZ}\,\|_{l^1(\ZZ)}\,\leq\, 1.
	\end{equation*}
	Secondly we apply a Young-type inequality between convolution of sequences as follows:
	\begin{equation*}
		\Bigg\|
				\sum_{j\in \ZZ}
				\sum_{q>j-5}2^{(j-q)(\Tt+1)}z_q(t)
		\Bigg\|_{\ell^1(\ZZ)}
		\,\leq\,
		\Bigg(
			\sum_{k<5}
			2^{k(\Tt+1)}
		\Bigg)
		\|\,(z_q(t))_{q\in\ZZ}\,\|_{l^1(\ZZ)}
		\,\leq\,C,
	\end{equation*}
	where $C$ is a positive constant $C$ which depends only on $\Tt$. Hence, we couple this estimate together with 
	\eqref{appx-lemma-ineq2}, to eventually gather
	\begin{equation}\label{appx-lemm-comm-ineq3}
		\begin{aligned}
			\sum_{j\in \ZZ}2^{2j\Tt}
			\sum_{q>j-5}
		\|\,\Dd_j({\rm v}_q \nabla \Sd_{q+2}\nabla B)\,\|_{L^2}
		\|\,B_j\,\|_{L^2}
		\,\lesssim\,
		\|\,\nabla {\rm v}\,\|_{L^2}
		\|\, B\,\|_{\Hh^{\Tt}}
		\|\,\nabla B\,\|_{\Hh^{\Tt}}.
		\end{aligned}
	\end{equation}
	In order to conclude the proof of the Lemma, we summarize inequalities \eqref{appx-lemma-ineq4}, \eqref{appx-lemm-comm-ineq3} and identity 
	\eqref{appx-lemma-eq2} into \eqref{appx-lemma-com-ineq1}, to finally achieve
	\begin{equation*}
		\sum_{j\in\ZZ}2^{2j\Tt}
 		\int_{\RR^2}\Dd_j({\rm v}\cdot \nabla B)\cdot B_j
 		\,\lesssim\,
 		|\,\nabla {\rm v}\,\|_{L^2}
		\|\, B\,\|_{\Hh^{\Tt}}
		\|\,\nabla B\,\|_{\Hh^{\Tt}},
	\end{equation*}
	which corresponds to the desired inequality. 
\end{proof}

\subsection{Proof of \eqref{eq:Lorentz_estimate_higher_Sobolev}}\label{sec:tec_est_Sobolev}

In order to prove the estimate \eqref{eq:Lorentz_estimate_higher_Sobolev} we will require the following auxiliary estimate

\begin{lemma}\label{lem:tec_est_for_tec_est}
The following estimate holds true
\begin{equation*}
\norm{\Dd_j \pare{M\cdot \nabla H}}_{L^{4/3}} \lesssim c_j \ 2^{-js} \pare{\norm{M}_{L^2}^{1/2} \norm{ \nabla M}_{L^2}^{1/2}\norm{\nabla H}_{\Hh^s} + \norm{M}_{\Hh^s}^{1/2} \norm{ \nabla M}_{\Hh^s}^{1/2}\norm{\nabla H}_{L^2}}, 
\end{equation*}
for some sequence $ \pare{c_j}_j=\pare{c_j\pare{M, H}}_j\in \ell^2\pare{\ZZ} $. 
\end{lemma}

\begin{proof}
Using Bony paraproduct decomposition as in \eqref{Bony2} we can say that
\begin{equation*}
\Dd_j \pare{M\cdot \nabla H} = \Dd_j \Th_M \nabla H + \Dd_j \Th' _{\nabla H} M, 
\end{equation*}
where thanks to the almost-orthogonality properties of the dyadic blocks (cf. \cite[Chapter 2]{BCD}) we can say that
\begin{align*}
\Dd_j \Th_M \nabla H & = \sum_{\av{j-q}\leq 4} \Dd_j \pare{ \Sd_{q-1} M \  \Dd_q\nabla H}, \\
\Dd_j \Th' _{\nabla H} M & = \sum_{q>j-4} \Dd_j \pare{ \Dd_q M \  \Sd_{q+2}\nabla H }. 
\end{align*}
We can use H\"older inequality, Gagliardo-Nirenberg inequality and \eqref{eq:regularity_dyadic} in order to deduce that
\begin{equation*}
\begin{aligned}
\norm{\Dd_j\Th_M\nabla H}_{L^{4/3}} & \lesssim \norm{\Sd_{q-1} M}_{L^4} \norm{\Dd_q \nabla H}_{L^2}, \\
&\lesssim c_j \norm{M}_{L^2}^{1/2}\norm{\nabla M}_{L^2}^{1/2} \norm{\nabla H}_{\Hh^s}. 
\end{aligned}
\end{equation*}
In the above estimate we used repeatedly the fact that for the term $ \Dd_j\Th_M\nabla H $ the summation indexes $ j $ and $ q $ differ at most for a uniform and constant value, hence we can interchange them at the prize of a multiplication for a constant. \\

\noindent
Similar computations  allows us to deduce the bound
\begin{align*}
\norm{\Dd_j \Th' _{\nabla H} M}_{L^{4/3}} \lesssim 2^{-js} \norm{M}_{\Hh^s}^{1/2} \norm{ \nabla M}_{\Hh^s}^{1/2}\norm{\nabla H}_{L^2} \sum_{q>j-4}2^{\pare{j-q}s}\tilde{c}_q,
\end{align*}
but we remark that
\begin{equation*}
\pare{\sum_{q>j-4}2^{\pare{j-q}s}\tilde{c}_q}_{j\in\ZZ} = \pare{\pare{2^{ps}1_{p<4}}\star \tilde{c}_p\big. }_{j\in\ZZ} = \pare{c_j}_{j\in\ZZ}\in\ell^2\pare{\ZZ}, 
\end{equation*}
concluding the proof. 
\end{proof}

With the result of Lemma \ref{lem:tec_est_for_tec_est} the proof of \eqref{eq:Lorentz_estimate_higher_Sobolev} is immediate. Let us write
\begin{equation}
\label{eq:Lorentz_technical_est_1}
\begin{aligned}
\angles{M\cdot\nabla H, \ \uu}_{\Hh^s} & \lesssim \sum_j 2^{2js}\angles{\Dd_j\pare{M\cdot\nabla H}, \ \Dd_j \uu}_{L^2}, \\
& \lesssim \sum_j 2^{2js} \norm{\Dd_j\pare{M\cdot\nabla H}}_{L^{4/3}}\norm{\Dd_j\uu}_{L^4}, \\
& \lesssim \pare{\norm{M}_{L^2}^{1/2} \norm{ \nabla M}_{L^2}^{1/2}\norm{\nabla H}_{\Hh^s} + \norm{M}_{\Hh^s}^{1/2} \norm{ \nabla M}_{\Hh^s}^{1/2}\norm{\nabla H}_{L^2}}\norm{\uu}_{\Hh^s}^{1/2}\norm{\nabla\uu}_{\Hh^s}^{1/2}. 
\end{aligned}
\end{equation}

\noindent We rely now on the following technical result, for a proof of which we refer to \cite[Lemma 3.3]{Stefano2}; 
\begin{lemma}\label{lem:reg_H}
Let us fix a $ s\in\bR $ and let $ M, \cG_F $ be such that $ { M,  \cG_F\in \Hh^s} $, then there exists a positive constant $ C $ such that
\begin{equation*}
\begin{aligned}
\norm{  H}_{\Hh^s} & \leqslant C \pare{\norm{ M}_{\Hh^s} + \norm{ \cG_F}_{\Hh^s}}
,
\end{aligned}
\end{equation*}
where $ \cG_F =\nabla\Delta^{-1}F $. 
\end{lemma}

\noindent Inserting the result of Lemma \ref{lem:reg_H} in the estimate \eqref{eq:Lorentz_technical_est_1} and using repeatedly a Young convexity inequality of the form 
\begin{align*}
ab\leq \frac{\varepsilon a^p}{p} + \frac{C\ b^q}{\varepsilon q}, && \frac{1}{p}+\frac{1}{q}=1,
\end{align*} 
 we deduce the following inequality for any $ \varepsilon >0 $
 \begin{multline*}
 \angles{M\cdot\nabla H, \ \uu}_{\Hh^s} \leq \varepsilon \pare{ \norm{\nabla \uu}_{\Hh^s}^2 + \norm{\nabla M}_{\Hh^s}^2 }\\
  + \frac{C}{\varepsilon} \bra{ \Big. 
 \pare{1+\norm{M}_{L^2}^2}\norm{\nabla M}_{L^2}^2 + \norm{\nabla \cG_F}_{L^2}^2 
 }\pare{ \norm{ u}_{\Hh^s}^2 + \norm{ M}_{\Hh^s}^2 } + \norm{\nabla\cG_F}_{\Hh^s}^2,
 \end{multline*}
 concluding the proof of estimate \eqref{eq:Lorentz_estimate_higher_Sobolev}.

 \subsection{Proof of \eqref{eq:est_Hs_MxH}} \label{sec:est_MxH}
 
 In a fashion analogous of what was done above we need first the following technical result; 
 
 \begin{lemma}\label{lem:tec_est_for_MxH}
 The following estimate hold true
 \begin{equation*}
 \norm{\Dd_j \pare{ M\times H}}_{L^2}\lesssim c_j \ 2^{-js} \pare{
 \norm{M}_{L^2}^{1/2} \norm{ \nabla M}_{L^2}^{1/2} \norm{H}_{\Hh^s}^{1/2} \norm{ \nabla H}_{\Hh^s}^{1/2} + 
 \norm{H}_{L^2}^{1/2} \norm{ \nabla H}_{L^2}^{1/2} \norm{M}_{\Hh^s}^{1/2} \norm{ \nabla M}_{\Hh^s}^{1/2}
 }
 \end{equation*}
 \end{lemma}
 
 \begin{proof}
 As in the proof of Lemma \ref{lem:tec_est_for_tec_est} we can use Bony paraproduct decomposition in order so that
 \begin{equation*}
 \Dd_j \pare{M\times H} = \Dd_j \Th _{M} H + \Dd_j \Th' _{H} M.  
 \end{equation*}
 Whence using the interpolation inequality $ \norm{f}_{L^4}\lesssim \norm{f}_{L^2}^{1/2} \norm{ \nabla f}_{L^2}^{1/2} $ and \eqref{eq:regularity_dyadic}; 
 \begin{align*}
 \norm{\Dd_j \Th _{M} H}_{L^2} & \leq \sum_{\av{j-q}\leq 4} \norm{\Sd_{q} M}_{L^4} \norm{\Dd_q H}_{L^4}, \\
 &  \lesssim c_j \ 2^{-js} 
 \norm{M}_{L^2}^{1/2} \norm{ \nabla M}_{L^2}^{1/2} \norm{H}_{\Hh^s}^{1/2} \norm{ \nabla H}_{\Hh^s}^{1/2}. 
 \end{align*}
 Similar computations lead to the bound
 \begin{equation*}
 \norm{\Dd_j \Th' _{H} M}_{L^2}\lesssim c_j \ 2^{-js} \norm{H}_{L^2}^{1/2} \norm{ \nabla H}_{L^2}^{1/2} \norm{M}_{\Hh^s}^{1/2} \norm{ \nabla M}_{\Hh^s}^{1/2}. 
 \end{equation*}
 \end{proof}

 We can now use the result of Lemma \ref{lem:tec_est_for_MxH} in order to conclude the proof of \eqref{eq:est_Hs_MxH}. Using the result of Lemma \ref{lem:tec_est_for_MxH} and \eqref{eq:regularity_dyadic} we can in fact argue that
 \begin{multline}\label{eq:tecest0}
 \int \Dd_j\pare{M\times H}\cdot \Dd_j\omega \  \dx
\\
\begin{aligned}
\lesssim & \ \norm{\Dd_j\pare{M\times H}}_{L^2}\norm{\Dd_j \omega}_{L^2}, \\
\lesssim & \  b_j \ 2^{-2js} \pare{
 \norm{M}_{L^2}^{1/2} \norm{ \nabla M}_{L^2}^{1/2} \norm{H}_{\Hh^s}^{1/2} \norm{ \nabla H}_{\Hh^s}^{1/2} + 
 \norm{H}_{L^2}^{1/2} \norm{ \nabla H}_{L^2}^{1/2} \norm{M}_{\Hh^s}^{1/2} \norm{ \nabla M}_{\Hh^s}^{1/2}
 } \norm{\omega}_{\Hh^s}. 
\end{aligned}
 \end{multline}
 
 We can hence use a Young convexity inequality and Lemma \ref{lem:reg_H} in order to deduce the estimate
 \begin{multline}\label{eq:tecest1}
 \norm{M}_{L^2}^{1/2} \norm{ \nabla M}_{L^2}^{1/2} \norm{H}_{\Hh^s}^{1/2} \norm{ \nabla H}_{\Hh^s}^{1/2}\norm{\omega}_{\Hh^s} \\
 \begin{aligned}
 \lesssim & \ \frac{\varepsilon}{2}\norm{\omega}_{\Hh^s}^2 + \frac{\varepsilon}{2} \norm{\nabla H}_{\Hh^s}^2 + \frac{C}{\varepsilon} \norm{M}_{L^2}^{2} \norm{ \nabla M}_{L^2}^{2} \norm{H}_{\Hh^s}^{2}, \\
 \lesssim & \ \frac{\varepsilon}{2}\norm{\omega}_{\Hh^s}^2 + \frac{\varepsilon}{2} \norm{\nabla M}_{\Hh^s}^2 + \frac{\varepsilon}{2} \norm{\nabla \cG_F}_{\Hh^s}^2 \\
   & \ + \frac{C}{\varepsilon} \norm{M}_{L^2}^{2} \norm{ \nabla M}_{L^2}^{2} \norm{M}_{\Hh^s}^{2} + \frac{C}{\varepsilon} \norm{M}_{L^2}^{2} \norm{ \nabla M}_{L^2}^{2} \norm{\cG_F}_{\Hh^s}^{2}. 
 \end{aligned}
 \end{multline}

 \noindent Similarly we can deduce that
 \begin{equation}\label{eq:tecest2}
 \begin{aligned}
 \norm{H}_{L^2}^{1/2} \norm{ \nabla H}_{L^2}^{1/2} \norm{M}_{\Hh^s}^{1/2} \norm{ \nabla M}_{\Hh^s}^{1/2}\norm{\omega}_{\Hh^s} & \lesssim \frac{\varepsilon}{2}\norm{\omega}_{\Hh^s}^2 + \frac{\varepsilon}{2} \norm{\nabla M}_{\Hh^s}^2 + \frac{C}{\varepsilon} \norm{H}_{L^2}^{2} \norm{ \nabla H}_{L^2}^{2} \norm{M}_{\Hh^s}^{2}. 
 \end{aligned}
 \end{equation}

 \noindent Plugging hence estimate \eqref{eq:tecest1} and \eqref{eq:tecest2} in \eqref{eq:tecest0} and summing in $ j\in\ZZ $ we deduce the bound
 
 \begin{multline*}
 \sum_{j\in\ZZ}2^{2js}\int \Dd_j\pare{M\times H}\cdot \Dd_j\omega \  \dx\\
\begin{aligned}
\lesssim & \  {\varepsilon}\norm{\omega}_{\Hh^s}^2 + {\varepsilon} \norm{\nabla M}_{\Hh^s}^2 + \frac{\varepsilon}{2} \norm{\nabla \cG_F}_{\Hh^s}^2 \\
    & \  + \frac{C}{\varepsilon} \pare{\norm{M}_{L^2}^{2} \norm{ \nabla M}_{L^2}^{2} + \norm{H}_{L^2}^{2} \norm{ \nabla H }_{L^2}^{2}} \norm{M}_{\Hh^s}^{2} + \frac{C}{\varepsilon} \norm{M}_{L^2}^{2} \norm{ \nabla M}_{L^2}^{2} \norm{\cG_F}_{\Hh^s}^{2}, 
\end{aligned}
 \end{multline*}
 concluding the proof of \eqref{eq:est_Hs_MxH}.

\section{Overview of the Littlewood Paley theory}\label{app:LP}
\noindent
This section is devoted to an overview of the Littlewood-Paley theory. We present here some technical tools that have been crucial to our proofs.
For more specifics, we refer the interested reader to \cite{BCD}.

\noindent
We begin with introducing the so called ``Littlewood-Paley decomposition'', which is characterized by an homogeneous partition of unity within the Fourier space
$\RR^\dd_\xi$. 

\noindent
We first take into account a radial function $\chi$ which belongs to ${\mathcal{D}}(B(0,2)) $, being identically 1 in $B(0,1/2)$. We also assume that 
the function
$$r\in\RR_+\,\rightarrow \,\chi(re)$$
is non increasing, for any vector $e$ in $\RR^\dd$. We then introduce the sequence $(\varphi_j)_{j\in\ZZ}$ as 
\begin{equation*}
\varphi_j(\xi)\,: =\,\chi\left(\frac{\xi}{2^{j+1}}\right)\,-\,\chi\left(\frac{\xi}{2^j}\right),
\end{equation*}
which satisfies
\begin{equation*}
|j-q|\,>\,5\,
\Rightarrow\,
{\rm Supp}\,\, \varphi_j\cap {\rm Supp}\,\, \varphi_q\,=\,\emptyset,\quad\text{and}\quad
\sum_{j\in\ZZ}\varphi_j(\xi)\,=\,1\;\;\forall \xi\in\RR^\dd\setminus\{0\}.
\end{equation*}
For any integer $j\in\ZZ$, the homogeneous dyadic block  $\Dd_j$ and the low frequency cut-off operator $\Sd_j$ are then defined by means of
\begin{equation*}
	\Dd_j\,:=\,	\varphi_j(D),\quad\quad\Sd_j\,:=\,\sum_{q\leq j-1}\Dd_q.
\end{equation*}
where throughout we agree  that  $f(D)$ stands for 
the pseudo-differential operator 
$$u\rightarrow\FF^{-1}(f\,\FF_x(u)),$$ for any smooth function $f$. Both $\Dd_j$ and $\Sd_j$ maps the Lebesgue space $L^p(\RR^\dd)$ into itself, for any $j\in\ZZ$ and  $p\in[1,+\infty]$. The norm of these maps are independent of the indexes $j$ and $p$.

\noindent 
\noindent
The following classical property holds true: for any homogeneous tempered distribution $u\in\mycal{S}_h'$, we can decompose $\uu$ as 
$$u\,=\,\sum_{j}\Dd_ju,$$ 
in the sense of $\mycal{S}_h'$.
Formally, for two appropriately tempered distributions $a$ and $b$ we have the so called Bony's paraproduct decomposition
\cite{Bony81}:
\begin{equation}\label{Bony}
	ab\,=\,\Th_a b\,+\,\Th_b a\,+\,\dot{R}(a,b)
\end{equation} 
being $\Th_a b$ the homogeneous paraproduct and $\dot{R}(a,b)$ the homogeneous reminder, i.e.
\begin{equation*}
\Th_a b=\sum_{\substack{j\in\ZZ}}\Sd_{j-1} a\Dd_{j}b,\quad\quad 
\dot{R}(a,b)=\sum_{\substack{j\in\ZZ,\\ i\in\{0,\pm 1\}} }\Dd_{j} a\Dd_{j+i} b.
\end{equation*}
It is common in literature to regroup the second homogeneous paraproduct of \eqref{Bony} into the reminder as follows:
\begin{equation}\label{Bony2}
	ab\,=\,\Th_a b\,+\,\Th_b' a,\quad\quad\text{where}\quad\quad \Th_b' a\,:=\,\sum_{\substack{j\in\ZZ}}\Dd_{j} a\Sd_{j+2}b
\end{equation}

\smallskip\noindent
Let us present the so-called \emph{Bernstein's inequalities}, which explain the way derivatives act on spectrally localized functions.
\begin{lemma} \label{l:bern}
Let  $0<r<R$.   A constant $C$ exists so that, for any nonnegative integer $k$, any couple $(p,q)$ 
in $[1,+\infty]^2$, with  $p\leq q$,  and any function $u\in L^p$,  we  have, for all $\lambda>0$,
$$
\displaylines{
{\rm supp}\, \widehat u \subset   B(0,\lambda R)\quad
\Longrightarrow\quad
\|\nabla^k u\|_{L^q}\, \leq\,
 C^{k+1}\,\lambda^{k+d\left(\frac{1}{p}-\frac{1}{q}\right)}\,\|u\|_{L^p}\;;\cr
{\rm supp}\, \widehat u \subset \{\xi\in\R^d\,|\, r\lambda\leq|\xi|\leq R\lambda\}
\quad\Longrightarrow\quad C^{-k-1}\,\lambda^k\|u\|_{L^p}\,
\leq\,
\|\nabla^k u\|_{L^p}\,
\leq\,
C^{k+1} \, \lambda^k\|u\|_{L^p}\,.
}$$
\end{lemma}   
\noindent
We are now in the position to define the class of homogeneous Besov spaces.
\begin{definition} \label{d:B}
  Let $s\in\R$ and $1\leq p,r\leq+\infty$. The \emph{homogeneous Besov space}
$\BB^{s}_{p,r}\,=\,\BB^{s}_{p,r}(\RR^\dd)$ is defined as the subset of homogeneous tempered distributions $u$ in $\mycal{S}'_h$ for which
the following norm is bounded
\begin{equation*}
	\|\,u\,\|_{\BB^{s}_{p,r}}\,:=\,
	\left\|\,\left(\,2^{js}\,\|\,\Dd_ju\,\|_{L^p}\,\right)_{j\in\ZZ}\,\right\|_{\ell^r}\,<\,+\infty\,.
\end{equation*}
\end{definition}
\noindent
Homogeneous Besov spaces can be understood as interpolation spaces between the Sobolev ones. Furthermore for all $s\in\R$ we have the isomorphism of Banach spaces $\BB^s_{2,2}\cong \Hh^s$, with
\begin{equation} \label{eq:LP-Sob}
\|\,f\,\|_{\Hh^s}\,\sim\,\left(\sum_{j\in\ZZ}2^{2 j s}\,\|\,\Dd_jf\,\|^2_{L^2}\right)^{1/2}\,.
\end{equation}
Indeed, the previous isomorphism is an isomorphism between Hilbert spaces: the $\BB^s_{2,2}$-inner product 
\begin{equation*}
\langle f,\, g\rangle_{\BB_{2,2}^{\s}} := \sum_{j\in\ZZ} 2^{2\,j\,\s}\langle \Dd_j f\,,\, \Dd_j g\rangle_{L^2}\,,
\end{equation*}
is equivalent to the classical one over $\dot H^s$.

\noindent
A consequence of the Bernstein's inequality is the following embedding result.
\begin{prop}\label{p:embed}
The space $\BB^{s_1}_{p_1,r_1}$ is continuously embedded in the space $\BB^{s_2}_{p_2,r_2}$ for all indexes $p_1\,\leq\,p_2$, $r_1\,\leq\,r_2$ and
$$
s_2\,=\,s_1-d\left(\frac{1}{p_1}-\frac{1}{p_2}\right)\,. 
$$
\end{prop}
\noindent
We finally recall a classical commutator estimate (see e.g. Lemma 2.97 of \cite{BCD}).
\begin{lemma} \label{l:commut}
Let $\Phi\in\Cc^1(\RR^d)$ such that $(1+|\,\cdot\,|)\hat{\Phi}\,\in\,L^1$. There exists a constant $C$ such that,
for any function $h$ for which $\nabla h\in L^{p}(\RR^d)$, for any $f\in L^q(\R^d)$ and for all $\lambda>0$, one has
$$
\left\|\bigl[\Phi(\lambda^{-1}D),h\bigr]f\right\|_{L^r}\,\leq\,C\,\lambda^{-1}\,\left\|\nabla h\right\|_{L^p}\,\|f\|_{L^q}\,,
$$
where $r\in[1,+\infty]$ satisfies the relation $1/r\,=\,1/p\,+\,1/q$.
\end{lemma}
\vspace{0.5cm}
\noindent
{\bf Acknowledgment} 
The authors express their sincere appreciation to Professor C. Liu and Professor A. Zarnescu for constructive
suggestions and discussions. The work proceeded substantially both at the the Department of
Mathematics of the Penn State University and the Basque Center for Applied Mathematics. We thank deeply these institutions for their generous support and for providing a stimulating environment. 

\noindent
The first author has been partially supported by the NSF (grants DMS-1714401 and DMS-1412005). The second author has been supported by the Basque Government through the BERC 2018-2021 program and by Spanish Ministry of Economy and Competitiveness MINECO through BCAM Severo Ochoa excellence accreditation SEV-2013-0323 and through project MTM2017-82184-R funded by (AEI/FEDER, UE) and acronym "DESFLU. 


{

}
\end{document}